\documentclass[11pt,letterpaper,upref]{amsart}
\usepackage[letterpaper,margin=1.5in]{geometry}
\usepackage{amssymb}
\usepackage{amsxtra}
\usepackage[mathscr]{eucal}
\usepackage{graphicx}
\usepackage{mathtools}
\usepackage{verbatim}
\usepackage[pagebackref]{hyperref}

\usepackage{lineno}

\numberwithin{equation}{section}

\newtheorem{theorem}{Theorem}[section]
\newtheorem{proposition}[theorem]{Proposition}
\newtheorem{lemma}[theorem]{Lemma}
\newtheorem{corollary}[theorem]{Corollary}

\theoremstyle{definition}

\newtheorem{definition}[theorem]{Definition}
\newtheorem{example}[theorem]{Example}
\newtheorem{examples}[theorem]{Examples}
\newtheorem{problem}[theorem]{Problem}
\newtheorem{remark}[theorem]{Remark}

\newcommand{\af}{\alpha_f}
\newcommand{\aF}{\alpha_F}

\newcommand{\al}{\alpha}
\newcommand{\aM}{\al_M}
\newcommand{\ann}{\operatorname{ann}}
\newcommand{\aut}{\operatorname{aut}}
\newcommand{\bn}{\mathbf{n}}
\newcommand{\bz}{\mathfrak{b}_{\zeta}}
\newcommand{\CC}{\mathbb{C}}

\newcommand{\del}{\delta}
\newcommand{\Del}{\Delta}

\newcommand{\dnd}{\mathrm{det}_{\nd}}
\newcommand{\eps}{\varepsilon}
\newcommand{\F}{\mathrm{F}}

\newcommand{\fb}{\mathfrak{b}}

\newcommand{\Fix}{\operatorname{\mathsf{Fix}}}
\newcommand{\FixL}{\Fix_{\Lam}}
\newcommand{\G}{\Gamma}

\newcommand{\ga}{\mathfrak{a}}
\newcommand{\gb}{\mathfrak{b}}
\newcommand{\gp}{\mathfrak{p}}
\newcommand{\h}{\mathsf{h}}
\newcommand{\hc}{\mathit{\Delta}}

\newcommand{\lam}{\lambda}
\newcommand{\Lam}{\Lambda}
\newcommand{\Lamq}{\Lam_{rq,sq,q}}
\newcommand{\lodr}{\ell^1(\Del,\RR)}
\newcommand{\lidr}{\ell^\infty(\Del,\RR)}
\newcommand{\lidz}{\ell^\infty(\Del,\ZZ)}
\newcommand{\lodc}{\ell^1(\Del,\CC)}
\newcommand{\logc}{\ell^1(\G,\CC)}
\newcommand{\logr}{\ell^1(\G,\RR)}
\newcommand{\lozc}{\ell^1(\Z,\CC)}
\newcommand{\ligc}{\ell^\infty(\G,\CC)}
\newcommand{\ligr}{\ell^\infty(\G,\RR)}
\newcommand{\ligz}{\ell^\infty(\G,\ZZ)}
\newcommand{\lilc}{\ell^\infty(\Lam,\CC)}
\newcommand{\ltdc}{\ell^2(\Del,\CC)}
\newcommand{\mahler}{\mathsf{m}}
\newcommand{\Mahler}{\mathsf{M}}
\newcommand{\mz}{\mathfrak{m}_{\zeta}}
\newcommand{\nd}{\SN\Del}
\newcommand{\nG}{\SN\G}
\newcommand{\Newt}{\mathscr{N}}
\newcommand{\Num}{\mathsf{N}}
\newcommand{\om}{\omega}
\newcommand{\Om}{\Omega}
\newcommand{\phiz}{\phi_{\zeta}}
\newcommand{\psiz}{\psi_{\zeta}}
\newcommand{\Pq}{P_{rq,sq,q}}
\newcommand{\Qq}{Q_{rq,sq,q}}
\newcommand{\QQ}{\mathbb{Q}}
\newcommand{\rf}{\rho_f}
\newcommand{\RR}{\mathbb R}
\newcommand{\setdiff}{\bigtriangleup}
\newcommand{\SB}{\mathscr{B}}
\newcommand{\SH}{\mathscr{H}}

\newcommand{\SN}{\mathscr{N}}
\newcommand{\spr}{\operatorname{spr}}
\renewcommand{\SS}{\mathbb{S}}
\newcommand{\SU}{\mathscr{U}}
\newcommand{\supp}{\operatorname{supp}}
\newcommand{\TD}{\TT^{\Del}}
\newcommand{\tnd}{\tr_{\SN \Del}}
\newcommand{\tr}{\operatorname{tr}}
\newcommand{\ttil}{\tilde{t}}
\newcommand{\TT}{\mathbb{T}}
\newcommand{\ttt}[2]{t^{(#1)}_{#2}}
\newcommand{\tzt}{\TT^{\ZZ^2}}
\newcommand{\U}{\mathsf{U}}
\newcommand{\V}{\mathsf{V}}
\newcommand{\vxyz}{v_{\xi,\eta,\zeta}}
\newcommand{\vxz}{v_{\xi,\zeta}}
\newcommand{\wtri}{w^{\vartriangle}}
\newcommand{\ttri}{t^{\vartriangle}}
\newcommand{\xb}{\bar{x}}
\newcommand{\yb}{\bar{y}}
\newcommand{\fbar}{\overline{f}}
\newcommand{\Z}{\mathrm{Z}}
\newcommand{\ZZ}{\mathbb{Z}}
\newcommand{\zd}{\ZZ^{d}}
\newcommand{\ZD}{\ZZ\Del}
\newcommand{\ZG}{\ZZ\G}
\newcommand{\ZL}{\ZZ\Lam}

\renewcommand{\ge}{\geqslant}
\renewcommand{\le}{\leqslant}
\newcommand{\<}{\langle}
\renewcommand{\>}{\rangle}  
\renewcommand{\emptyset}{\varnothing}

\renewcommand\Re{\operatorname{Re}}
\renewcommand{\setminus}{\smallsetminus}

\begin{document}
\allowdisplaybreaks\frenchspacing


\title[Survey of Heisenberg Actions]
{A Survey of Algebraic Actions \\ of the Discrete Heisenberg Group}

\author{Douglas Lind}

\address{Douglas Lind: Department of Mathematics, University of
  Washington, Seattle, Washington 98195, USA}
  \email{lind@math.washington.edu}

\author{Klaus Schmidt}

\address{Klaus Schmidt: Mathematics Institute, University of Vienna, Oskar-Morgenstern-Platz 1, A-1090 Vienna, Austria \newline\indent \textup{and} \newline\indent Erwin Schr\"odinger Institute for Mathematical Physics, Boltzmanngasse~9, A-1090 Vienna, Austria} \email{klaus.schmidt@univie.ac.at}

\date{\today}



\dedicatory{Dedicated to Anatoly Vershik on the occasion of his 80th birthday}


\begin{abstract}
   The study of actions of countable groups by automorphisms of compact
   abelian groups has recently undergone intensive development,
   revealing deep connections with operator algebras and other areas.
   The discrete Heisenberg group is the simplest noncommutative example,
   where dynamical phenomena related to its noncommutativity already illustrate
   many of these connections. The explicit structure of this group means
   that these phenomena have concrete descriptions, which are not only
   instances of the general theory but are also testing grounds for
   further work. We survey here what is known about such actions of the
   discrete Heisenberg group, providing numerous examples and
   emphasizing many of the open problems that remain.
\end{abstract}

\maketitle

\section{Introduction}\label{sec:introduction}

Since Halmos's observation \cite{Halmos} in 1943 that automorphisms of
compact groups automatically preserve Haar measure, these maps have provided
a rich class of examples in dynamics. In case the group is abelian, its
dual group is a module over the Laurent polynomial ring $\ZZ[x^{\pm}]$.
Such modules have a well-developed structure theory, which enables a
comprehensive analysis of general automorphisms in terms of basic
building blocks that can be completely understood.

The roots of the study of several commuting algebraic maps can be traced
back to the seminal 1967 paper of Furstenberg \cite{Furstenberg}, where
he considered the joint dynamical properties of multiplication by
different integers on the additive torus. In 1978 Ledrappier
\cite{Ledrappier} gave a simple example of a mixing action of $\ZZ^2$ by
automorphisms of a compact abelian group that was not mixing of higher
orders. For an action by $d$ commuting automorphisms, the dual group is
a module over the Laurent polynomial ring
$\ZZ[u_1^{\pm},\dots,u_d^{\pm}]$, i.e., the integral group ring
$\ZZ\ZZ^d$ of $\ZZ^d$. The commutative algebra of such modules
provides effective machinery for analyzing such actions. This point of
view was initiated in 1989 by Kitchens and the second author
\cite{KitchensSchmidt}, and a fairly complete theory of the dynamical
properties of such actions is now available \cite{SchmidtBook}.

Let $\Del$ denote an arbitrary countable group, and let $\al$ denote an
action of $\Del$ by automorphisms of a compact abelian group, or an
\textit{algebraic $\Del$-action}. The initial steps in analyzing such
actions were  taken in \cite[Chap.\ 1]{SchmidtBook}, and give general
criteria for some basic dynamical properties such as ergodicity and
mixing.

In 2001 Einsiedler and Rindler \cite{EinsiedlerRindler} investigated the
particular case when $\Del=\G$, the discrete Heisenberg group, as a
first step towards algebraic actions of noncommutative groups. Here the
concrete nature of $\G$ suggests that there should be specific answers
to the natural dynamical questions, and they give several instances of
this together with instructive examples. However, the algebraic
complexity of the integral group ring $\ZG$ prevents the comprehensive
analysis available in the commutative case.

A dramatic new development occurred in 2006 with the work of Deninger on
entropy for principal $\Del$-actions. Let $f\in\ZD$, and let $\ZD f$
denote the principal left ideal generated by $f$. Then $\Del$ acts on
the quotient $\ZD/\ZD f$, and there is a dual $\Del$-action $\af$ on the
compact dual group, called a \textit{principal $\Del$-action}. Deninger
showed in \cite{D} that in many cases the entropy of $\af$ equals the
logarithm of the Fuglede-Kadison determinant the linear operator corresponding
to $f$ on the group von Neumann algebra of $\Del$. In case $\Del=\ZZ^d$,
this reduces to the calculation in \cite{LSW} of entropy in terms of the
logarithmic Mahler measure of $f$.  Subsequent work by Deninger, Li,
Schmidt, Thom, and others shows that this and related results hold in
great generality (see for example \cite{DS}, \cite{LiAutomorphisms}, and
\cite{LiThom}). In \cite{LiThom} the authors proved that three different concepts connected with $\Del$-actions,
namely entropy, Fuglede-Kadison determinants, and $L^2$-torsion,
coincide, revealing deep connections that are only partly understood.

These ideas have some interesting consequences. For example, by
computing the entropy of a particular Heisenberg action in two
different ways, we can show that
\begin{equation}
   \label{eqn:random-product}
   \lim_{n\to\infty}\frac1n \log \,\Biggl\| \prod_{k=0}^{n-1}
   \begin{bmatrix}
      0 & 1\\1 & e^{2 \pi i (ka+b)}
   \end{bmatrix}
   \Biggr\|=0
\end{equation}
for almost every pair $(a,b)$ of real numbers. Despite its simplicity,
this fact does not appear to follow from known results on random
matrix products.

Our purpose here is to survey what is known for the Heisenberg case
$\Del=\G$, and to point out many of the remaining open questions. As
$\G$ is the simplest noncommutative example (other than finite extensions of $\mathbb{Z}^d$, which are too close to the abelian case to be interesting), any results will indicate limitations of what a
general theory can accomplish. Also, the special structure of $\G$ should
enable explicit answers to many questions, and yield particular
examples of various dynamical phenomena. It is also quite instructive to
see how a very general machinery, used for algebraic actions of arbitrary
countable groups, can be made quite concrete for the case of ~$\G$. We
hope to inspire further work
by making this special case both accessible and attractive.

\section{Algebraic actions}\label{sec:algebraic-actions}

Let $\Del$ be a countable discrete group. The \textit{integral group ring} $\ZD$
of $\Del$ consists of all finite sums of the form
$g=\sum_{\del}g_{\del}\del$ with $g_{\del}\in\ZZ$, equipped with the
obvious ring operations inherited from multiplication in $\Del$. The
\textit{support} of $g$ is the subset $\supp(g)=\{\del\in\Del:g_{\del}\ne0\}$.

Suppose that $\Del$ acts by automorphisms of a compact abelian group
~$X$. Such actions are called \textit{algebraic $\Del$-actions}. Denote
the action of $\del\in\Del$ on $t\in X$ by $\del\cdot t$. Let $M$ be the
(discrete) dual group of $X$, with additive dual pairing denoted by $\<
t,m\>\in\RR$ for $t\in X$ and $m\in M$. Then $M$ becomes a module over
$\ZD$ by defining $\del\cdot m$ to be the unique element of ~$M$ so that
$\<t,\del\cdot m\>=\<\del^{-1}\cdot t, m\>$ for all $t \in X$, and extending this
to additive integral combinations in ~$\ZD$.

Conversely, if $M$ is a $\ZD$-module, its compact dual group
$X_M=\widehat{M}$ carries a $\Del$-action $\aM$ dual to the
$\Del$-action on $M$. Thus there is a 1-1 correspondence between
algebraic $\Del$-actions and $\ZD$-modules.

Let $\TT=\RR/\ZZ$ be the additive torus. Then the dual group of $M=\ZD$
can be identified with $\TT^{\Del}$ via the pairing $\<t,g\>=\sum_{\del}
t_{\del}g_{\del}$ where $t=(t_{\del})\in \TT^{\Del}$ and $g=\sum
g_{\del}\del\in\ZD$.

For $\theta\in\Del$, the action of $\theta$ on $t\in\TD$ is defined via
duality by $\<\theta\cdot t,g\>=\<t,\theta^{-1}\cdot g\>$ for all $g\in\ZD$. By
taking $g=\del\in\Del$, we obtain that
$(\theta\cdot t)_{\del}=t_{\theta^{-1}\del}$. It is sometimes convenient to
think of elements in $\TD$ as infinite formal sums $t =
\sum_{\del}t_{\del}\del$, and then
$\theta\cdot t=\sum_{\del}t_{\del}\theta\del=\sum_{\del}t_{\theta^{-1}\del}\del$. This
allows a well-defined multiplication of elements in $\TD$ by elements from ~$\ZD$, both
on the left and on the right.

We remark that the shift-action $(\theta\cdot t)_{\del^{}}=t_{\theta^{-1}\del}$ is
opposite to the traditional shift direction when $\Del$ is $\ZZ$ or $\ZZ^d$, but is
forced when $\Del$ is noncommutative. This has sometimes caused
confusion; for example the last displayed equation in
\cite[p.\ 118]{EinsiedlerRindler} is not correct.

Now fix $f\in\ZD$. Let $\ZD f$ be the principal left ideal generated by
$f$. The quotient module $\ZD/\ZD f$ has dual group
$X_f\subset\TD$. An element $t\in\TD$ is in $X_f$ iff $\<t,gf\>=0$ for
all $g\in\ZD$. This is equivalent to the condition that $\<t\cdot f^*,g\>=0$
for all $g\in\ZD$, where $f^*=\sum_{\del}f_{\del}\del^{-1}$. Hence
$t\in X_f$ exactly when $t\cdot f^*=0$, using the conventions above for right
multiplication of elements in $\TD$ by members of ~$\ZD$. In other words,
if we define $\rf(t)=t\cdot f^*$ to be right convolution by $f^*$, then
$X_f$ is the kernel of ~$\rf$.
In terms of coordinates, $t$ is in $X_f$ precisely when
$\sum_{\del}t_{\theta \del}f_{\del}=0$ for all $\theta\in\Del$.

Our focus here is on the discrete Heisenberg group $\G$, generated by
$x$, $y$, and $z$ subject to the relations $xz=zx$, $yz=zy$, and
$yx=xyz$. Alternatively, $\G$ is the subgroup of $SL(3,\ZZ)$ generated
by
\begin{displaymath}
   x \leftrightarrow \begin{bmatrix} 1&0&0\\0&1&1\\0&0&1\end{bmatrix},
   \quad
   y \leftrightarrow \begin{bmatrix} 1&1&0\\0&1&0\\0&0&1\end{bmatrix},
   \quad \text{and}\quad
   z \leftrightarrow \begin{bmatrix} 1&0&1\\0&1&0\\0&0&1\end{bmatrix}
   \quad.
\end{displaymath}

We will sometimes use the notation $R$ for the integral group ring ~$\ZG$
of ~$\G$ when emphasizing its ring-theoretic properties. The center of
$\G$ is $\Z=\{z^k:k\in\ZZ\}$. The center of $R$ is then the Laurent
polynomial ring $\ZZ\Z=\ZZ[z^{\pm}]$. Hence every
element of $R$ can be written as
\begin{displaymath}
   g=\sum_{k,l,m} g_{kl m}x^ky^l z^m = \sum_{k, l}g_{k l}(z)x^ky^l,
\end{displaymath}
where $g_{k l m}\in\ZZ$ and $g_{k l}(z)\in \ZZ\Z$.  For
$g=\sum_{k,l}g_{kl}(z)x^ky^l\in\ZG$, define the \textit{Newton polygon}
$\Newt(g)$ of $g$ to be the convex hull in $\RR^2$ of those points
$(k,l)$ for which $g_{kl}(z)\ne0$. In particular,
$\Newt(0)=\emptyset$. Because $\ZZ\Z$ is an integral domain, it is easy to
verify that $\Newt(gh)$ equals the Minkowski sum $\Newt(g)+\Newt(h)$ for
all $g,h\in\ZG$. This shows that $\ZG$ has no nontrivial zero-divisors.
However, a major difference between the commutative case and $\ZG$ is
that unique factorization into irreducibles fails for ~$\ZG$.

\begin{example}
   It is easy to verify that
   \begin{displaymath}
      (y-1)(y-z)(x+1)=(xyz^2-xz+y-z)(y-1).
   \end{displaymath}
   Each of the linear factors is clearly irreducible by a Newton polygon
   argument. We claim that $f(x,y,z)=xyz^2-xz+y-z$ cannot be factored in
   $\ZG$. Note that $\Newt(f) =[0,1]^2$. Suppose that $f=gh$. Adjusting
   by units and reordering factors if necessary, we may assume that $g$
   and $h$ have the form $g(x,y,z)=g_0(z)+g_1(z)x$ and
   $h(x,y,z)=h_0(z)+h_1(z)y$. Expanding $gh$, we find that
   \begin{displaymath}
      g_0(z)h_0(z)=-z,\ g_0(z)h_1(z)=1,\ g_1(z)h_0(z)=-z,\ \text{and}\
      g_1(z)h_1(z)=z^2.
   \end{displaymath}
   Hence $g_0(z)=h_1(z)^{-1}=g_1(z)z^{-2}$ and
   $h_0(z)=-zg_1(z)^{-1}$. Then we would have
   \begin{displaymath}
      -z=g_0(z)h_0(z)=\bigl(g_1(z)z^{-2}\bigr)\bigl(-zg_1(z)^{-1}\bigr)=-z^{-1}.
   \end{displaymath}
   This proves that $f$ has no nontrivial factorizations in $\ZG$.
\end{example}

Since $\G$ is nilpotent of rank 2, it is polycyclic, and so $R$ is both
left- and right-noetherian, i.e. $R$ satisfies the ascending chain
condition on both left ideals and on right ideals \cite{PassmanBook}.

Suppose now that $M$ is a finitely generated left $R$-module, say
generated by $m_1,\dots,m_l$. The map $R^l\to M$ defined by
$[g_1,\dots,g_l ]\mapsto g_1m_1+\dots g_l m_l$ is
surjective. Its kernel $K$ is a left $R$-submodule of the noetherian
module $R^l$, hence also finitely generated, say by
$[f_{11},\dots,f_{1l}], \dots,[f_{k1},\dots,f_{kl}]$. Let $F
=[f_{ij}]\in R^{k\times l}$ be the rectangular matrix whose rows are
the generators of $K$. Then $K=R^kF$, and $M\cong R^l/R^kF$. We will
denote the corresponding algebraic $\G$-action for this presentation of
$M$ by $\aF$. Notice that when $k=l=1$ we are reduced to the case
$F=[f]$, corresponding to the quotient module $R/Rf$ and the principal
$\G$-action $\af$.

\section{Ergodicity}\label{sec:ergodicity}

Let $X$ be a compact abelian group and let $\mu$ denote Haar measure on $X$,
normalized so that $\mu(X)=1$. If $\phi$ is a continuous automorphism of
$X$, then the measure $\nu$ defined by $\nu(E)=\mu\bigl(\phi(E)\bigr)$
is also a normalized translation-invariant measure. Hence $\nu=\mu$, and
$\mu$ is $\phi$-invariant.

This shows that if $\al$ is an algebraic action of a countable group
$\Del$ on $X$, then $\al$ is $\mu$-measure-preserving. A measurable set
$E\subset X$ is \textit{$\al$-invariant} if $\al^{\del}(E)$ agrees with
$E$ off a null set for every $\del\in\Del$. The action $\al$ is
\textit{ergodic} if the only $\al$-invariant sets have measure 0 or
1. The following, which is a special case of a result due to Kaplansky
\cite{Kaplansky},  gives an algebraic characterization of ergodicity.

\begin{lemma}[{\cite[Lemma 1.2]{SchmidtBook}}]
   Let $\Del$ be a countable discrete group, and $\al$ be an algebraic
   $\Del$-action on a compact abelian group $X$ whose dual group is
   $M$. Then $\al$ is ergodic if and only if the $\Del$-orbit of every
   nonzero element of $M$ is infinite.
\end{lemma}

Roughly speaking, this result follows from the observation that the
existence of a bounded measurable $\al$-invariant function on $X$ is equivalent
to the existence of a nonzero finite $\Del$-orbit in ~$M$.

For actions of the Heisenberg group $\G$, this raises the question of
characterizing those $F\in R^{k\times l}$ for which $\aF$ is
ergodic. The first result in this direction is due to Ben Hayes.

\begin{theorem}[{\cite[Thm.\ 2.3.6]{Hayes}}]\label{thm:hayes}
   For every $f\in\ZG$ the principal algebraic $\G$-action $\af$ is ergodic.
\end{theorem}

\begin{proof}
   We give a brief sketch of the proof. First recall that $\ZZ\Z$ is a
   unique factorization domain. Define the \textit{content} $c(g)$ of
   $g=\sum_{i,j}g_{ij}(z)x^iy^j\in\ZG\setminus \{0\}$ to be the greatest
   common divisor in $\ZZ\Z$ of the nonzero coefficient polynomials
   $g_{ij}(z)$, and put $c(0)=0$.  A simple variant of the proof of
   Gauss's lemma shows that $c(gh)=c(g)c(h)$ for all $g,h\in\ZG$.

   Now fix $f\in\ZG = R$. The case $f=0$ is trivial, so assume that
   $f\ne0$. Suppose that $h+Rf$ has finite $\G$-orbit in $R/Rf$. Then
   there are $m,n\ge1$ such that $(x^m-1)h=g_1 f$ and $(z^n-1)h=g_2f$
   for some $g_1,g_2\in R$. Then
   \begin{displaymath}
      c\bigl((x^m-1)h\bigr)=c(x^m-1)c(h)=c(h)=c(g_1)c(f),
   \end{displaymath}
   so that $c(f)$ divides $c(h)$ in $\ZZ\Z$. Also,
   \begin{displaymath}
      c\bigl((z^n-1)h\bigr)=(z^n-1)c(h)=c(g_2)c(f),
   \end{displaymath}
   so that $(z^n-1)[c(h)/c(f)]=c(g_2)$, and hence $z^n-1$ divides
   $c(g_2)$. Thus $g_2/(z^n-1)\in R$, and so $h=[g_2/(z^n-1)]f\in Rf$,
   showing that $h+Rf=0$ in $R/Rf$.
\end{proof}

Hayes called a group $\Del$ \textit{principally ergodic} if every
principal algebraic $\Del$-action is ergodic. He extended Theorem
\ref{thm:hayes} to show that the following classes of groups are
principally ergodic: torsion-free nilpotent groups that are not
virtually cyclic (i.e., do not contain a cyclic subgroup of finite
index), free groups on more than one generator, and groups that are not
finitely generated. Clearly $\ZZ$ is not principally ergodic, since for
example the action of $x$ on the module $\ZZ[x^{\pm}]/\<x^k -1\>$
dualizes to a $k\times k$ permutation matrix on ~$\TT^k$, which is not
ergodic.

Recently Li, Peterson, and the second author used a very
different approach to proving principal ergodicity, based on
cohomology \cite{LiPetersonSchmidt}. These methods greatly increased the collection
of countable discrete groups known to be principally ergodic, including
all such groups that contain a finitely generated amenable subgroup that
is not virtually cyclic.

We will describe now how their ideas work in the case of $\G$. We begin
by describing two important properties of finite-index subgroups of
~$\G$, namely that they are amenable and have only one end.

For an arbitrary discrete group $\Del$ let
$\lodr=\{w\in\RR^\Del:\|w\|_1:=\sum_\del |w_{\del}|<\infty\}$ and
$\lidr=\{w\in\RR^\Del:\|w\|_\infty:=\sup_{\del}|w_{\del}|<\infty\}$, so that $\lidr$
is the dual space to $\lodr$.

Fix $K, L\ge1$, and put $M=KL$. Let $\Lam=\Lam_{KLM}=\< x^K,y^L,z^M\>$,
the finite-index subgroup of $\G$ generated by $x^K, y^L$, and $z^M$.

\begin{lemma}\label{lem:amenability}
   Let $\Lam=\Lam_{KLM}$ and suppose that $\{T_\lam :\lam\in\Lam\}$ is
   an action of $\Lam$ by continuous affine operators on $\ligr$. If $C$
   is a weak*-compact, convex subset of $\ligr$ such that
   $T_\lam(C)\subset C$ for every $\lam\in\Lam$, then there is a common
   fixed point $b\in C$ for all the $T_\lam$.
\end{lemma}

\begin{proof}
   Put
   \begin{displaymath}
      F_n=\{x^{pK}, y^{qL}, z^{rM}: 0\le p<n, 0\le q<n, 0\le r < n^2\}.
   \end{displaymath}
   The condition of the powers of $z$ is imposed so that any distortion
   caused by left multiplication of $F_n$ by a given element
   $\lam\in\Lam$ is eventually small. More precisely, for every
   $\lam\in\Lam$ we have that
   \begin{displaymath}
      \frac{|\lam F_n\setdiff F_n|}{|F_n|} \to 0 \text{\quad as $n\to\infty$},
   \end{displaymath}
   where $\setdiff$ denotes symmetric difference and $|\cdot |$ denotes
   cardinality.

   Now fix $b_0\in C$, and let $b_n=\frac1{|F_n|} \sum_{\lam\in
   F_n}T_\lam(b_0)$. Then $b_n\in C$ since $C$ is convex. Since $C$
   is weak*-compact, there is a subsequence $b_{n_j}$ converging weak*
   to some $b\in C$. Note that $\sup _{c\in C}\|c\|_\infty<\infty$ by
   compactness of $C$. Then since each $T_{\theta}$ is continuous, we have
   that $T_\theta(b_{n_j})\to T_\theta(b)$ for every $\theta
   \in\Lambda$. Furthermore,
   \begin{displaymath}
      \| T_\theta(b_{n_j})-b_{n_j}\|_\infty \le
      \frac{|\theta F_{n_j}\setdiff F_{n_j}|}{|F_n|} \cdot
      \sup_{c\in C}\|c\|_\infty \to 0 \text{\quad as \quad $j\to\infty$}.
   \end{displaymath}
   It follows that $T_\theta(b)=b$ for all $\theta\in\Lam$.
\end{proof}

The essential point in the previous proof is that $\Lam$ is amenable, and
that $\{F_n\}$ forms a F{\o}lner sequence.

\smallskip We call a set $A\subset\Lam=\Lam_{KLM}$ \textit{almost invariant} if $|A\lam\setdiff A|$ is finite for every $\lam\in\Lam$. Clearly, if $A$ is almost invariant, then so is $A\lambda '$ for every $\lambda '\in \Lambda $.

\begin{lemma}\label{lem:ends}
   Let $A\subset\Lam=\Lam_{KLM}$ be an infinite almost invariant subset. Then
   $\Lam\smallsetminus A$ is finite.
\end{lemma}

\begin{proof}
   Let $S=\{x^K,x^{-K},y^L,y^{-L},z^M,z^{-M}\}$ be a set of generators
   for $\Lam$. The \textit{Cayley graph} $\mathscr{G}$ of $\Lam$ with
   respect to $S$ has as vertices the elements of $\Lam$, and for every
   vertex $\lam$ and $s\in S$ there is a directed edge from $\lam$ to $\lam
   s$. Let $E$ be the union of $As\setdiff A$ over $s\in S$, so $E$ is
   finite. We can therefore enclose $E$ in a box of the form
   $B=\{x^{iK}y^{kL}z^{lM}:|i|,|j|,|k|\le n\}$. Since $A$ is infinite,
   choose $a\in A\setminus B$. Then for every $b\in \Lambda \setminus B$ there
   is a finite directed path in $\mathscr{G}$ from $a$ to $b$ that
   avoids $B$, say with vertices  $a$, $as_1$,
   $as_1s_2$, $\ldots$, $as_1s_2\dots s_r=b$, and by definition of $E$
   each of these is in $A$. Hence $\Lam\setminus A\subset B$, and so is
   finite.
\end{proof}

\begin{proof}[Second proof of Theorem \ref{thm:hayes}, adapted from
d\cite{LiPetersonSchmidt}] Suppose that
$h\in\ZG$ with $h+\ZG f$ having finite $\G$-orbit in $\ZG/\ZG f$.  We
will prove that $h\in\ZG$.

There are $K,L\ge1$ such that $(x^K-1)h\in\ZG f$ and $(y^L -1)h\in\ZG
f$. Then $x^{-K}y^Lx^K y^{-L}=z^{KL}$ also stabilizes $h+\ZG f$. Let
$M=KL$. Then there are $g_1,g_2,g_3\in\ZG$ such that $(x^K-1)h=g_1f$,
$(y^L-1)h=g_2f$, and $(z^M-1)h=g_3f$.

Consider the finite-index subgroup $\Lam=\Lam_{KLM}$ of $\G$ as
above. For every $\lam\in\Lam$ there is a $c(\lam)\in\ZG$ such that
$(\lam-1)h=c(\lam)f$, and this is unique since $\ZG$ has no
zero-divisors. Then $c\colon\Lam\to\ZG$ is a \textit{cocycle}, that is,
it obeys $c(\lam\lam')=c(\lam)+\lam c(\lam')$ for all
$\lam,\lam'\in\Lam$.

Consider $\ZG$ as a subset of $\ligr$. We claim that $c$ is a uniformly
bounded cocycle, i.e., that
$\sup_{\lam\in\Lam}\|c(\lam)\|_\infty<\infty$. The reason for this is
that we can calculate the value of $c(\lam)$ for arbitrary $\lam$ using
left shifts of the generators $g_i$ that are sufficiently spread out to
prevent large accumulations of coefficients. For example, if
$p,q,r\ge1$, then applying the cocycle equation first for powers of
$x^K$, then powers of $y^L$, and then powers of $z^M$, we get that
\begin{multline}\label{eqn:cocycle}
c(x^{pK}y^{qL}z^{rM})=g_1+x^Kg_1+\dots+x^{(p-1)K}g_1
+x^{pK}g_2+x^{pK}y^Lg_2+\cdots \\+ x^{pK}y^{(q-1)L}g_2+
x^{pK}y^{qL}g_3+x^{pK}y^{qL}z^Mg_3+\dots +x^{pK}y^{qL}z^{(r-1)M}g_3 .
\end{multline} Since the supports of the $g_i$ are finite, there is a
uniform bound $P<\infty$ so that for every $\gamma\in\G$ and every
$\lam\in\Lam$, there are at most $P$ summands in the expression
\eqref{eqn:cocycle} for $c(\lam)$ whose support contains $\gamma$. Hence
$\|c(\lam)\|_\infty\le P \sup_{1\le i\le 3}\|g_i\|_\infty =B<\infty$,
establishing our claim.

Now let $C$ be the closed, convex hull of $\{c(\lam):\lam\in\Lam\}$ in
$\ligr$, which is weak*-compact since the $c(\lam)$ are uniformly
bounded. Consider the continuous affine maps $T_\lam\colon
\ligr\to\ligr$ defined by $T_\lam(v)=\lam\cdot v+c(\lam)$. Then
$T_\lam\circ T_{\lam'}=T_{\lam\lam'}$ by the cocycle property of $c$,
and $T_\lam(c(\lam'))=c(\lam\lam')$, so that $T_\lam(C)\subset C$ for
all $\lam\in\Lam$. By Lemma \ref{lem:amenability} there is a common
fixed point $v=(v_\gamma)\in C$ for the $T_\lam$, so that $v-\lam\cdot
v=c(\lam)\in\ZG$ for all $\lam\in\Lam$.  Write each
$v_\gamma=w_\gamma+u_\gamma$ with $w_\gamma\in\ZZ$ and $u_\gamma\in[0,1)$. Then
\begin{displaymath}
u_\gamma-u_{\lam^{-1}\gamma}=v_{\gamma}-v_{\lam^{-1}\gamma}+
w_{\gamma}-w_{\lam^{-1}\gamma}\in (-1,1) \cap \ZZ=\{0\},
\end{displaymath} so that $w$ also satisfies that $w-\lam\cdot
w=c(\lam)$, where now $w\in\ell^\infty(\G,\ZZ)$ has integer
coordinates. Replacing $w$ with $-w$, we have found a $w\in
\ell^\infty(\G,\ZZ)$ with $\|w\|_\infty\le B$ and $\lam\cdot
w-w=c(\lam)$ for all $\lam\in\Lam$.

Next, we use Lemma \ref{lem:ends} to show that we can replace $w$ with
an element of $\ligz$ having finite support, and so is an element of
$\ZG$. Fix $\gamma\in\G$. For $-B\le k\le B$ consider the ``level set''
for the restriction of $w$ to the right coset $\Lam\gamma$,
$A_{\gamma,k}=\{\lam\in\Lam:w_{\lam^{-1}\gamma}=k\}$. We claim that for
each $\gamma$ there is exactly one $k$ for which $A_{\gamma,k}$ is
infinite. For suppose that $A_{\gamma,k}$ is infinite. Let
$\theta\in\Lam$. Since $\theta\cdot w-w=c(\theta)$ has finite support,
$w_{\theta^{-1}\lam^{-1}\gamma} = w_{\lam^{-1}\gamma}$ for all but
finitely many $\lam\in\Lam$. Hence $|A_{\gamma,k}\theta\setdiff
A_{\gamma,k}|<\infty$ for every $\theta\in\Lam$. By Lemma
\ref{lem:ends}, we see that $\Lam\setminus A_{\gamma,k}$ is
finite. Hence for every $\gamma\in\G$ we can adjust the value of $w$ on
the coset $\Lam\gamma$ so that the restriction of $w$ to $\Lam\gamma$
has finite support. Doing this for each of the finitely many right
cosets $\Lam\gamma$ results in a $w$ with finite support on $\G$, so
that $w\in\ZG$.

Thus $c(\lam)=(\lam-1)w$ for every $\lam\in\Lam$. Since
$(\lam-1)h=c(\lam)f=(\lam-1)wf$, we obtain $h=wf\in\ZG f$, as required.
\end{proof}

\medskip

Theorem \ref{thm:hayes} answers the $1\times1$ case of the following
natural question.

\begin{problem}\label{prob:ergodicity}
   Describe or characterize those $F\in R^{k\times l}$ for which
   $\aF$ is ergodic, or, equivalently, those noetherian $R$-modules $M$
   for which $\aM$ is ergodic. Is there a finite algorithm that will
   decide whether or not a given $\aF$ is ergodic? Are there easily
   checked sufficient conditions on $F$ for ergodicity of ~$\aF$?
\end{problem}

Einsiedler and Rindler provided one answer to Problem
\ref{prob:ergodicity}, which involves the notion of prime ideals in
$R$. A two-sided ideal $\gp$ in $R$ is \textit{prime} if whenever $\ga$ and
$\gb$ are two-sided ideals in $R$ with $\ga \gb\subseteq \gp$, then
either $\ga\subseteq\gp$ or $\gb\subseteq\gp$. If $N$ is an
$R$-submodule of an $R$-module $M$, then the annihilator $\ann_R(N)$ of $N$ is
defined as $\{f\in R: fn=0 \text{ for all $n\in N$}\}$, which is a
two-sided ideal in $R$. A prime ideal $\gp$ is \textit{associated to}
$M$ if there is a submodule $N\subseteq M$ such that for every nonzero
submodule $N'\subseteq N$ we have that $\ann_R(N')=\gp$. Every
noetherian $R$-module has associated prime ideals, and there are only
finitely many of them.

Call a prime ideal $\gp$ of $R$ \textit{ergodic} if the subgroup
$\{\gamma\in\G:\gamma-1\in\gp\}$ of $\G$ has infinite index in $\G$. For
instance, the ideal $\gp$ in Example \ref{exam:nonergodic} is prime
(being the kernel of a ring homomorphism onto an commutative integral domain), but
is not ergodic. It is easy to verify that a prime ideal is ergodic if
and only if $\al_{R/\gp}$ is ergodic, and that the only ideal associated
with $R/\gp$ is $\gp$.

\begin{theorem}[{\cite[Thm.\ 3.3]{EinsiedlerRindler}}]
   Let $M$ be a noetherian $R$-module. Then $\aM$ is ergodic if and only
   if every prime ideal associated with $M$ is ergodic.
\end{theorem}

\begin{example}\label{exam:nonergodic}
   Let $\gp$ be the left ideal in $R$ generated by $x-1$ and $y-1$.
   Then
   \begin{displaymath}
      z-1=(y-z)(x-1)+(1-xz)(y-1)\in\gp,
   \end{displaymath}
   and so the map $\phi\colon R\to\ZZ$ defined by $\phi(f)=f(1,1,1)$ is
   a well-defined surjective ring homomorphism with kernel $\gp$. Thus
   $\gp$ is a prime ideal with
   $R/\gp$  isomorphic to $\ZZ$. The dual $\G$-action is simply
   the identity map on $\TT$, which is nonergodic. Hence $\gp$ is a left
   ideal generated by two elements with $\al_{R/\gp}$ nonergodic,
   showing that Theorem ~\ref{thm:hayes} does not extend to nonprincipal
   actions.

   We remark that if we consider $\ZZ^3$ instead of $\G$, then the
   characterization of ergodic prime ideals in
   \cite[Thm. 6.5]{SchmidtBook} shows that their complex variety is
   finite, and in particular by elementary dimension theory they must
   have at least three generators.
\end{example}

A relatively explicit description of all prime ideals in $R$ is given in
\cite{MacKenzie}.

\begin{problem}
   Characterize the ergodic prime ideals in $R$.
\end{problem}

An answer to this problem would reduce Problem \ref{prob:ergodicity}
to computing the prime ideals associated to a given noetherian
$R$-module. However, this appears to be difficult, even for modules of
the form $R/Rf$, although it follows from Theorem
\ref{thm:hayes} that all prime ideals associated to $R/Rf$ must be
ergodic.

\section{Mixing}\label{sec:mixing}
Let $\Del$ be a countable discrete group, and let $M$ be a left
$\ZD$-module. Denote by $\mu$ Haar measure on $X_M$. The associated
algebraic $\Del$-action $\aM$ is called \textit{mixing} if, for every pair of
measurable sets $E,F\subset X_M$, we have that
$\lim_{\del\to\infty}\mu(\aM^\del(E)\cap F)=\mu(E)\mu(F)$, where
$\del\to\infty$ refers to the one-point compactification of $\Del$.
For $m\in M$, the \textit{stabilizer} of $m$ is the subgroup
$\{\del\in\Del: \del\cdot m=m\}$.

\begin{proposition}[{\cite[Thm.\ 1.6]{SchmidtBook}}]\label{prop:general-mixing}
   The algebraic $\Del$-action $\aM$ is mixing if and only if the
   stabilizer of every nonzero $m\in M$ is finite. In the case $\Del=\G$,
   this is equivalent to requiring that for every nonzero $m\in M$ the
   map $\gamma\mapsto \gamma\cdot m$ is injective on $\G$.
\end{proposition}

Using this together with some of the ideas from the previous section, we
can give a simple criterion for $g(z)\in \ZZ\Z=\ZZ[z^{\pm}]$ so that $\al_g$ is
mixing.

\begin{proposition}
   Let $g=g(z)\in \ZZ\Z$. Then the principal $\G$-action $\al_g$ is mixing
   if and only if $g(z)$ has no roots that are roots of unity.
\end{proposition}

\begin{proof}
   Suppose first that $g(z)$ has a root that is a root of unity, so that
   $g(z)$ has a factor $g_0(z)\in \ZZ\Z$ dividing $z^n-1$ for some
   $n\ge1$. Then $h=g/g_0\notin \ZG g$, but
   $(z^n-1)h=[(z^n-1)/g_0]g\in\ZG g$, so that $\al_g$ is not mixing by
   Prop.\ \ref{prop:general-mixing}.

   Conversely, suppose that $g$ has no root that is a root of
   unity. Recall that the content $c(h)$ of
   $h=\sum_{i,j}h_{ij}(z)x^iy^j$ is the greatest common divisor in $\ZZ\Z$
   of the $h_{ij}(z)$. Then $h\in\ZG g$ if and only if $g\mid c(h)$.

   Suppose that $(x^py^qz^r-1)h\in \ZG g$ with $(p,q)\neq (0,0)$. Then
   $g$ divides $c((x^py^qz^r-1)h)=c(h)$, showing that $h\in\ZG
   g$. Similarly, if $(z^r-1)h\in\ZG g$, then $g \mid (z^r-1)c(h)$, and
   by assumption $g$ is relatively prime to $z^r-1$. Hence again $g\mid
   c(h)$, and so $h\in\ZG g$. Then Prop.\ \ref{prop:general-mixing}
   shows that $\al_g$ is mixing.
\end{proof}

Using more elaborate algebra, Hayes found several sufficient conditions
on $f\in\ZG$ for $\af$ to be mixing. To make the cyclotomic nature of
these conditions clear, recall that the $n$-th \textit{cyclotomic polynomial}
$\Phi_n(u)$ is given by $\Phi_n(u)=\prod (u-\omega)$, where the product
is over all primitive $n$-th roots of unity. Each $\Phi_n(u)$ is
irreducible in $\QQ[u]$, and $u^n-1=\prod_{d\mid n}\Phi_d(u)$ is the
irreducible factorization of $u^n-1$ in $\QQ[u]$.

Let $u_1,\dots,u_r$ be $r$ commuting variables. Then a
\textit{generalized cyclotomic polynomial in} $\ZZ[u_1^{\pm}, \dots,
u_r^{\pm}]$ is one of the form $\Phi_k(u_1^{n_1}\dots u_r^{n_r})$ for
some $k\ge1$ and choice of integers $n_1,\dots,n_r$, not all $0$.

There is a well-defined ring homomorphism
$\pi\colon\ZG\to\ZZ[\xb^{\pm},\yb^{\pm}]$, where $\xb$ and $\yb$ are
commuting variables, given by $x\mapsto\xb$, $y\mapsto\yb$, and
$z\mapsto1$. For $f=\sum_{ij}f_{ij}(z)x^iy^j\in\ZG$, its image under
$\pi$ is $\fbar(\xb,\yb)=\sum_{ij}f_{ij}(1)\xb^i\yb^j$.

\begin{proposition}[\cite{Hayes2}]
   Each of the following conditions on $f\in\ZG$ is sufficient for
   $\al_f$ to be mixing:
   \begin{enumerate}
     \item $f\in\ZZ[x^{\pm},z^{\pm}]$ and $f$ is not divisible by a
      generalized cyclotomic polynomial in $x$ and $z$.
     \item $f\in \ZZ[y^{\pm},z^{\pm}]$ and $f$ is not divisible by a
      generalized cyclotomic polynomial in $y$ and $z$.
     \item $f=\sum_{ij}f_{ij}(z)x^iy^j$ with the content $c(f)$ not
      divisible by a cyclotomic polynomial in $z$, and $\fbar(\xb,\yb)$
      not divisible by a cyclotomic polynomial in $\xb$ and ~$\yb$.
   \end{enumerate}
\end{proposition}

\begin{examples}
   (a) If $f=1+x+y$, then $\al_f$ is mixing by part (3).

   (b) If $f=x+z-2$, then $\al_f$ is mixing by part (1), yet
   $\fbar(\xb,\yb)=\xb-1$ is cyclotomic, showing that part (3) is not
   always necessary.

   (c) A generalized cyclotomic polynomial in $u_1$ and $u_2$ has a root
   both of whose coordinates have absolute value 1. It follows that, for
   example, $4u_1+3u_2+8u_1u_2$ cannot be divisible by any generalized
   cyclotomic polynomial. Then part (3) above implies that for every
   choice of nonzero integers $p, q, r$, the polynomial
   $f=(z^p+3)x+(z^q+2)y+(z^r+7)xy$ yields a mixing $\al_f$.

   More generally, if $\sum_{ij}b_{ij}u_1^iu_2^j$ is not divisible by a
   generalized cyclotomic polynomial, and $p_{ij}(z)\in \ZZ\Z$ all satisfy
   $p_{ij}(1)=b_{ij}$ and have no common root that is a root of unity,
   then $f=\sum_{ij}p_{ij}(z)x^iy^j$ results in a mixing $\af$.
\end{examples}

\begin{problem}
   Does there exist a finite algorithm that decides, given $f\in\ZG$,
   whether or not $\af$ is mixing? More generally, is there such an
   algorithm that decides mixing for $\G$-actions of the form $\al_F$,
   where $F\in\ZG^{k\times l}$?
\end{problem}

There is another, simply stated, sufficient condition for $\af$ to be
mixing. Recall that $\logr$ is a Banach algebra under convolution, with
identity element 1.

\begin{proposition}
   If $f\in\ZG$ is invertible in $\logr$, then $\af$ is mixing.
\end{proposition}

\begin{proof}
   If $\af$ were not mixing, there would be an $h\notin\ZG f$ and an
   infinite subgroup
   $\Lam\subset\G$ such that $\lam\cdot h-h\in\ZG f$ for all
   $\lam\in\Lam$, say $\lam\cdot h-h=c(\lam)f$. Let $w\in\logr$ be the
   inverse of $f$. Then $\lam\cdot h\cdot w-h\cdot
   w=c(\lam)\in\ZG$. Letting $\lam\to\infty$ in $\Lam$, we see that
   $g=h\cdot w \in\ZG$. Hence $h=h\cdot w\cdot f=gf\in\ZG f$, showing
   that $\af$ is mixing.
\end{proof}

We will see in Theorem \ref{thm:expansive} that invertibility of $f$
in $\logr$ corresponds to an important dynamical property of ~$\al_f$.

\section{Expansiveness}\label{sec:expansiveness}

Let $\Del$ be a countable discrete group and $\al$ be an algebraic
$\Del$-action on a compact abelian group $X$. Then $\al$ is called
\textit{expansive} if there is a neighborhood $U$ of $0_X$ in $X$ such
that $\bigcap_{\del\in\Del}\al^{\del}(U)=\{0_X\}$. All groups $X$ we
consider are metrizable, so let $d$ be a metric on $X$ compatible with
its topology. By averaging $d$ over $X$, we may assume that $d$ is
translation-invariant. Then $\al$ is expansive provided there is a
$\kappa>0$ such that if $d \bigl(\al^{\del}(t),\al^{\del}(u)\bigr)\le
\kappa$ for all $\del\in\Del$, then $t=u$.

Expansiveness is an important and useful property, with many
implications. It is therefore crucial to know when algebraic actions are
expansive. For principal actions there is a simple criterion.

\begin{theorem}[\protect{\cite[Theorem 3.2]{DS}}]\label{thm:expansive}
   Let $\Del$ be a countable discrete group and let $f\in\ZD$. Then the
   following are equivalent:
   \begin{enumerate}
     \item The principal algebraic action $\af$ on $X_f$ is expansive,
     \item The principal algebraic action $\al_{f^*}$ on $X_{f^*}$ is
      expansive,
     \item f is invertible in $\lodr$,
     \item $\rho_f$ is injective on $\lidr$.
   \end{enumerate}
\end{theorem}

Before sketching the proof, we isolate a crucial property of $\lodr$
called \textit{direct finiteness}: if $v,w\in\lodr$ with $v\cdot w=1$,
then $w\cdot v=1$, where $v\cdot w$ denotes the usual convolution
product in $\lodr$.  This was originally proved by Kaplansky \cite[p.\
122]{Kaplansky2} using von Neumann algebra techniques. Later Montgomery
\cite{Montgomery} gave a short proof using $C^* $-algebra methods. A
more self-contained argument using only elementary ideas was given by
Passman \cite{Passman2}. All these arguments use a key feature of
$\lodr$, that it has a faithful trace function $\tr\colon\lodr\to\RR$
given by $\tr(w)=w_{1_\Delta}$. This function has the properties that it
is linear, $\tr(1)=1$, $\tr(v\cdot w)=\tr(w\cdot v)$, and $\tr(v\cdot
v)\ge0$, with equality iff $v=0$. The key argument is that if
$e\cdot e=e$, then $0\le\tr(e)\le1$, and $\tr(e)=1$ implies that
$e=1$. If $v\cdot w=1$, then $e=w\cdot v$ satisfies that $e\cdot
e=(w\cdot v)\cdot (w\cdot v)=w\cdot (v\cdot w)\cdot v=w\cdot v=e$, and
$\tr(w\cdot v)=\tr(v\cdot w)=1$, hence $w\cdot v=1$.

\begin{proof}[Proof of Theorem \ref{thm:expansive}]
   Let $d_{\TT}$ be the usual metric on $\TT=\RR/\ZZ$ defined by
   $d_{\TT}(t+\ZZ,u+\ZZ)=\min_{n\in\ZZ}|t-u+n|$. It is straightforward
   to check that we may use the pseudometric $d_1$ on $X_f$ defined by
   $d_1(t,u)=d_{\TT}(t_{1_\Del},u_{1_\Del})$ to determine expansiveness
   (see \cite[Prop. 2.3]{D} for details).

   First suppose that
   there is a $w\in\lidr$ such that $\rf(w)=w\cdot f^*=0$. Let
   $\beta\colon\RR\to\TT$ be the usual projection map, and extend $\beta$
   to $\lidr$ coordinatewise. For every $\epsilon>0$ we have that
   $\rf(\epsilon w)=0$, so that $\beta(\epsilon w)\in X_f$. Since
   $\|\epsilon w\|_{\infty}=\epsilon\|w\|_{\infty}$ can be made
   arbitrarily small, it follows that $\af$ is not
   expansive. Conversely, if $\af$ is not expansive, there is a point
   $t\in X_f$ with $d_{\TT}(t_\del,0)<(3\|f^*\|_1)^{-1}$ for all
   $\del\in\Del$. Pick $\tilde{t}_\del\in[0,1)$ with
   $\beta(\tilde{t}_\del)=t_\del$ for all $\del\in\Del$, and let
   $\tilde{t}=(\tilde{t}_\del)\in\lidr$. Then $\rf(\tilde{t})\in\lidz$),
   and $\|\rf(\tilde{t})\|_\infty<1$, hence $\rf(\tilde{t})=0$. This
   shows that (1) $\Leftrightarrow$ (4).

   Now suppose that $\rf$ is injective on $\lidr$. Then
   $\rho_{f^*}\bigl(\lodr\bigr)$ is dense in $\lodr$ by the Hahn-Banach
   Theorem. Since the set of invertible elements in $\lodr$ is open,
   there is a $w\in\lodr$ with $w\cdot f^*=1$. By direct finiteness, $f^*$ is
   invertible in $\lodr$ with inverse $w$. Hence
   $f^{-1}=(w^*)^{-1}$. This shows that (4) $\Rightarrow$ (3), and
   the implication (3) $\Rightarrow$ (4) is obvious. Since $f$ is
   invertible if and only if $f^*$ is invertible, (2)
   $\Leftrightarrow$ (1).
\end{proof}

Let us call $f\in\ZD$ \textit{expansive} if $\af$ is expansive. For the
case $\Del=\ZZ$, Wiener's theorem on invertibility in the convolution
algebra $\ell^1(\ZZ,\CC)$ shows that $f(u)\in\ZZ[u^{\pm}]$ is expansive
if and only if $f$ does not vanish on $\SS$.  The usual proof of
Wiener's Theorem via Banach algebras is nonconstructive since it uses
Zorn's lemma to create maximal ideals. Paul Cohen \cite{Cohen} has given
a constructive treatment of this and similar results.  We are grateful
to David Boyd for showing us a simple algorithm for deciding
expansiveness in this case.

\begin{proposition}
   There is a finite algorithm, using only operations in $\QQ[u]$, that
   decides, given $f(u)\in\ZZ[u^{\pm}]$, whether or not $f$ is expansive.
\end{proposition}

\begin{proof}
   We may assume that $f(u)\in\ZZ[u]$ with $f(0)\ne0$, say of degree
   $d$. We can compute the greatest common divisor $g(u)$ of $f(u)$ and
   $u^df(1/u)$ using only finitely many operations in $\QQ[u]$. Observe
   that any root of $f(u)$ on $\SS$ must also be a root of $g(u)$. Let
   the degree of $g$ be $e$. Then $g(u)=u^eg(1/u)$, so that the
   coefficients of $g(u)$ are symmetric. If $e$ is odd, then $-1$ is a
   root of $g$ since all other possible roots come in distinct pairs. If $e$ is
   even, it is simple to compute $h(u)\in\QQ[u]$ so that
   $g(u)=u^{e/2}h(u+1/u)$. Any root of $g$ on $\SS$ corresponds to a
   root of $h$ on $[-2,2]$. We can then apply Sturm's algorithm, which
   uses a finite sequence of calculations in $\QQ[u]$ and sign changes
   of rationals, to compute the number of roots of $h$ in ~$[-2,2]$.
\end{proof}

Decidability of expansiveness for other groups $\Del$, even just $\G$,
 is a fascinating open question.

\begin{problem}\label{prob:expansiveness}
   Is there a finite algorithm that decides, given $f\in\ZG$, whether or
   not $f$ is expansive?
\end{problem}

There is one type of polynomial in $\ZD$ that is easily seen to be
expansive. Call $f\in\ZD$ \textit{lopsided} if there is a $\del_0\in\Del$
such that $|f_{\del_0}|>\sum_{\del\ne\del_0}|f_\del|$. This terminology
is due to Purbhoo \cite{Purbhoo}.

If $f$ is lopsided with dominant coefficient $f_{\del_0}$, adjust $f$ by
multiplying by $\pm \del_0^{-1}$ so that $f_1>
\sum_{\del\ne1}|f_\del|$. Then $f=f_1(1-g)$, where $\|g\|_1<1$. We can
then invert $f$ in $\lodr$ by geometric series:
\begin{displaymath}
   f^{-1}=\frac{1}{f_1}(1+g+g*g+\dots)\in\lodr .
\end{displaymath}
Thus lopsided polynomials are expansive.

The product of lopsided polynomials need not be lopsided (expand
$(3+u+u^{-1})^2$). Surprisingly, if $f\in\ZD$ is expansive, then there is
always a $g\in\ZD$ such that $fg$ is lopsided. This was first proved by
Purbhoo \cite{Purbhoo} for $\Del=\ZZ^d$ using a rather complicated
induction from the case $\Del=\ZZ$, but his methods also provided
quantitative information he needed to approximate algorithmically the
complex amoeba of a Laurent polynomial in several variables. The
following short proof is due to Hanfeng Li.

\begin{proposition}
   \label{prop:lopsided-is-invertible}
   Let $\Del$ be a countable discrete group and let $f\in\ZD$ be
   expansive. Then there is  $g\in\ZD$ such that $fg$ is lopsided.
\end{proposition}

\begin{proof}
   Since $f$ is expansive, by Theorem \ref{thm:expansive} there is a
   $w\in\lodr$ such that $f\cdot w=1$. There is an obvious extension of the
   definition of lopsidedness to $\lodr$. Note that
   lopsidedness is an open condition in $\lodr$. First perturb $w$
   slightly to $w'$ having finite support, and then again slightly to
   $w''$ having finite support and rational coordinates. This results in
   $w''\in\QQ\Del$ such that $fw''$ is lopsided. Choose an integer $n$
   so that $g=nw''\in\ZD$. Then $fg$ is lopsided.
\end{proof}

One consequence of the previous result is that if $f\in\ZD$ is
expansive, then the coefficients of $f^{-1}$ must decay exponentially
fast. Recall that if $\Del$ is finitely generated, then a choice of
finite symmetric generating set $S$ induces the \textit{word norm}
$|\cdot|_S$, where $|\del|_S$ is the length of the shortest word in
generators in $S$ whose product is $\del$. Clearly
$|\del_1\del_2|_S\le|\del_1|_S|\del_2|_S$. A different choice $S'$ for
symmetric generating set gives an equivalent word norm $|\cdot|_{S'}$ in
the sense that there are two constants $c_1,c_2>0$ such that
$c_1|\del|_{S}\le |\del|_{S'}\le c_2|\del|_S$ for all $\del\in\Del$.

\begin{proposition}
   \label{prop:exponential-decay}
   Let $\Del$ be a finitely generated group, and fix a finite symmetric
   generating set $S$. Suppose that $f\in\ZD$ is invertible in
   $\lodr$. Then there are constants $C>0$ and $0<r<1$ such that
   $|(f^{-1})_{\del}|\le C\,r^{|\del|_S}$ for all $\del\in\Del$.
\end{proposition}

\begin{proof}
   By the previous proposition, there is a $g\in\ZD$ such that $h=fg$ is
   lopsided. We may assume that the dominant coefficient of $h$ occurs
   at $1_{\Del}$, so that $h=q(1-b)$, where $q\in\ZZ$ and $b\in\QQ\Del$
   has $\|b\|_1=s<1$. Let $F=\supp(h)$ and put
   $\tau=\max\{|\del|_S:\del\in F\}$. Now $h^{-1}=q^{-1}(1+b+b^2+\dots)$
   and $\|b^k\|_1\le \|b\|_1^k\le s^k$ for all $k\ge1$. Furthermore,
   if $(b^n)_\del\ne0$, then $|\del|_S\le n\tau$. Hence if
   $|\del|_S>n\tau$, then $(1+b+b^2+\dots+b^n)_\del=0$. Thus
   \begin{displaymath}
      |(h^{-1})_\del|=|q^{-1}(b^{n+1}+b^{n+2}+\dots)_\del|\le
      q^{-1}\sum_{k=n+1}^\infty \|b^k\|_1\le \frac{q^{-1}}{1-r}s^{n+1}
   \end{displaymath}
   whenever $|\del|_S>n\tau$.
   This shows that $(h^{-1})_\del\le C r^{|\del|_S}$ with $r=s^{1/\tau}$
   and suitable $C>0$. Since $f^{-1}=g h^{-1}$, we obtain the result
   with the same $r$ and different $C$.
\end{proof}

We remark that a different proof of this proposition, using functional
analysis and under stricter hypotheses on $\Del$, was given in
\cite[Prop.\ 4.7]{DS}.
If $\Del=\ZZ$ and $f(u)\in\ZZ[u^{\pm}]$ is invertible in
$\ell^1(\ZZ,\RR)$, then $f$ does not vanish on $\SS$. Then $1/f$ is
holomorphic in an annular region around $\SS$ in $\CC$, and so its
Laurent expansion decays exponentially fast, giving a direct proof of
the proposition in this case.

Next we focus on the case $\Del=\G$ and obstructions of invertibility
coming from representations of ~$\G$.

Let $\SH$ be a complex Hilbert space, and let $\SB(\SH)$ be the algebra
of bounded linear operators on $\SH$. Denote by $\SU(\SH)$ the group of unitary
operators on $\SH$. An \textit{irreducible unitary representation} of
$\G$ is a homomorphism $\pi\colon\G\to\SU(\SH)$ for some complex
Hilbert space $\SH$ such that there is no nontrivial closed subspace
of $\SH$ invariant under all the $\pi(\gamma)$.
 Then $\pi$ extends to an algebraic homomorphism
$\pi\colon\logc\to\SB(\SH)$ by $\pi(\sum_{\gamma} f_\gamma
\gamma)=\sum_\gamma f_\gamma\pi(\gamma)$. If there is a nonzero $v\in\SH$
with $\pi(f)v=0$, then clearly $f$ cannot be invertible in $\logr$. The
converse is also true.

\begin{theorem}\label{thm:representations}
   Let $f\in\ZG$. Then $f$ is not invertible in $\logr$ if and only if
   there is an irreducible unitary representation
   $\pi\colon\G\to\SU(\SH)$ on some complex Hilbert space $\SH$ and a
   nonzero $v\in\SH$ such that $\pi(f)v=0$.
\end{theorem}

This result is stated in \cite[Thm.\ 8.2]{EinsiedlerRindler}, and a
detailed proof is given in \cite[Thm.\ 3.2]{GollSchmidtVerbitskiy2}. A
key element of the proof is that $\logc$ is \textit{symmetric}, i.e.,
that $1+g^*g$ is invertible for every $g\in\logc$, and in particular for
every $g\in\ZG$. There are examples of countable amenable groups that
are not symmetric \cite{Jenkins}.

We remark that Theorem \ref{thm:representations} remains valid when $\G$
is replaced by any nilpotent group, in particular by $\ZZ$ or
$\ZZ^d$. In the latter cases all irreducible unitary representations are
1-dimensional, and obtained by evaluation at a point in $\SS^d$. This is
exactly the Wiener criterion for invertibility in $\ell^1(\ZZ^d)$.

The usefulness of Theorem \ref{thm:representations} is at best limited
since $\G$ is not of Type I and so its representation theory is murky.
However, there is an extension of Gelfand theory, called Allan's local
principle, that detects invertibility
in the noncommutative Banach algebra $\logc$. The use of this principle
for algebraic actions was initiated in \cite{GollSchmidtVerbitskiy2}. In
the case of $\G$-actions this principle has an explicit and easily verified
form, which we now describe.

To simplify notation, let $B=\logc$ be the complex convolution Banach
algebra of $\G$, and $C=\lozc$ be its center. The maximal ideals of $C$
all have the form $\mz=\{v=\sum_{i=-\infty}^\infty v_jz^j:v(\zeta)=0\}$,
where $\zeta\in\SS$. For $v\in C$ the quotient norm of $v+\mz\in C/\mz$
is easily seen to be $\|v+\mz\|_{C/\mz}=|v(\zeta)|$. Let $\bz$ denote
the two-sided ideal in $B$ generated by $\mz$, so that
$$\bz=\Bigl\{w=\sum_{i,j} w_{ij}(z)x^iy^j\in B:w_{ij}(z)\in \mz \text{\ for all
$i,j\in \ZZ$}\Bigr\}.$$ Then for $w=\sum_{i,j}w_{ij}(z)x^iy^j\in B$, the
quotient norm of $w+\bz\in B/\bz$ is
\begin{displaymath}
   \|w+\bz\|_{B/\bz}=\sum_{i,j=-\infty}^\infty |w_{ij}(\zeta)|.
\end{displaymath}

To give a concrete realization of $B/\bz$, introduce variables $U$, $V$,
subject to the skew commutativity relation $VU=\zeta UV$. Then the skew
convolution Banach algebra $\ell_{\zeta}^1(U,V)$ consists of all sums
$\sum_{i,j=-\infty}^\infty c_{ij}U^iV^j$ with $c_{ij}\in \CC$ and
$\sum_{i,j}|c_{ij}|<\infty$, and with multiplication the usual
convolution modified by skew commutativity. Then the map
$B/\bz\to\ell_{\zeta}^1(U,V)$ given by $\sum_{i,j}w_{ij}(z)x^iy^j+\bz\to
\sum_{i,j}w_{ij}(\zeta)U^iV^j$ is an isometric isomorphism of complex
Banach algebras. Identifying these algebras, the quotient map
$\pi_\zeta\colon B\to B/\bz\cong \ell_{\zeta}^1(U,V)$  takes the
concrete form
$\pi_{\zeta}\bigl(\sum_{i,j}w_{ij}(z)x^iy^j\bigr)=
\sum_{i,j}w_{ij}(\zeta)U^iV^j$.

In \cite[Section 3]{GollSchmidtVerbitskiy2}, \textit{Allan's local
principle} was introduced as a convenient device for checking
expansiveness of principal actions of $\Gamma $ and applied to a number
of examples.

\begin{theorem}[Allan's local principle, \cite{Allan}]
   \label{thm:allan-principle}
   An element $w=\logc$ is invertible if and only if $\pi_{\zeta}(w)$ is
   invertible in $\ell_{\zeta}^1(U,V)$ for every $\zeta\in\SS$.
\end{theorem}

\begin{proof}
   Using the notations introduced above, clearly if $w$ is invertible in
   $B$, then its homomorphic images $\pi_{\zeta}(w)$ must all be
   invertible.

   Conversely, suppose $w\in B$ is not invertible, but
   $\pi_{\zeta}(w)$ is invertible in $\ell_{\zeta}^1(U,V)$ for every
   $\zeta\in \SS$. The left ideal $Bw$ is proper since $w$ is not
   invertible, so let $\fb$ be a maximal left ideal in $B$ containing
   $Bw$. It is easy to see that $\fb\cap C$ must be a maximal ideal in
   $C$, hence $\fb\cap C=\mz$ for some $\zeta\in\SS$. Then
   $\bz\subseteq\fb$. By assumption, $\pi_{\zeta}(w)=w+\bz$ is
   invertible in $B/\bz$, and $w\in \fb$, so that $\fb=B$, a
   contradiction.
\end{proof}

We now turn to some concrete examples where expansiveness can be
analyzed geometrically. Before doing so, let us introduce a quantity
that we will use extensively.

\begin{definition}
   Let $f(u)\in\ZZ[u^{\pm}]$. The \textit{logarithmic Mahler measure} of
   $f$ is defined as $\mahler(f)=\int_\SS \log|f(\xi)|\,d\xi$. The
   \textit{Mahler measure} of $f$ is defined as $\Mahler(f)=\exp(\mahler(f))$.
\end{definition}

Suppose that $f(u)=c_nu^n+\dots+c_u+c_0$ with $c_nc_0\ne0$. Factor
$f(u)$ over $\CC$ as $c_n\prod_{j=1}^n (u-\lam_j)$. Then Jensen's
formula shows that $\mahler(f)$ has the alternative expression as
\begin{equation}\label{eqn:mahler}
   \mahler(f)=\log|c_n|+\sum_{|\lam_j|>1}\log|\lam_j|=\log|c_n|+\sum_{j=0}^n\log^+|\lam_j|,
\end{equation}
where $\log^+ r = \max\{\log r, 0\}$ for $r\ge0$.

Mahler's motivation was to derive important inequalities in transcendence
theory. Using that $\Mahler(fg)=\Mahler(f)\Mahler(g)$, he showed that if
$f,g\in\CC[u]$, then $\|fg\|_1\ge 2^{-\deg f-\deg g}\|f\|_1\|g\|_1$, and
that the constant is best possible.

Let us consider the case of $f\in\ZG$ that is linear in $y$, so that
$f(x,y,z)=h(x,z)y-g(x,z)$. We will find 1-dimensional $\af$-invariant
subspaces of $\ligc$ as follows. Let $\Lam=\<x,z\>$ be the subgroup of $\G$
generated by $x$ and $z$. To make calculations in $\ligc$ and $\lilc$
more transparent, we will write elements as formal sums $w=\sum_{\gamma\in\G}
w(\gamma)\gamma$ and $v=\sum_{\lam\in\Lam}v(\lam)\lam$.

Fix $(\xi,\zeta)\in\SS^2$ and define
\begin{equation}
   \label{eqn:vxz}
   \vxz = \sum_{k,m=-\infty}^\infty \xi^k\zeta^m\,x^kz^m\in\lilc.
\end{equation}
Observe that
\begin{displaymath}
   \rho_x(\vxz)=\Bigl(\sum_{k,m}\xi^k\zeta^m\,x^kz^m\Bigr)x^{-1}=\xi
   \,\sum_{k,m}\xi^k\zeta^m \,x^kz^m=\xi \,\vxz,
\end{displaymath}
and similarly $\rho_z(\vxz)=\zeta\,\vxz$. Hence the 1-dimensional space
$\CC\vxz$ is a common eigenspace for $\rho_x$ and $\rho_z$. It follows
that $\rho_q(\vxz)=\vxz\,q^*(x,z)=q(\xi,\zeta)\,\vxz$ for every
$q\in\ZL$.

Let $\{c_n\}$ be a sequence of complex constants to be determined, and consider
the point $w=\sum_{n=-\infty}^\infty c_n\,\vxz\,y^n$. Using the
relations $y^kq(x,z)=q(xz^k,z)y^k$ for all $k\in\ZZ$ and all $q\in\ZL$,
the condition
$\rf(w)=w\cdot f^*=0$ becomes
\begin{align*}
   0&=\Bigl(\sum_{n=-\infty}^\infty
   c_n\,\vxz\,y^n\Bigr)\Bigl(y^{-1}h^*(x,z)-g^*(x,z)\Bigr)\\
   &=\sum_{n=-\infty}^\infty
   \Bigl\{c_n\,\vxz\,h^*(xz^{n-1},z)y^{n-1}-c_n\,\vxz\,g^*(xz^{n},z)y^n\Bigr\}\\
   &=\sum_{n=-\infty}^\infty
   \Bigl\{c_{n+1}h(\xi\zeta^{n},\zeta)-c_ng(\xi\zeta^{n},\zeta)\Bigr\}\vxz\,y^n.
\end{align*}
This calculation shows that $\rf(w)=0$ if and only if
the $c_n$ satisfy
\begin{equation}\label{eqn:recurrence}
   c_{n+1}h(\xi\zeta^{n},\zeta)=c_ng(\xi\zeta^{n},\zeta)\quad\text{for all $n\in\ZZ$}.
\end{equation}
Since $\|w\|_\infty=\sup_{n\in\ZZ}|c_n|$, one way to create nonexpansive
$f$'s of this form is to find conditions on $g$ and $h$ that guarantee
the existence of a nonzero bounded solution $\{c_n\}$ to \eqref{eqn:recurrence}
for some choice of $\xi$ and $\zeta$.

Suppose that both $g$ and $h$ do not vanish on $\SS^2$.
Fix a nonzero value of $c_0$. Then by \eqref{eqn:recurrence}, the other values
of $c_n$ are determined:
\begin{equation}
   \label{eqn:coefficients}
   c_n=
   \begin{cases}
      c_0\prod_{j=0}^{n-1}\displaystyle\frac{g(\xi\zeta^{j},\zeta)}
      {h(\xi\zeta^{j},\zeta)}
      &\text{for $n\ge1$}, \\
      c_0\prod_{j=1}^{-n}\Bigl[\displaystyle\frac{g(\xi\zeta^{-j},\zeta)}
      {h(\xi\zeta^{-j},\zeta)}\Bigr]^{-1}
      &\text{for $n\le-1$}.
   \end{cases}
\end{equation}
Let
\begin{equation}
   \label{eqn:cocycle-gh}
   \phiz(\xi)=\log|g(\xi,\zeta)/h(\xi,\zeta)|,
\end{equation}
and consider the map $\psiz(n,\xi)\colon\ZZ\times\RR\to\RR$ given by
\begin{equation}
   \label{eqn:additive-cocycle}
   \psiz(n,\xi)=
   \begin{cases}
      \sum_{j=0}^{n-1}\phiz(\xi\zeta^{j})&\text{for $n\ge1$},\\
      0 &\text{for $n=0$},\\
      -\sum_{j=1}^{-n}\phiz(\xi\zeta^{-j}) &\text{for $n\le -1$}.
   \end{cases}
\end{equation}
Then $\psiz$ satisfies the cocycle equation
\begin{displaymath}
   \psiz(m+n,\xi)=\psiz(m,\xi)+\psiz(n,\xi\zeta^{m})
\end{displaymath}
for all $m,n\in\ZZ$ and $\xi\in\SS$. Furthermore, there is a nonzero
bounded solution $\{c_n\}$ to \eqref{eqn:recurrence} if and only if
$\{\psiz(n,\xi):-\infty< n < \infty\}$ is bounded above.

Suppose that $\zeta\in\SS$ is irrational, and that there is a
$\xi\in\SS$ for which \eqref{eqn:recurrence} has a nonzero bounded
solution. Observe that $\phiz$ is continuous on $\SS^2$ since neither
$g$ nor $h$ vanish there by assumption. By the ergodic theorem,
   \begin{gather*}
      \int_{\SS}\phiz(\xi)\,d\xi=\lim_{n\to\infty}\frac1{n}
      \sum_{j=0}^{n-1}\phiz(\xi\zeta^{j})=\lim_{n\to\infty}\frac1n \psiz(n,\xi)\le0,
      \text{\quad and}\\
      -\int_{\SS}\phiz(\xi)\,d\xi=\lim_{n\to\infty}-\frac{1}{n}
      \sum_{j=1}^n\phiz(\xi\zeta^{-j})=\lim_{n\to\infty}\frac{1}{n}\psiz(-n,\xi)\le0.
   \end{gather*}
Thus $\int_{\SS}\phiz(\xi)\,
d\xi=\mahler(g(\cdot,\zeta))-\mahler(h(\cdot,\zeta))=0$.
There is a similar necessary condition when $\zeta$ is ration\-al, but
here the integral is replaced by a finite sum over a coset of the finite
orbit of $\zeta$ in $\SS$. It turns out that these necessary conditions are also
sufficient.

\begin{theorem}\label{thm:linear-expansive}
   Let $f(x,y,z)=h(x,z)y-g(x,z)\in\ZG$, and suppose that both $g$ and
   $h$ do not vanish anywhere on $\SS^2$. Let
   $\phiz(\xi)=\log|g(\xi,\zeta)/h(\xi,\zeta)|$. Then $\af$ is expansive
   if and only if
   \begin{enumerate}
     \item $\int_{\SS}\phiz(\xi)\,d\xi\ne0$ for every irrational
      $\zeta\in\SS$, and
     \item for every $n$-th root of unity $\zeta\in\SS$ and
      $\xi\in\SS$ we have that $\sum_{j=0}^{n-1}\phiz(\xi\zeta^j)\ne~0$.
   \end{enumerate}
\end{theorem}

\begin{proof}
   Suppose first that $\af$ is nonexpansive. By Theorem
   \ref{thm:expansive}, the $\G$-invariant subspace
   $K=\{w\in\ligc:\rf(w)=0\}$ is nontrivial. By restricting the
   $\G$-action on $K$ to the commutative subgroup $\Lam$, we can apply
   the argument in \cite[Lemma 6.8]{SchmidtBook} to find a 1-dimensional
   $\Lam$-invariant subspace $W\subset K$. In other words, there are a
   nonzero $w\in K$ and $\xi,\zeta\in\SS$ such that $\rho_x(w)=\xi w$
   and $\rho_z(w)=\zeta w$. It follows from this and the above
   discussion that $w$ must have the form
   $w=\sum_{n=-\infty}^{\infty}c_n\vxz y^n$ with the $c_n\in\CC$ given
   by \eqref{eqn:recurrence}, and with $\{c_n\}$ a bounded sequence. If
   $\zeta$ is an $n$-th root of unity, then clearly
   $\sum_{j=0}^{n-1}\phiz(\xi\zeta^j)=0$ by boundedness of the $c_n$,
   while if
   $\zeta$ is irrational, then the discussion above shows that
   $\int_{\SS}\phiz(\xi)\,d\xi=0$.

   For the converse, suppose first that $\zeta$ is rational, say
   $\zeta^k=1$. If there is a $\xi\in\SS$ with $\xi\in\SS$ with
   $\sum_{j=0}^{k-1}\phiz(\xi\zeta^{-j})=0$, then \eqref{eqn:coefficients}
   and \eqref{eqn:cocycle-gh} show that there is a bounded sequence
   $\{c_n\}$, in this case even periodic with period $k$, so that $w=\sum
   c_n\vxz\,y^n\in\ligc$ with $\rf(w)=0$, showing that $\af$ is
   nonexpansive.

   Finally, suppose that $\zeta$ is irrational and
   $\int_{\SS}\phiz(\xi)\,d\xi=0$. If there were a continuous coboundary
   $b$ on $\SS$ such that $\phiz(\xi)=b(\xi\zeta)-b(\xi)$, then the
   cocycle $\psiz(\cdot,\xi)$ would be bounded for every value of $\xi$,
   and as before this means we can form an nonzero point $w\in\ligc$
   with $\rf(w)=0$, so $\af$ is nonexpansive. So suppose no such
   coboundary exists. Let $Y=\SS\times\RR$, and define a skew product
   transformation $S\colon Y\to Y$ by
   $S(\xi,r)=(\xi\zeta,r+\phiz(\xi))$. Since $\phiz$ is not a coboundary,
   but $\int_{\SS}\phiz(\xi)\,d\xi=0$, the homeomorphism $S$ is
   topologically transitive (see \cite{GottschalkHedlund} or
   \cite[Thms.\ 1 and 2]{Atkinson}). By \cite[p.\ 38f]{Bes}, there exists
   a point $\xi\in\SS$ such that the entire $S$ orbit of $(\xi,0)$ has
   its second coordinate bounded above. Since
   $S^n(\xi,0)=(\xi\zeta^{n},\psiz(n,\xi))$, it follows that for this
   choice of $\xi$, the sequence $\{c_n\}$ from \eqref{eqn:coefficients}
   and \eqref{eqn:cocycle-gh} is also bounded, and therefore
   $w=\sum_{n=-\infty}^\infty c_n\vxz y^n\in\ligc$ with $\rf(w)=0$, and so
   $\af$ is nonexpansive.
\end{proof}

\begin{remark}
   With the assumptions of Theorem \ref{thm:linear-expansive} that
   neither $g$ nor $h$ vanish on $\SS^2$, the functions
   $\mahler\bigl(g(\cdot,\zeta)\bigr)$ and
   $\mahler\bigl(h(\cdot,\zeta)\bigr)$ are both continuous functions of
   $\zeta$. Hence the conditions (1) and (2) in Theorem
   \ref{thm:linear-expansive} combine to say that the graphs of these
   functions never cross (for rational $\zeta$ use the Mean Value
   Theorem). In particular, if $h(x,z)\equiv 1$, then since
   $\int_{\SS}\mahler\bigl(g(\cdot,\zeta)\bigr)\,d\zeta=\mahler(g)\ge0$,
   the condition for expansiveness of $f(x,y,z)=y-g(x,z)$ becomes simply
   that $\mahler\bigl(g(\cdot,\zeta)\bigr)>0$ for all $\zeta$.
\end{remark}

\begin{example}
   The polynomial $f(x,y,z)=3+x+y+z$, although not lopsided, was shown
   to be expansive in \cite[Example\ 7.4]{EinsiedlerRindler}. In
   \cite[Example 3.6]{GollSchmidtVerbitskiy2} four different ways to
   verify its expansiveness are given: (1) using irreducible unitary
   representations, (2) direct computation of its inverse in $\logr$,
   (3) using Allan's local principle, and (4) using the geometric
   argument in Theorem \ref{thm:linear-expansive}.

   Many more examples illustrating various aspects of expansiveness (or
   its lack) for polynomials in $\ZG$ are contained in
   \cite{GollSchmidtVerbitskiy2}.
\end{example}

\begin{example}[\protect{\cite[Example 5.11]{GollSchmidtVerbitskiy2}}]
   \label{exam:x+y+z+2}
   Let $f(x,y,z)=x+y+z+2=y-g(x,z)\in\ZG$, where $g(x,z)=-x-z-2$. Since
   $g(-1,-1)=0$, it does not quite satisfy the hypothesis of Theorem~
   \ref{thm:linear-expansive}. However, we can directly use the
   violation of (2) there to show that $\af$ is nonexpansive.

   Let $\zeta_0=-1$. We want to find $\xi\in\SS$ such that
   $|g(\xi,-1)g(-\xi,-1)|=1$. This amounts to solving $|\xi^2-1|=1$, and
   this has the four solutions $\pm e^{\pm\pi i/6}$. Let $\xi_0=e^{\pi
   i/6}$, and consider the point $v_{\xi_0,-1}$ as defined in
   \eqref{eqn:vxz}. Then the coefficients $c_n$ of the point
   $\sum_{n=-\infty}^\infty c_nv_{\xi_0,-1}y^n\in\ker \rf$ satisfy
   \eqref{eqn:recurrence}, and are hence alternately multiplied by
   $g(\xi_0,-1)$ and by $g(-\xi_0,-1)$, where
   $|g(\xi_0,-1)|=\sqrt{2-\sqrt{3}}\cong 0.51764$ and
   $|g(-\xi_0,-1)|=1/|g(\xi_0,-1)|\cong 1.93185$. Thus
   $\{c_n\}_{n=-\infty}^\infty$ is bounded, and hence $\af$ is
   nonexpansive.

   Note that here $\mahler(g(\cdot,\zeta))=\log|\zeta+2|$, which vanishes
   only at $\zeta=-1$. Hence for all irrational $\zeta$, condition (1)
   of Theorem \ref{thm:linear-expansive} is satisfied. It is an easy
   exercise to show that if $\zeta\ne-1$ is rational, then condition
   (2) is also satisfied. So here the only values of $(\xi,\zeta)$
   leading to a bounded solution $\{c_n\}$ of \eqref{eqn:recurrence}
   are $(\pm e^{\pm\pi i/6},-1)$. This example appears as \cite[Example
   10.6]{EinsiedlerRindler}, but was claimed there to be expansive.
\end{example}

\begin{example}
   \label{exam:y^2-xy-1}
   Let $f(x,y,z)=y^2-xy-1\in\ZG$. For $(\xi,\zeta)\in\SS^2$ let $\vxz$
   be defined by \eqref{eqn:vxz}, and let $w=\sum_{n=-\infty}^\infty
   c_n\vxz y^n$. Although $f$ is now quadratic in $y$, we can still
   calculate $\rf(w)$ as before, finding that $\rf(w)=0$ if and only if
   \begin{equation}
      \label{eqn:quadratic-recurrence}
      c_{n+2}=\xi\,\zeta^{n} c_{n+1}+c_n  \text{\quad for all $n\in\ZZ$}.
   \end{equation}
   Thus we need conditions on $(\xi,\zeta)$ for which
   \eqref{eqn:quadratic-recurrence} has a bounded solution.

   For $n\ge1$ put
   \begin{equation}
      \label{eqn:quadratic-cocycle}
      A_n(\xi,\zeta)=\begin{bmatrix} 0 & 1\\
         1&\xi\zeta^{n-1}\end{bmatrix}
         \begin{bmatrix} 0 & 1\\
            1&\xi\zeta^{n-2}\end{bmatrix}\dots
         \begin{bmatrix} 0 & 1\\
         1&\xi\end{bmatrix}.
   \end{equation}
   Then the recurrence \eqref{eqn:quadratic-recurrence} shows that
   \begin{displaymath}
      A_n(\xi,\zeta)
     \begin{bmatrix} c_0 \\ c_1 \end{bmatrix} =
     \begin{bmatrix} c_{n} \\ c_{n+1}\end{bmatrix}
      \text{\quad for all $n\ge1$},
   \end{displaymath}
   and there is a similar formula for $n\le-1$.

   We are therefore reduced to finding $(\xi,\zeta)$ such that the
   matrix-valued cocycle $\{A_n(\xi,\zeta)\}$ is bounded. The easiest
   place to look is at $\zeta=1$, since $A_n(\xi,1)=A_1(\xi,1)^n$. Now
   $A_1(\xi,1)$ has eigenvalues on $\SS$ provided that the roots of
   $u^2-\xi u-1=0$ are there. This happens exactly when $\xi=\pm
   i$. Hence if we define $c_n$ by $$\begin{bmatrix}
      c_{n}\\c_{n+1}\end{bmatrix}=
   \begin{bmatrix}0 & 1\\1 &
      i\end{bmatrix}^n\begin{bmatrix}0\\1\end{bmatrix} \text{\quad for
   all $n\in\ZZ$,}$$
   then $w=\sum_{k,n,m=-\infty}^\infty
   c_ni^k\,x^ky^nz^m\in\ligc\cap\ker(\rf)$, and so $\af$ is
   nonexpansive.

   For some values of $(\xi,\zeta)$ the growth rate of $A_n(\xi,\zeta)$ can be
   strictly positive, for example $(1,1)$. However,
   as we will see later,
   the growth rate of $A_n(\xi,\zeta)$ as
   $n\to\pm\infty$ is zero for almost every $(\xi,\zeta)\in\SS^2$,
   although this is not at all obvious.
   \end{example}

   If $g$ or $h$ are allowed to vanish on $\SS^2$,
   there is a completely different source of nonexpansive behavior.
   \begin{example}
      Let $f(x,y,z)=h(x,z)y-g(x,z)\in\ZG$, and suppose that $g(x,z)$ and
      $h(xz^{-1},z)$ have a common zero $(\xi,\zeta)\in\SS^2$. Consider
      $\vxz$ as a point in $\ligc$. Then
      \begin{align*}
         \rf(\vxz)&=\vxz\cdot\bigl(y^{-1}h^*(x,z)-g^*(x,z)\bigr)
         =\vxz\cdot\bigl(h^*(xz^{-1},z)y^{-1}-g^*(x,z)\bigr)\\
         &=h(\xi\zeta^{-1},\zeta)\vxz\cdot y^{-1}-g(\xi,\zeta)\vxz = 0,
      \end{align*}
      so that $\rf$ is not expansive.

      There is a simple geometric method to create examples of this
      situation. Start with two polynomials $g(x,z)$ and $h_0(x,z)$ whose unitary varieties in $\SS^2$ do not intersect, but which have the property that the unitary variety of the skewed
      polynomial $h_0(x z^m,z)$ intersects that of $g$ for sufficiently large $m$. An
      example of this form are described in \cite[Example
      8.5]{GollSchmidtVerbitskiy1}, which also contains other results
      about expansiveness in cases when at least one unitary variety is
      nonempty. However, this analysis leaves open one very interesting
      case. The simplest version of this question is the following.

      \begin{problem}
         Let $g(x,z)\in\ZZ[x^{\pm},z^{\pm}]$, and suppose there is a
         $\zeta\in\SS$ for which $\mahler\bigl(g(\cdot,\zeta)\bigr)=0$
         and such that $g(\xi,\zeta)$ vanishes for at least one value of
         $\xi\in\SS$. Is there a value of $\xi$ for which the partial
         sums $\sum_{j=0}^n \log|g(\xi\zeta^j,\zeta)|$ are bounded above
         for all $n$?
      \end{problem}

       Fr\c{a}czek and Lema\'{n}czyk showed in \cite{FraczekLemanczyk} that
       these sums are unbounded for almost every $\xi\in\SS$. The
       argument of Besicovitch \cite{Bes} to prove the existence of bounded sums
       makes essential use of continuity, and does not apply in this
       case where there are logarithmic singularities.
   \end{example}

\begin{remark}\label{rem:algebraic}
   By invoking some deep results in diophantine approximation theory, we
   can show that the second alternative in the last paragraph of the
   proof of Theorem \ref{thm:linear-expansive} never occurs. For we claim that if
   $\int_{\SS}\phiz(\xi)\,d\xi=0$, then $\zeta$ must be an algebraic number
   in $\SS$. Assuming this for the moment, then a result of Gelfond
   \cite{Gelfond} shows that for every $\epsilon>0$ there is a $C>0$
   such that $|\zeta^n-1|>C e^{-n\epsilon}$ for every $n\ge1$. Since $\phiz(\xi)$ is
   smooth, its Fourier coefficients $\widehat{\phiz}(n)$ decay rapidly
   as $|n|\to\infty$, and so the formal solution $b(\xi)=\sum
   _{-\infty}^\infty b_n\xi^n$ to $\phiz(\xi)=b(\xi\zeta)-b(\xi)$, with
   $b_n=\widehat{\phiz}(n)/(\zeta^n-1)$ for $n\neq0$, also decays rapidly, giving a
   continuous solution $b$.

   To justify our claim, write
   \begin{displaymath}
      g(x,\zeta)=A(\zeta)\prod_{j=1}^m(x-\lam_j(\zeta)) \text{\quad
      and\quad}
      h(x,\zeta)=B(\zeta)\prod_{k=1}^n(x-\mu_k(\zeta)),
   \end{displaymath}
   where $A(z)$ and)
   $B(z)$ are Laurent polynomials with integer coefficients and the
   $\lam_j$ and $\mu_k$ are algebraic functions. The condition
   $\int_{\SS}\phiz(\xi)\,d\xi=0$ becomes, via Jensen's formula,
   \begin{displaymath}
      \Bigl| A(\zeta) \prod_{|\lam_j(\zeta)|>1}\lam_j(\zeta)\Bigr|=
      \Bigl|B(\zeta)\prod_{|\mu_k(\zeta)|>1}\mu_k(\zeta)\Bigr|,
   \end{displaymath}
   which is an algebraic equation in $\zeta$, so $\zeta$ is algebraic.

   From this, and the preceding proof, under our assumptions that
   neither $g$ nor $h$ vanish anywhere on $\SS^2$, we conclude that if
   $\int_{\SS}\phiz(\xi)\,d\xi=0$ and $\zeta$ is irrational, then
   \textit{every} choice of $\xi$ will yield a nonzero point
   $w=\sum c_n\vxz\,y^n\in\ligc$ with $\rf(w)=0$.

   This idea was observed independently by Evgeny Verbitskiy (oral
   communication).
\end{remark}

To illustrate some of the preceding ideas, we provide an informative
example. This was chosen so that the diophantine estimates mentioned in
Remark \ref{rem:algebraic} can be given an elementary and self-contained
proof, rather than appealing to difficult general diophantine
results. In addition, the constants in our analysis are effective enough
to rule out nonexpansive behavior at all rational $\zeta$. One
consequence is that for this algebraic $\G$-action, nonexpansiveness
cannot be detected by looking at only finite-dimensional representations
of $\G$.

\begin{example}\label{exam:nonexpansive}
   Let $g(x,z)=a(x)c(z)$, where $a(x)=x^2-x-1$ and
   $c(z)=z^{12}+z^2+1$. It is easy to check that neither $a$ nor $c$
   vanishes on $\SS$, so that $g$ does not vanish anywhere on
   $\SS^2$. Also, $a(x)$ has roots $\tau=(1+\sqrt{5})/2$ and
   $\sigma=(1-\sqrt{5})/2=-\tau^{-1}$.

   Consider $f(x,y,z)=y-g(x,z)$, so that here $h(x,z)\equiv1$. Then
   $\phiz(\xi)=\log|g(\xi,\zeta)|\linebreak[0]=\log|a(\xi)|+\log|c(\zeta)|$, and
   define $\psiz(n,\xi)$ as in \eqref{eqn:additive-cocycle}. Call
   $(\xi,\zeta)\in\SS^2$ \textit{nonexpansive for} $g$ if the sequence
   $\{\psiz(n,\xi):n\in\ZZ\}$ is bounded.

   We claim there are eight values $\zeta_1,\dots,\zeta_8\in\SS$, which
   are algebraically conjugate algebraic integers of degree 48, such
   that the nonexpansive points for $g$ are exactly those of the form
   $(\xi,\zeta_k)$, where $\xi$ is any element of $\SS$ and $1\le k\le
   8$.

   We start with two simple results.

   \begin{lemma}\label{lem:coboundary}
      Suppose that $\zeta\in\SS$ is irrational, and that there are
      constants $C>0$ and $0<r<1$ such that $|\zeta^n-1|\ge Cr^n$ for
      all $n\ge1$. If $\kappa\in\CC$ with $|\kappa|<r$, then the
      function $\xi\mapsto\log|1-\xi\kappa|$ on $\SS$ is a continuous
      coboundary for $\zeta$, i.e., there is a continuous
      $b\colon\SS\to\RR$ such that $\log|1-\xi\kappa|=b(\xi\zeta)-b(\xi)$.\
   \end{lemma}

   \begin{proof}
      Since $|\xi\kappa|<1$, the Taylor series for $\log(1-\xi\kappa)$
      converges, so that
      \begin{displaymath}
         \log(1-\xi\kappa)=-\sum_{n=1}^\infty\frac{\xi^n\kappa^n}{n}.
      \end{displaymath}
      Let $B(\xi)=\sum_{n=1}^\infty b_n\xi^n$. Then
      $B(\xi\zeta)-B(\xi)=\log(1-\xi\kappa)$ provided that
      $b_n=-\kappa^n/[(\zeta^n-1)n]$ for all $n\ge1$. The assumption on $\zeta$ implies
      that $|b_n|\le C^{-1}(|\kappa|/r)^n$, so that the series for $B$
      converges uniformly on $\SS$. Then $b(\xi)=\Re\{B(\xi)\}$ gives
      the required coboundary.
   \end{proof}

   \begin{lemma}\label{lem:diophantine}
      Let $p(u)\in\ZZ[u]$ be monic and irreducible. Suppose that $p$ has
      a root $\zeta\in\SS$ that is irrational. Then there is a constant
      $C>0$ such that
      \begin{equation}
         \label{eqn:mahler-estimate}
         |\zeta^n-1|\ge C \Mahler(p)^{-n/2} \text{\quad for all $n\ge1$},
      \end{equation}
      where $\Mahler(p)>1$ is the Mahler measure of $p$.
   \end{lemma}

   \begin{proof}
      Factor $p$ over $\CC$ as
      $p(u)=(u-\zeta)(u-\bar{\zeta})\prod_{j=1}^r(u-\lam_j)$, where
      $\bar{\zeta}$ is the complex conjugate of $\zeta$. The polynomial
      $p^{(n)}(u)=(u-\zeta^n)(u-\bar{\zeta}^n)\prod_{j=1}^r(u-\lam_j^n)$
      has integer coefficients, and $p^{(n)}(1)\ne0$,  hence
      $|p^{(n)}(1)|\ge1$. Then using the trivial estimates that
      $|\lam^n-1|\le2$ if $|\lam|\le1$ and
      $|\lam^n-1|\le(1-1/|\lam|)|\lam^n|$ if $|\lam|>1$, we obtain that
      \begin{align*}
         \label{eqn:lower-estimate}
         1\le|p^{(n)}(1)| &=
         |\zeta^n-1||\bar{\zeta}^n-1|\prod_{j=1}^r|\lam_j^n-1| \\
         &\le |\zeta^n-1|^2 2^r\prod_{|\lam_j|>1}|\lam_j^n-1| \\
         &\le |\zeta^n-1|^2
         2^r\Bigl\{\prod_{|\lam_j|>1}\Bigl(1-\frac{1}{|\lam_j|}\Bigr)\Bigr\}
         \Mahler(p)^n.
      \end{align*}
      Hence \eqref{eqn:mahler-estimate} is valid with
      \begin{displaymath}
         C=2^{-r/2}\prod_{|\lam_j|>1}\Bigl(1-\frac{1}{|\lam_j|}\Bigr)^{-1/2}.
         \qedhere
      \end{displaymath}
   \end{proof}

   Recall that $a(x)=(x-\tau)(x-\sigma)$, so that
   $\mahler(a)=\log\tau$. By Theorem \ref{thm:linear-expansive}, we need
   to find just those $\zeta\in\SS$ for which
   \begin{displaymath}
      0=\int_{\SS}\log|g(\xi,\zeta)|d\,\xi=\int_{\SS}(\log|a(\xi)|+\log|c(\zeta)|)\,d\xi
      =\log \tau + \log|c(\zeta)|,
   \end{displaymath}
   that is, we must solve the equation $|c(\zeta)|=\tau^{-1}$. Any such
   solution  $\zeta\in\SS$ would also satisfy $c(z)c(1/z)=\tau^{-2}$,
   or, equivalently, satisfy
   \begin{displaymath}
      F(z)=z^{12}\bigl(c(z)c(1/z)-\tau^{-2}\bigr)\in\QQ(\tau)[z].
   \end{displaymath}
   Here $F(z)$ has degree 24, and has eight roots
   $\zeta_1,\dots,\zeta_8\in\SS$. To show these are algebraic integers,
   multiply $F(z)$ by its Galois conjugate over $\QQ(\tau)$, yielding
   the following polynomial
   \begin{align*}\label{eqn:degree48}
      G(z)=z^{48}+2 &z^{46}+z^{44}+2 z^{38}+5 z^{36}+5
   z^{34}+2 z^{32}+z^{28}+5 z^{26}\\&+7 z^{24}+5
   z^{22}
   +z^{20}+2 z^{16}+5 z^{14}+5 z^{12}+2
   z^{10}+z^4+2 z^2+1,
  \end{align*}
  whose irreducibility is confirmed by \textit{Mathematica}. Hence the
  $\zeta_k$ are conjugate algebraic integers of degree 48 as claimed.

  Now $G$ has 10 roots outside the unit circle, whose product is
  $\Mahler(G)\cong1.90296$. Then$\sqrt{\Mahler(G)}\cong1.37948<\tau$,
  which plays a crucial role in dealing with rational $\zeta$.

  By Theorem \ref{thm:expansive}, the only irrational $\zeta$ we need to
  consider are the eight $\zeta_k$. By Lemma \ref{lem:diophantine}, there
  is a constant $C>0$ such that $|\zeta_k^n-1|>Cr^n$, where
  $r=\Mahler(G)^{-1/2}<1$. Since $\tau^{-1}=|\sigma|<r$, by Lemma
  \ref{lem:coboundary}, we can find continuous coboundaries $b_1$ and
  $b_2$ so that
  \begin{displaymath}
    \log|1-\tau^{-1}\xi|=b_1(\xi\zeta_k)-b_1(\xi)\text{\quad and\quad}
    \log|\xi-\sigma|=b_2(\xi\zeta_k)-b_2(\xi).
  \end{displaymath}
  Hence
  \begin{displaymath}
     \psi_{\zeta_k}(\xi)=\log|(\xi-\tau)(\xi-\sigma)c(\zeta_k)|
     =\log|1-\tau^{-1}\xi|+\log|\xi-\sigma|
  \end{displaymath}
  is also a coboundary. Thus $(\xi,\zeta_k)$ is nonexpansive for every
  $\xi\in\SS$. This is an example of the algebraic phenomenon discussed
  in Remark \ref{rem:algebraic}.

  To complete our analysis, we turn to the rational case, say
  $\zeta=\om$, a primitive $n$-th root of unity. The idea of the
  following argument is that is to show that there is a large enough $N_0$
  such that for all $n\ge N_0$, the variation of the $n$-periodic function
  $\prod_{j=0}^{n-1}a(\xi\om^j)$ is small compared to $|c(\om)|$. The
  estimates are sharp enough to obtain the bound $N_0=143$, and the
  remaining cases $n<143$ can be checked by hand.

  First observe that
  \begin{align*}
     \prod_{j=0}^{n-1}\bigl\{a(\xi\om^j)\tau^{-1}\bigr\}&=\prod_{j=0}^{n-1}
     \Bigl(\frac{\xi\om^j-\tau}{\tau}\Bigr) (\xi\om^j-\sigma)
     =(-1)^n\prod_{j=0}^{n-1}(1-\tau^{-1}\xi\om^j)(\xi\om^j-\sigma)\\
     &=(-\xi^n)(1-\tau^{-n}\xi^n)(1-\sigma^n\xi^{-n}).
  \end{align*}
  Since $\bigl|\log|1-\kappa|\bigr|\le2|\kappa|$ for all $\kappa\in\CC$
  with $|\kappa|\le\tau^{-1}$, it follows that
  \begin{equation}
     \label{eqn:a-variation}
     \Bigl|\log |\prod_{j=0}^{n-1} a(\xi\om^j)\tau^{-1}|\Bigr|\le 4\tau^{-n}
  \end{equation}
  for all $\xi\in\SS$ and all $n\ge1$. Recalling that
  $c(\zeta_k)=\tau^{-1}$, we obtain that
  \begin{displaymath}
     \log\Bigl|\prod_{j=0}^{n-1}a(\xi\om^j)c(\om)\Bigr|
     =\log\Bigl|\prod_{j=0}^{n-1}a(\xi\om^j)\tau^{-1}\Bigr|
     +\log|[c(\om)/c(\zeta_k)]^n|.
  \end{displaymath}
  Thus, if
  \begin{equation}
     \label{eqn:estimate}
     \bigl|n\log|c(\om)\tau|\bigr|>5\tau^{-n}
  \end{equation}
  then we must have $\sum_{j=0}^{n-1}\phi_{\zeta_k}(\xi \om^j)\ne0$ for
  all $\xi\in\SS$.

  To obtain a reasonable bound for $N_0$ so that \eqref{eqn:estimate}
  holds for all $n\ge N_0$, we need to make a  more careful estimate
  in the proof of Lemma \ref{lem:diophantine} for the polynomial
  $G(z)$. Here $G(z)$ has ten roots $\lam_1,\dots,\lam_{10}$ outside the
  unit circle. Thus besides these and $\zeta$ and $\bar{\zeta}$, there
  are 36 roots of $G$ on or inside the unit circle. The estimate in the
  proof of Lemma \ref{lem:diophantine} can be refined as
  \begin{displaymath}
     1\le|\zeta_k^n-1|^2\cdot 2^{36}\prod_{j=1}^{10}|\lam_j^n-1|
     =|\zeta_k^{n}-1|^2\cdot 2^{36}\Mahler(G)^n\prod_{j=1}^{10}
     \Bigl|1-\frac{1}{\lam_j^n}\Bigr|.
  \end{displaymath}
  It is easy to check that $\prod_{j=1}^{10}|1-\lam_j^{-n}|$ has a
  maximum of about 37.94 at $n=6$. Hence
  \begin{displaymath}
     |\zeta_k^n-1|\ge\frac{1}{2^{18}\sqrt{40}}\Mahler(G)^{-n/2}
     \text{\quad for all $n\ge1$}
  \end{displaymath}
  Since
  $\zeta_k^n-1=\zeta_k^n-\om^n=(\zeta_k-\om)(\zeta_k^{n-1}+\zeta_k^{n-2}\om+
  \dots+\om^{n-1})$, it follows that $|\zeta_k-\om|\ge\frac{1}{n}|\zeta_k^n-1|$.

  Verifying \eqref{eqn:estimate} breaks into two cases, depending on
  whether or not $\om$ is close to some $\zeta_k$. Let
  $\epsilon_0=0.01$. It is an exercise in calculus to show that if
  $|\om-\zeta_k|>\epsilon_0$ for every $k$, then
  $\bigl|\log|c(\om)\tau|\bigr|\ge\epsilon_0$, while if $|e^{2 \pi i
  s}-\zeta_k|<\epsilon_0$ for some $k$, then the derivative of
  $\log|c(e^{2\pi i s})\tau|$ has absolute value $\ge1$. A glance at
  Figure \ref{fig:graph}
  should make clear the
  meaning of these statements.

  \begin{figure}[h!]
     \centering
     \includegraphics[scale=.35]{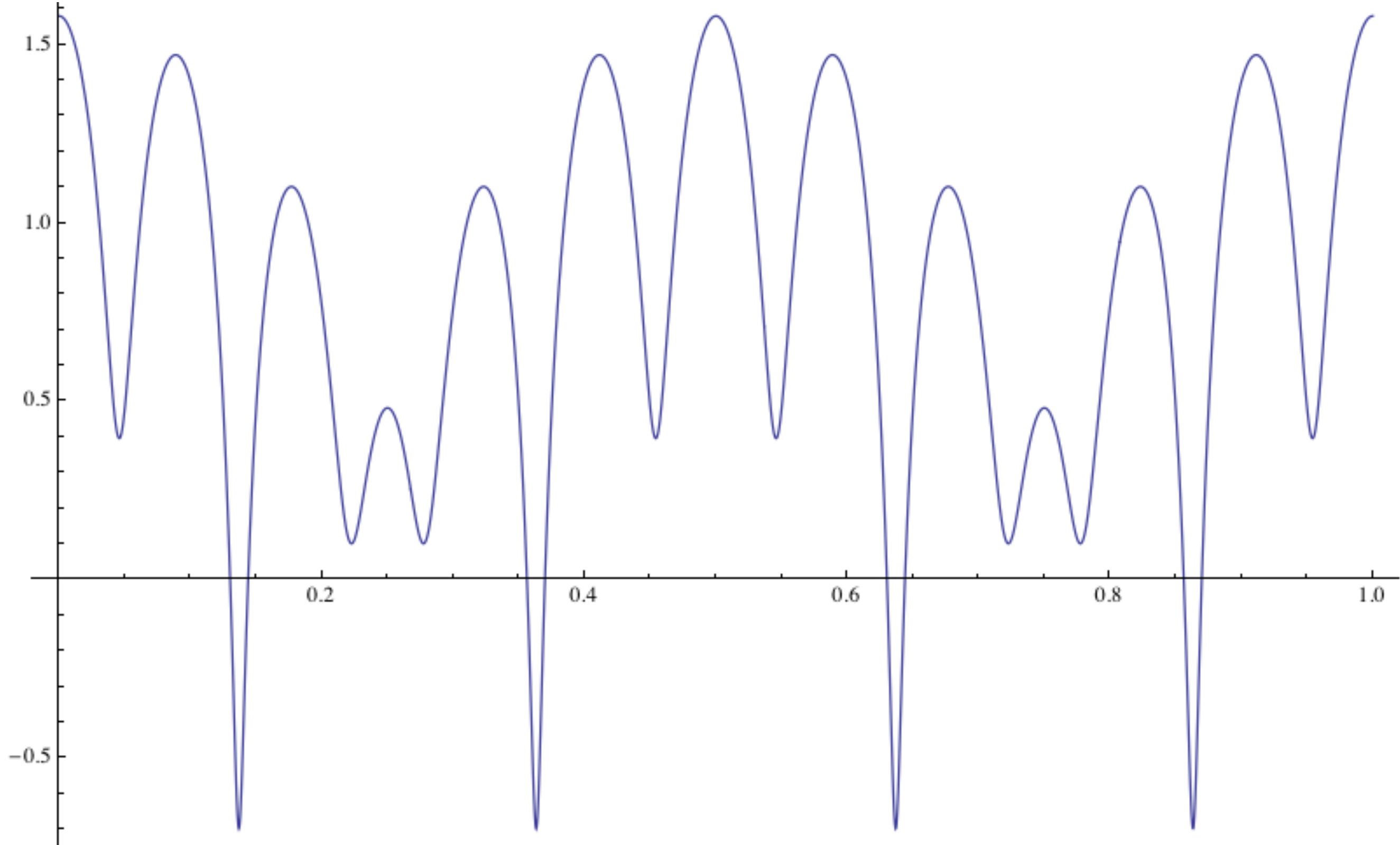}
     \caption{Graph of $\log|c(e^{2\pi i s})\tau|$  \label{fig:graph}}
  \end{figure}

  In the first case, the inequality \eqref{eqn:estimate} is satisfied if
  $n\epsilon_0>5\tau^{-n}$, which is true for all $n\ge8$.

  In the second case, using the lower bound on the absolute value of the
  derivative,
  \begin{align*}
     \bigl|\log|c(\om)\tau|\bigr|&=\bigl|\log|c(\om)|-\log|c(\zeta_k)|\bigr|\\
     &\ge|\om-\zeta_k|\ge \frac{1}{n}|\zeta_k^n-1|\ge
     \frac{1/n}{2^{18}\sqrt{40}}
     \Mahler(G)^{-n/2}.
  \end{align*}
  Since $\Mahler(G)^{1/2}<\tau$, the last term is eventually greater than
  $5\tau^{-n}$,
  and in fact this holds for all $n\ge 143$. One can then check by hand that
  that \eqref{eqn:estimate} holds for all $n<143$, completing our
  analysis of this example.
\end{example}

It is sometimes useful to make a change of variables to transform a
polynomial into a form that is easier to analyze (see for example
\cite[(3--6)]{LSW}). Let $\Del$ be a countable discrete group, and
$\aut(\Del)$ denote the group of automorphisms of $\Del$. If
$\Phi\in\aut(\Del)$, then $\Phi$ will act on various objects associated
with $\Del$. For example, if $f=\sum_{\del}f_\del \del\in\ZD$, then
$\Phi f=\sum_{\del} f_{\del}\Phi(\del)$, and so
$\Phi(fg)=\Phi(f)\Phi(g)$ and $\Phi(f^*)=\Phi(f)^*$. Analogous formulae
hold for $\TT^\Del$, $\lodr$, and $\lidr$. If $\al$ is an action of
$\Del$, there is a new $\Del$-action $\al^{\Phi}$ defined by
$(\al^\Phi)^\del=\al^{\Phi(\del)}$.

\begin{lemma}
   \label{lem:change-of-variables}
   Let $\Phi\in\aut(\Del)$ and $f\in\ZD$. Then $\Phi$ induces a
   continuous group isomorphism $\Phi\colon X_f\to X_{\Phi f}$
   intertwining the $\Del$-actions $\al_f^{}$ and $\af^\Phi$, so that
   $(X_f,\af)$ and $(X_{\Phi f},\af^\Phi)$ are topologically conjugate
   $\Del$-actions.
\end{lemma}

\begin{proof}
   If $t\in\TT^\Del$, then $t\in X_f$ iff $t\cdot f^*=0$ iff $0=\Phi(t
   \cdot f^*)=\Phi(t)\cdot \Phi(f)^*$ iff $\Phi(t)\in X_{\Phi f}$, so that
   $\Phi$ induces the required isomorphism. Since $\Phi(\del\cdot
   t)=\Phi(\del)\cdot t$, we see that $\Phi$ intertwines $\al_f^{}$ and ~$\af^\Phi$.
\end{proof}

\begin{remark}
   \label{rem:automorphism-effects}
   It is important to emphasize what Lemma \ref{lem:change-of-variables}
   does \textit{not} say. Although  $\af^\Phi$ and $\al_{\Phi f}$ are both
   $\Del$-actions on $X_{\Phi f}$, there is no obvious relation between
   them, even if $\Del$ is commutative. For example, if $\Del=\ZZ\cong
   \<u\>$, $f(u)=u^2-u-1$, and $\Phi(u)=u^{-1}$, then $(X_{\Phi
   f},\al_{\Phi f})$ is conjugate to the $\ZZ$-action of
   $A=\begin{bmatrix} 0&1\\1&-1\end{bmatrix}$ on $\TT^2$, while
   $(X_{\Phi f},\af^\Phi)$ is conjugate to the $\ZZ$ action of $A^{-1}$
   on $\TT^2$. But $A$ and $A^{-1}$ do not even have the same
   eigenvalues, so cannot give algebraically conjugate $\ZZ$-actions.

   However, certain dynamical properties are clearly shared between
   $\al_f$ and $\al_{\Phi f}$, for example ergodicity, mixing, and
   expansiveness.
\end{remark}

Automorphisms of $\G$ have an explicit description.

\begin{lemma}
   \label{lem:automorphism-heisenberg}
   Every automorphism of $\G$ is uniquely determined by integers
   $a$, $b$, $c$, $d$, $r$, and $s$ with $ad-bc=\pm1$, and
   is given by $\Phi(x)=x^a
   y^b z^r$, $\Phi(y)=x^cy^dz^s$, and $\Phi(z)=z^{ad-bc}$.
\end{lemma}

\begin{proof}
   Clearly $\Phi$ induces an automorphism of $\G/\Z\cong\ZZ^2$, hence
   $ad-bc=\pm1$. The condition $\Phi(yx)=\Phi(xyz)$ shows that
   $\Phi(z)=z^{ad-bc}$. These necessary conditions are easily checked to
   also be sufficient.
\end{proof}

An analysis of expansiveness for nonprincipal actions has been carried out
by Chung and Li \cite{ChungLi}. Let $F\in\ZG^{k\times k}$ be a square
matrix over $\ZG$. Then $\ZG^k/\ZG^k F$ is a left $\ZG$-module, whose
dual gives an algebraic action $\al_F$ on $X_F$. An argument similar to
the proof of (3) implies (1) in Theorem \ref{thm:expansive} shows
that if $F$ has an inverse in
$\logr^{k\times k}$, then $\al_F$ is expansive on $X_F$. The converse is
also true, and leads to a description of all expansive $\G$-actions.

\begin{theorem}[{\cite[Thm.\ 3.1]{ChungLi}}]
   Let $F\in\ZG^{k\times k}$ and $\al_F$ be the associated algebraic
   $\G$-action on $X_F$. Then $\al_F$ is expansive if and only if $F$ is
   invertible in $\logr^{k \times k}$. Moreover, every expansive
   $\G$-action is isomorphic to the restriction of such an $\al_F$ to a
   closed $\al_F$-invariant subgroup of ~$X_F$.
\end{theorem}

\begin{example}
   Let $F=\begin{bmatrix} 2 & x\\y &
      2\end{bmatrix}\in\ZG^{2\times2}$. By simply formally solving the
   equations for the inverse, one arrives at
   \begin{displaymath}
      F^{-1}=\begin{bmatrix} 2(4-xy)^{-1} & -x(4-yx)^{-1} \\
      -y(4-xy)^{-1}&2(4-yx)^{-1}\end{bmatrix} ,
  \end{displaymath}
  where the inverses appearing in the entries are all in $\logr$ by lopsidedness.
\end{example}

Obviously an algorithm for invertibility of square matrices would
immediately answer Problem \ref{prob:expansiveness}. But if the answer to
the latter is affirmative, would this provide an algorithm to decide
invertibility of square matrices?

\begin{problem}
   Suppose there is an algorithm which decides whether or not an element in
   $\ZG$ has an inverse in $\logr$. Is there then an algorithm that
   decides, given $F\in \ZG^{k\times k}$, whether or not $F$ has a
   inverse with entries in ~$\logr$?
\end{problem}

If $\al$ is an expansive algebraic $\Del$-action on $X$, and $Y$ is a
closed $\al$-invariant subgroup, then clearly the restriction of $\al$
to $Y$ is also expansive. However, whether the quotient action
$\al_{X/Y}$ on $X/Y$ is expansive is much more difficult. When
$\Del=\ZZ^d$, expansiveness of the quotient is always true (\cite[Thm.\
3.11]{SchmidtAutomorphisms}), but the proof uses commutative algebra and
is not dynamical. Chung and Li conjecture \cite{ChungLi} that for
nilpotent groups $\Del$, quotients of expansive actions are always
expansive. Even for the Heisenberg group this is not known.

\begin{problem}
   If $\al$ is an expansive action of $\G$ on a compact abelian group $X$,
   and if $Y$ is a closed $\al$-invariant subgroup of $X$, then must the
   quotient action of $\al$ on $X/Y$ be expansive?
\end{problem}

\section{Homoclinic points}\label{sec:homoclinic}
Let $\Del$ be a countable discrete group and $\al$ an algebraic
$\Del$-action on a compact abelian group $X$. An element $t\in X$ is
called \textit{homoclinic for $\al$}, or simply \textit{homoclinic}, if
$\al^{\del}(t)\to 0_X$ as $\del\to\infty$. The set of homoclinic points
forms a subgroup of $X$ called the \textit{homoclinic group of $\al$}. The
established notation in the literature for this group is
$\hc_{\al}(X)$. It will always be clear from context (and a slight font
change) what $\Del$ (or $\hc$) refers to.

Homoclinic points are an important technical device for localizing the
behavior of points in the group. For example, they are used to construct
periodic points, to prove a strong orbit tracing property called
specification, and to estimate entropy. They are also a natural starting
point for constructing symbolic covers of algebraic actions.

For $\Del=\ZZ^d$, many properties of  homoclinic groups were studied
in detail in \cite{LS}, especially for principal actions. Let us briefly
describe some of the main results there, with a view to extensions to
$\G$.

For $f\in\ZZ \ZZ^d=\ZZ[u_1^{\pm},\dots,u_d^{\pm}]$, define the
\textit{complex variety} of $f$ to be
\begin{displaymath}
   \V(f)=\{(z_1,\dots,z_d)\in(\CC^\times)^d:f(z_1,\dots,z_d)=0\},
\end{displaymath}
where $\CC^\times=\CC\setminus\{0\}$, and the \textit{unitary variety} of
$f$ to be
\begin{displaymath}
   \U(f)=\{(z_1,\dots,z_d)\in\V(f):|z_1|=\dots=|z_d|=1\}.
\end{displaymath}
Then by Theorem \ref{thm:expansive} and Wiener's theorem, $\af$ is
expansive if and only if $\U(f)=\emptyset$. In this case, let
$\wtri=(f^*)^{-1}\in\ell^1(\ZZ^d)$. As before, let  $\beta\colon
\lidr\to\TT^\Del$ be given by $(\beta w)_{\del}= w_{\del} \pmod{1}$,
which clearly commutes with the left $\Del$-actions.
Put $\ttri=\beta(\wtri)$. Since
$\rf(\ttri)=\beta(\rf(\wtri))=\beta(w\cdot f^*)=\beta(1)=0$, we see that
$\ttri\in X_f$, and is also homoclinic. Furthermore, $\ttri$ is
\textit{fundamental}, in the sense that every homoclinic point is a
finite integral combination of translates of $\ttri$ (cf.\ \cite[Lemma
4.5]{LS}). In this case all homoclinic points decay rapidly enough to
have summable coordinates.

In order to describe homoclinic points of principal $\Del$-actions
$\af$, we first ``linearize'' $X_f$ as follows. Put
\begin{displaymath}
   W_f=\beta^{-1}(X_f)=\{w\in\lidr\colon\rf(w)\in\lidz\}.
\end{displaymath}

Suppose now that $f\in\ZD$ is expansive, and define
$\wtri=(f^*)^{-1}\in\lodr$. Then $\rf$ is invertible on $\lidr$, and
$W^{}_f=\rf^{-1}\bigl(\lidz)\bigr)$, where $\rf^{-1}(u)=u\cdot\wtri$ for
every $u\in\lidz$.

\begin{proposition}[{\cite[Props.\ 4.2, 4.3]{DS}}]
   \label{prop:cover}
   Let $\Del$ be a countable discrete group, and $f\in\ZD$ be expansive,
   so that $f$ is invertible in $\lodr$. Put $\wtri=(f^*)^{-1}$ and let
   $\pi\colon \lidz\to X_f$ be defined as $\pi(u)=\beta(u\cdot\wtri)$,
   where $\beta$ is reduction of coordinates $\pmod1$. Then:
   \begin{enumerate}
     \item $\pi\colon\lidz\to X_f$ is surjective, and in fact the
      restriction of $\pi$ to the set of those $u$ with
      $\|u\|_\infty\le\|f\|_1$ is also surjective;
     \item $\ker \pi=\rf\bigl(\lidz\bigr)$;
     \item $\pi$  commutes with the relevant left $\Del$-actions; and
     \item $\pi$ is continuous in the weak* topology on closed, bounded
      subsets of $\lidz$.
   \end{enumerate}
\end{proposition}

\begin{proof}
   Suppose that $t\in X_f$. There is a unique lift $\ttil\in\lidr$ with
   $\beta(\ttil)=t$ and $\ttil_{\del}\in[0,1)$ for all
   $\del\in\Del$. Then
   $\beta\bigl(\rf(\ttil)\bigr)=\rf\bigl(\beta(\ttil)\bigr)=\rf(t)=0$ in
   $X_f$, hence $\rf(\ttil)\in\lidz$, and in fact
   $\|\rf(\ttil)\|_\infty\le\|f\|_1$. Furthermore
   \begin{displaymath}
      \pi\bigl(\rf(\ttil)\bigr)=\beta\bigl(\ttil\cdot f^*\cdot \wtri)
      =\beta\bigl(\ttil\cdot f^*\cdot (f^*)^{-1}\bigr)=t.
   \end{displaymath}
   This proves (1), and the remaining parts are routine verifications.
\end{proof}

If $f\in\ZD$ is expansive,  let $\ttri=\beta(\wtri)$ and call $\ttri$
the \textit{fundamental homoclinic point} of $\af$. This name is
justified by the following.

\begin{proposition}
   \label{prop:fundamental-homoclinic-point}
   Let $\Del$ be a countable discrete group, and $f\in\ZD$ be
   expansive. Put $\wtri=(f^*)^{-1}\in\lodr$ and $\ttri=\beta(\wtri)\in
   \hc_{\af}(X_f)$. Then every element of $\hc_{\af}(X_f)$ is a finite
   integral combination of left translates of $\ttri$.
\end{proposition}

\begin{proof}
   Suppose that $t\in \hc_{\af}(X_f)$, and lift $t$ to $\ttil\in \lidr$
   as in the proof of Proposition \ref{prop:cover}. Then
   $\rf(\ttil)\in\lidz$, and since $\ttil_{\del}\to0$ as
   $\del\to\infty$, the coordinates of $\rf(\ttil)$ must vanish outside
   of a finite subset of $\Del$, i.e., $\rf(\ttil)=g\in\ZD$. Then
   $t=\pi\bigl(\rf(\ttil)\bigr)=\pi(g)=\beta(g\cdot\wtri)=g\cdot\ttri$
   has the required form.
\end{proof}

Next we show that expansive principal actions have a very useful orbit
tracing property called specification.

\begin{proposition}[{\cite[Prop.\ 4.4]{DS}}]
   \label{prop:specification}
   Let $\Del$ be a countable discrete group, and $f\in\ZD$ be
   expansive. Then for every $\epsilon>0$ there is a finite subset
   $K_\epsilon$ of $\Del$ such that if $F_1$ and $F_2$ are arbitrary
   subsets of $\Del$ with $K_\epsilon F_1\cap K_\epsilon F_2=\emptyset$
   and if $t^{(1)}$ and $t^{(2)}$ are arbitrary points in $X_f$, then we
   can find $t\in X_f$ such that
   $d_{\TT}(t^{}_{\del},t_{\del}^{(i)})<\epsilon$ for every $\del\in
   F_i$, for $i=1,2$.
\end{proposition}

\begin{proof}[Sketch of proof]
   Let $\epsilon>0$. The set $K_\epsilon$ is chosen so that
   $\sum_{\del\notin K_\epsilon} |\wtri_{\del}|<\epsilon/\|f\|_1$. Lift
   each $t^{(i)}$ to $\ttil^{(i)}$, and then truncate each
   $\rf(\ttil^{(i)})$ to a $u^{(i)}$ having support in $K_\epsilon F_i$. It is then
   easy to verify that $t=\pi\bigl(\rf(u^{(1)})+\rf(u^{(2)})\bigr)$ as
   the required properties (see \cite[Prop.\ 4.4]{DS} for details).
\end{proof}

A point $t\in X_f$ with $\sum_{\del\in\Del}d_{\TT}(t_{\del},0)<\infty$
is called \textit{summable}. Let $\hc_{\af}^1(X_f)$ denote the group of
all summable homoclinic points for $\af$. Summability is crucial in
using homoclinic points for dynamical purposes.

\begin{example}
   Let
   $f(u_1,u_2)=3-u_1^{}-u_1^{-1}-u_2^{}-u_2^{-1}\in\ZZ[u_1^{\pm},u_2^{\pm}]$. It is
   shown in \cite[Example 7.3]{LS} that $\hc_{\af}(X_f)$ is uncountable
   (indeed, the Fourier series of every smooth density on $\U(f)$ decays
   to 0 at infinity, and so gives a homoclinic point), but
   $\hc_{\af}^1(X_f)=\{0\}$. Despite their large number, the nonsummable
   homoclinic points here are essentially useless.
\end{example}

Summable homoclinic points may still exist for nonexpansive actions. For example,
consider $f(u_1,u_2)=2-u_1-u_2$. The formal inverse $w$ of $f^*$ via
geometric series is well-defined and has coordinates decaying to 0 at
infinity, so that $\beta(w)$ is homoclinic, but the decay is so
slow that $w$ is not summable (see \cite[Example 7.2]{LS}). Define a
function $F\colon\TT^2\to\CC$ by putting $F(s_1,s_2)=f^*(e^{2\pi i s_1},
e^{2\pi i s_2})$. Then $1/F$ is integrable on $\TT^2$, and $w$ is just
the Fourier transform of $1/F$. Now $1/F$ has a singularity at $(0,0)$,
and we can try to cancel this by multiplying it by a sufficiently high
power $N$ of another polynomial $G(s_1,s_2)=g(e^{2\pi i s_1},
e^{2\pi i s_2})$ that also vanishes at $(0,0)$ so that $G^N/F$ has
absolutely convergent Fourier series, resulting in a summable homoclinic
point $g^N\cdot w=g^N/f^*$. For this example, taking $g(u_1,u_2)=u_1-1$,
a detailed analysis in \cite[\S5]{LSV1} shows that $N=3$ is the smallest
power such that $G^N/F$ has absolutely convergent Fourier series,
providing a summable homoclinic point for ~$\af$.

This ``multiplier method'' can be generalized to all
$f\in\ZZ[u_1^{\pm},\dots, u_d^{\pm}]$ provided that the dimension of
$\U(f)\subset\SS^d$ is at most $d-2$. More precisely, with this
condition, there is another polynomial $g\in \ZZ[u_1^{\pm},\dots,
u_d^{\pm}]$, not a multiple of $f$, such that $\U(f)\subset\U(g)$. The
corresponding quotient $G^N/F$ has absolutely convergent Fourier series
for sufficiently large $N$
\cite{LSV2}, and hence $\af$ has summable homoclinic points. However, if
$\text{dim\ }\U(f)=d-1$, this method fails, and in fact there are no
nonzero summable homoclinic points \cite[Thm.\ 3.2]{LSV2}.

Let us turn to considering homoclinic points for principal actions of
$\G$. If $f\in\ZG$ is expansive, we have already seen how to describe
$\hc_{\af}(X_f)$, and that this agrees with $\hc_{\af}^1(X_f)$.

Consider $f(x,y,z)=2-x-y\in\ZG$. If $\omega=e^{2\pi i/3}$, then $f$ is in the
kernel of the algebra homomorphism $\logc\to\CC$ given by
$x\mapsto\omega$, $y\mapsto \omega^2$, and $z\mapsto1$. Hence $f$ is not
expansive.

However, it is shown in \cite{GollVerbitskiy} that the formal
inverse of $f$ can be smoothed by using the multiplier $(z-1)^2$ to
create a summable homoclinic point for $\af$. The proof uses highly
nontrivial combinatorial arguments, starting with the noncommutative
expansion
\begin{displaymath}
   (x+y)^n=\sum_{k=0}^n \begin{bmatrix}n\\k\end{bmatrix}_z x^ky^{n-k},
   \text{\quad where \quad}\begin{bmatrix}n\\k\end{bmatrix}_z=\prod_{j=0}^k
   \frac{z^{n-j}-1}{z^{j+1}-1} .
\end{displaymath}
If $f(x,y,z)=4-x-x^{-1}-y-y^{-1}$ the authors of
\cite{GollSchmidtVerbitskiy2} state that they shown that the multiplier
$(z-1)^{12}$ results in a summable homoclinic point, and conjecture that
$(z-1)^2$ actually suffices.

For more complicated nonexpansive polynomials in $\ZG$, it is not at all
clear what could substitute for the dimension condition on the unitary
variety in the commutative case.

\begin{problem}
   For $f\in\ZG$, determine explicitly both $\hc_{\af}(X_f)$ and
   $\hc_{\af}^1(X_f)$.
\end{problem}

Anticipating the entropy material from the next section, we remark that
Chung and Li \cite[Thm.\ 1.1]{ChungLi}, generalizing earlier work for
$\ZZ^d$ in \cite{LS}, showed that $\hc_{\af}(X_f)\ne\{0\}$ if and only if
$\af$ has positive entropy, and that $\hc_{\af}(X_f)$ is dense in $X_f$
if and only if $\af$ has completely positive entropy.

For expansive $\Del$-actions $\af$, Proposition \ref{prop:cover}(1)
gives a continuous, equivariant, and surjective map $\pi$ from the full $\Del$-shift with symbols
$\{-\|f\|_1,-\|f\|_1+1,\dots,\|f\|_1\}$ to $X_f$, which allows us to view this shift space as a symbolic cover of $X_f$. In 1992 Anatoly
Vershik showed \cite{Vershik} that for certain hyperbolic toral
automorphisms, this symbolic cover could be pruned to a shift of finite
type for which the covering map $\pi$ is one-to-one almost
everywhere. This provided an arithmetic approach to constructing Markov
partitions, which were originally found geometrically by Adler and
Weiss, and are one of the main motivations for symbolic dynamics. Vershik's
arithmetic construction was further investigated in
\cite{SidorovVershik}, \cite{Sidorov},
\cite{KenyonVershik}, and \cite{SchmidtBeta}. Einsiedler and the second author \cite{EinsiedlerSchmidt}
considered the problem of extending this idea to obtain symbolic
representations of algebraic $\ZZ^d$-actions, and gave an example of an
algebraic $\ZZ^2$-action for which the symbolic cover could be pruned to
a shift of finite type to obtain a map that is one-to-one almost
everywhere, but the proof involved a complicated percolation argument.
Even for Heisenberg actions virtually nothing is known about the
existence of good symbolic covers.

\begin{problem}
   Find general sufficient conditions on an expansive $f\in\ZG$ so that
   the symbolic cover from Proposition \ref{prop:cover} (1) can be pruned
   to one that is (a) of finite type or at least sofic, (b) of equal entropy, or (c) one-to-one almost everywhere.
\end{problem}

\section{Entropy of Algebraic Actions}\label{sec:entropy-actions}

We give here several equivalent definitions of the (topological) entropy
of an algebraic action, sketch some background material on von Neumann
algebras, and then describe recent results relating entropy to
Fuglede-Kadison determinants.

Let $\Del$ be a countable discrete group. For finite subsets
$F,K\subset\Del$, define $FK=\{\del\theta:\del\in F, \theta\in K\}$. For
what follows to make sense, we require that $\Del$ be \textit{amenable},
namely that there is a sequence $\{F_n:n\ge1\}$ of finite subsets of
$\Del$ such that for every finite subset $K$ of $\Del$ we have that
\begin{displaymath}
   \frac{|F_n\setdiff F_nK|}{|F_n|} \to 0\text{\quad as $n\to\infty$}.
\end{displaymath}
Such a sequence is called a \textit{right-F{\o}lner sequence}.

Suppose that $\al$ is an algebraic $\Del$-action on a compact abelian
group $X$. We assume there is a translation-invariant metric $d$ on $X$,
and let $\mu$ denote normalized Haar measure on ~$X$.

As before, we abbreviate the (left) $\al$-action of $\Del$ on $X$ by
using $\del\cdot t$ for $\al^{\del}(t)$. Then $\Del$ acts on subsets
$E\subset X$ by $\del\cdot E=\{\del\cdot t:t\in E\}$. Although this
differs from the traditional action of transformations on subsets using
inverse images, this seems better suited to our purposes, since all
$\al^{\del}$ are invertible, and its use is consistent with the action
of $\Del$ on functions on $X$: if $\chi_E$ is the indicator function of
$E$, then $\del\cdot \chi_E=\chi_E\circ \del^{-1}=\chi_{\del\cdot E}$.

To define topological entropy, we consider open covers $\SU$ of $X$. If
$\SU_1,\dots,\SU_n$ are open covers, define their \textit{span} as
$\bigvee_{j=1}^n\SU_j=\{U_1\cap\dots\cap U_n:U_j\in\SU_j \text{ for
$1\le j\le n$}\}$. If $\SU$ is an open cover and $F$ is a finite subset
of $\Del$, let $\SU^F=\bigvee_{\del\in F}\del\cdot\SU$, where
$\del\cdot\SU=\{\del \cdot U:U\in\SU\}$. For an open cover $\SU$ let
$\Num(\SU)$ denote the cardinality of the open subcover with fewest
elements, which is finite by compactness. It is easy to check that
$\Num(\SU\vee\mathscr{V}) \le \Num(\SU)\Num(\mathscr{V})$. We define the
\textit{open cover entropy} of $\al$ to be
\begin{equation}
   \label{eqn:open-cover}
   \h_{\text{cov}}(\al)=\sup_{\SU}\,\limsup_{n\to\infty}\frac{1}{|F_n|}\log \Num(\SU^{F_n})
\end{equation}
where the supremum is taken over all open covers of $X$.

Recall the elementary fact that if $\{a_n:n\ge1\}$ is a sequence of
nonnegative reals with $a_{m+n}\le a_m+a_n$, then $a_n/n$ converges to a
limit as $n\to\infty$, and this limit equals $\inf_{1\le
n<\infty}a_n/n$. Hence for $\Del=\ZZ$ and $F_n=\{0,1,\dots,n-1\}$, it
follows that the $\limsup$ in \eqref{eqn:open-cover} is actually a
limit. There is a general version of this argument valid for amenable
groups, due to Lindenstrauss and Weiss.

\begin{proposition}[{\cite[Thm.\ 6.1]{LindenstraussWeiss}}]
   \label{prop:lindenstrauss-weiss}
   Suppose that $\phi(F)$ is a real-valued function defined for all
   nonempty finite subsets $F$ of $\Del$, satisfying:
   \begin{enumerate}
     \item  $0\le\phi(F)<\infty$,
     \item  if $F'\subseteq F$, then $\phi(F')\le\phi(F)$,
     \item  $\phi(\del  F)=\phi(F)$ for all $\del\in \Del$, and
     \item  $\phi(F\cup F')\le \phi(F)+\phi(F')$ if $F\cap F'=\varnothing $.
   \end{enumerate}
   Then for every right-F{\o}lner sequence $\{F_n\}$ the numbers
   $\phi(F_n)/|F_n|$ converge to a finite limit, and this limit is
   independent of the choice of right-F{\o}lner sequence.
\end{proposition}

Roughly speaking, this fact is proved by showing that if $K$ is a large
finite subset of $\Del$ and $\epsilon$ is small, then any $F$ with
$|F\setdiff FK|/|F|<\epsilon$ can be almost exactly tiled by left
translates of $F_n$s of various sizes. Then subadditivity and
translation-invariance of $\phi$ show that $\phi(F)/|F|$ is bounded
above, within a small error, by $\inf_n \phi(F_n)/|F_n|$.

Fix an open cover $\SU$ of $X$, and put $\phi(F)=\log \Num(\SU^F)$ for
every nonempty finite subset $F$ of $\Del$. Since each $\al^\del$ is a
homeomorphism of $X$, it follows that
\begin{displaymath}
   \phi(\del F)=\log \Num\Bigl(\bigvee_{\theta\in F}(\del\theta)\cdot
   \SU\Bigr) =\log \Num\Bigl(\del\cdot \bigl(\bigvee_{\theta\in
   F}\theta\cdot \SU\bigr)\Bigr)=\log \Num\Bigl(\bigvee_{\theta\in
   F}\theta\cdot\SU\Bigr) =\phi(F).
\end{displaymath}
Conditions (1), (2), and (4) in Proposition
\ref{prop:lindenstrauss-weiss} are trivially satisfied for this
$\phi$. Hence for every open cover $\SU$, the $\limsup$ in
\eqref{eqn:open-cover} is a limit, and this limit does not depend on the
choice of right F{\o}lner sequence $\{F_n\}$.

The open cover definition of topological entropy is due to Adler,
Konheim, and McAndrew \cite{AKM}. Bowen \cite{Bowen} introduced
equivalent definitions that are better suited for many purposes, which
we now describe.

If $F$ is a finite subset of $\Del$ and $\epsilon>0$, a subset $E\subset
X$ is called \textit{$(F,\epsilon)$-spanning} if for every $t \in X$
there is a $u\in E$ such that $d(\del^{-1}\cdot t,\del^{-1}\cdot
u)<\epsilon$ for every $\del\in F$. Dually, a set $E\subset X$ is called
\textit{$(F, \epsilon)$-separated} if for distinct elements $t,u\in E$
there is a $\del\in F$ such that $d(\del^{-1}\cdot t,\del^{-1}\cdot
u)\ge \epsilon$. Let $r_F(\epsilon)$ denote the smallest cardinality of
any $(F,\epsilon)$-spanning set, and $s_F(F,\epsilon)$ be largest
cardinality of any $(F,\epsilon)$-separated set. Put
\begin{displaymath}
   \h_{\text{span}}(\al)=\lim_{\epsilon\to0}\,\limsup_{n\to\infty}
   \frac{1}{|F_n|} \,\log r_{F_n}(\epsilon),\text{\ and\ }
    \h_{\text{sep}}(\al)=\lim_{\epsilon\to0}\,\limsup_{n\to\infty}
   \frac{1}{|F_n|} \,\log s_{F_n}(\epsilon).
\end{displaymath}

If $\SU_{\epsilon}$ denotes the open cover of $X$ by $\epsilon$-balls,
the elementary inequalities
\begin{displaymath}
   \Num(\SU_\epsilon^F)\le r_F(\epsilon)\le s_F(\epsilon)\le \
   \Num(\SU_{\epsilon/2}^F)
\end{displaymath}
show that
$\h_{\text{cov}}(\al)=\h_{\text{span}}(\al)=\h_{\text{sep}}(\al)$, and so
all three are independent of choice of right F{\o}lner sequence.

One more variant of the entropy definition, using volume decrease, is
also useful. Let $B_\epsilon=\{t\in X:d(t,0)<\epsilon\}$, and for finite
$F\subset\Del$ put $B_\epsilon^F=\bigcap_{\del\in F}\del\cdot
B_\epsilon$.
Define
\begin{displaymath}
   \h_{\text{vol}}(\al)=\lim_{\epsilon\to0}\,\limsup_{n\to\infty}
   \,\,      -\frac{1}{|F_n|} \log \mu\bigl(B_\epsilon^{F_n}\bigr).
\end{displaymath}
If $E$ is $(F,\epsilon)$-spanning, then $X=\bigcup_{t\in
E}(t+B_\epsilon^{F})$, so that $1\le |E|\mu(B_\epsilon^F)$, and hence
$\h_{\text{vol}}(\al)\le\h_{\text{span}}(\al)$. If $E$ is
$(F,\epsilon)$-separated, then the sets $\{t+B_{\epsilon/2}^F):t\in E\}$
are disjoint, so that $|E|\mu(B_{\epsilon/2}^F)\le1$, proving that
$\h_{\text{sep}}(\al)\le \h_{\text{vol}}(\al)$.
Thus all these notions of entropy coincide, and we let $\h(\al)$ denote
their common value. We remark that these are also equal to the measure
theoretic entropy of $\al$ with respect to Haar measure, but we will not
be using this fact. Deninger's paper \cite{D} has complete proofs of
these facts, and also that the $\limsup$'s in these definitions can be
replaced by $\liminf$'s without affecting the results.

If $\Phi\in\aut(\Del)$ and $\{F_n\}$ is a right F{\o}lner sequence, then
clearly so is $\{\Phi(F_n)\}$. It follows that $\h(\al_{\Phi
f})=\h(\af)$, i.e., entropy is invariant under a change of variables.

Suppose that $\al$ is an algebraic $\Del$-action on $X$, and that $Y$ is
a closed $\Del$-invariant subgroup of $X$. Let $\al_Y$ denote the
restriction of $\al$ to $Y$, and $\al_{X/Y}$ be the quotient action on
$X/Y$. An important property of entropy is that it adds over the exact
sequence $0\to Y\to X\to X/Y\to0$.

\begin{theorem}[The Addition Formula, {\cite[Cor.\ 6.3]{LiAutomorphisms}}]
   Let $\Del$ be an amenable group, let $\al$ be an algebraic
   $\Del$-action on $X$, and suppose that $Y$ is a closed, $\Del$-invariant subgroup
   of $X$. Then $\h(\al)=\h(\al_Y)+\h(\al_{X/Y})$.
\end{theorem}

The Addition Formula has a long history. The basic approach is to take a
Borel cross-section to the quotient map $X\to X/Y$, and regard $\al$ as
a skew product with base action $\al_{X/Y}$ and fiber actions that are
affine maps of $Y$ with the same automorphism part $\al_Y$ but with
different translations. The idea is then to show that the translation
parts of these affine maps, being isometries, do not affect entropy. The
case $\Del=\ZZ$ was proved by Bowen \cite{Bowen}, and the case
$\Del=\ZZ^d$ is handled in \cite[Appendix B]{LSW} using arguments due
originally to Thomas \cite{ThomasAddition}. Fiber entropy for amenable
actions was dealt with in  \cite{WardZhang}. There is a serious
difficulty in generalizing these ideas to noncommutative $\Del$, namely
the lack of a scaling argument used to eliminate a universal constant
due to overlaps of open sets in a cover. However, machinery developed by
Ollagnier
\cite{Ollagnier} handles this issue, and this was used by Li
to give the most general result cited above.

When $\Del=\ZZ^d$ there are explicit formulas for entropy. First
consider the case $\Del=\ZZ$. Without loss of generality, we can assume
that $f(u)\in\ZZ[u^{\pm}]$ has the form $f(u)=c_nu^n+\cdots c_1u+c_0$ with
$c_nc_0\ne0$. Factor $f(u)$ over $\CC$ as
$f(u)=c_n\prod_{j=1}^n(u-\lam_j)$. Then Yuzvinskii
\cite{Yuzvinskii1,Yuzvinskii2}
showed that
\begin{equation}
   \label{eqn:one-variable-mahler}
   \h(\al_f)=\log|c_n|+\sum_{j=1}^n \log^+|\lam_j|=\mahler(f).
\end{equation}
An interpretation of \eqref{eqn:one-variable-mahler} from
\cite{LindWard} shows that term $\sum_{j=1}^n\log^+|\lam_j|$ is due to
geometric expansion, while the term $\log|c_n|$ is due to $p$-adic
expansions for those primes $p$ dividing $c_n$, an adelic viewpoint that
has been useful in other contexts as well.

Mahler measure is  defined for polynomials
$f\in\ZZ[u_1^{\pm},\dots,u_d^{\pm1}]=R_d$ by the formulas
\begin{displaymath}
   \mahler(f)=\int_{\SS^d}\log|f|=\int_0^1\dots\int_0^1
   \log|f(e^{2\pi i s_1},\dots,e^{2\pi i s_d})|\,ds_1\dots ds_d,
\end{displaymath}
and $\Mahler(f)=\exp\bigl(\mahler(f)\bigr)$ \cite{Mahler2}. One of the
main results in \cite{LSW} is that for nonzero
$f\in R_d$ we have that
$\h(\al_f)=\mahler(f)$. With this information, entropy for arbitrary
algebraic $\ZZ^d$-actions can be easily found. Let $\mathfrak{a}$ be an
ideal in $R_d$ that is not principal. A simple argument \cite[Thm.\
4.2]{LSW} shows that $\h(\al_{R_d/\mathfrak{a}})=0$. Any finitely
generated $R_d$-module $M$ has a prime filtration $0=M_0\subset
M_1\subset \dots\subset M_r$ with $M_j/M_{j-1}\cong R_d/\mathfrak{p}_j$,
where the $\mathfrak{p}_j$ are prime ideals in $R_d$. Then the Addition
Formula shows that
$\h(\al_M)=\h(\al_{R_d/\mathfrak{p}_1})+\dots+\h(\al_{R_d/\mathfrak{p}_r})$,
and each summand can be computed from our preceding remarks.

To conclude this discussion of the $\Del=\ZZ^d$ case, we point out that
there is a complete characterization for which principal actions have
zero entropy. Recall the definition of generalized cyclotomic polynomial
from \S\ref{sec:mixing}.

\begin{proposition}[\cite{Boyd,Smyth}]
   \label{prop:zero-entropy}
   Let $f\in\ZZ[u_1^{\pm},\dots,u_d^{\pm1}]$. Then $\h(\al_f)=0$ if
   and only if $f$ is a product of generalized cyclotomic polynomials
   times a monomial or its negative.
\end{proposition}

This result was originally proved by Boyd \cite{Boyd} using deep results
of Schinzel, but was later given a simpler and more geometric proof by
Smyth \cite{Smyth}.

Turning now to noncommutative $\Del$, we sketch some background material
on related von Neumann algebras. The functional analysis used can be
found, for example, in \cite[Chaps.\ VII, VIII]{Conway}. Let
\begin{displaymath}
   \ltdc=\Bigl\{w=\sum_{\del\in\Del}w_\del \del: w_\del\in\CC
   \text{ and } \|w\|_2^2:=\sum_{\del\in\Del}|w_\del|^2<\infty\Bigr\},
\end{displaymath}
which is a complex Hilbert space with the standard inner product
$\bigl\langle \sum_{\del}w_\del \del,\sum_{\del}v_\del
\del\bigr\rangle \linebreak[0]=\sum_{\del}w_\del\overline{v_\del}$. As in the case of
$\lodc$, there is a left-action of $\Del$ on $\ltdc$ given by
$\theta\cdot\bigl(\sum_{\del}w_\del
\del\bigr)=\sum_{\del}w_\del\,\theta\del$.

The \textit{group von Neumann algebra} $\nd$ of $\Del$ consists of all
those bounded linear operators $T\in\SB\bigl(\ltdc\bigr)$ commuting with the
left $\Del$-action, i.e., $T(\del\cdot w)=\del\cdot T(w)$. There is a
natural inclusion of $\CC\Del$ into $\nd$ given by $f\to\rf$, where as
before $\rf(w)=w\cdot f^*$. In addition, there is a faithful normalized trace
function $\tnd\colon\nd\to\CC$ given by
$\tnd(T)=\<T(1_\Del),1_\Del\>$. This means that $\tnd$ is linear,
$\tnd(1)=1$, $\tnd(TT^*)>0$ for every $T\ne0$, and $\tnd(ST)=\tnd(TS)$.

Using this trace, Fuglede and Kadison \cite{FugledeKadison} defined a
determinant function on $\nd$ as follows. Let $T\in\nd$. Then $TT^*\ge0$,
so the spectral measure $\nu$ of the self-adjoint operator $TT^*$ is
supported on $[0,\infty)$, in fact on $[0,\|TT^*\|]$. Using the
functional calculus for $\SB\bigl(\ltdc\bigr)$, we can form
the operator $\log(TT^*)=\int_{0^+}^\infty \log t \,d\nu(t)$, where the
lower limit $0^+$ indicates that we ignore any point mass at
$0$ that $\nu$ may have. We then define
\begin{displaymath}
   \dnd T=\exp\Bigl[\frac12 \tnd\bigl(\log(TT^*)\bigr)\Bigr].
\end{displaymath}
This Fuglede-Kadison determinant has the following very useful
properties (see \cite[\S3.2]{Luck} for details):
\begin{align*}
     & \text{$\bullet$\quad  $\dnd T^*=\dnd T$,}\\
     & \text{$\bullet$\quad  if $T>0$ in $\nd$, then $\dnd T=\exp(\tr_{\nd} \log T)$,}\\
     & \text{$\bullet$\quad  if $0\le S\le T$ in $\nd$, then $\dnd S\le\dnd T$},\\
     & \text{$\bullet$\quad  $\dnd ST=(\dnd S)(\dnd T)$}.
\end{align*}
We remark that multiplicativity of $\dnd$ is not obvious, essentially being
a consequence of the Campbell-Baker-Hausdorff formula and vanishing of
trace on commutators, although for technical reasons a complex variables
approach is more efficient.

\begin{example}
   Let $\Del=\ZZ^d$, and write elements of $\ZZ^d$ as
   $\bn=(n_1,\dots,n_d)$. The Fourier transform identifies
   $\ell^2(\ZZ^d,\CC)$ with $L^2(\TT^d,\CC)$, with $\bn$ being identified with
   the function $\chi_{\bn}$, where
   $\chi_{\bn}(\mathbf{s})=\exp\bigl[2\pi i(\bn\cdot\mathbf{s})\bigr]$
   for $\mathbf{s}\in\TT^d$. Any bounded linear operator $T$ on
   $\ell^2(\ZZ^d,\CC)$ commuting with the $\ZZ^d$-action must have the
   form of convolution with some element $v\in\ell^2(\ZZ^d,\CC)$. Hence
   the Fourier transform $V$ of $v$ must give a bounded linear operator on
   $L^2(\TT^d)$ via pointwise multiplication, and this forces
   $V\in L^\infty(\TT^d,\CC)$. Conversely, every
   $V\in L^\infty(\TT^d,\CC)$ corresponds to an element of
   $\SN\ZZ^d$. This identifies $\SN\ZZ^d$ with
   $L^\infty(\TT^d,\CC)$. Under this identification,
   \begin{displaymath}
      \tr_{\SN\ZZ^d}(V)=\int_{\TT^d} V(\mathbf{s})\,d\mathbf{s}
      \text{\quad and\quad } \text{det}_{\SN\ZZ^d}(V)=\exp\Bigl[\int_{\TT^d}
      \log|V(\mathbf{s})|\,d\mathbf{s}\Bigr].
   \end{displaymath}
   For $f\in\ZZ\ZZ^d$, we observed that $\rf\in\SN\ZZ^d$, and this
   corresponds to multiplication by $f(e^{-2\pi i s_1},\dots,e^{-2\pi i
   s_d})$. Hence in this case $\text{det}_{\SN\ZZ^d} f=\Mahler(f)$ and
   so $\log \text{det}_{\SN\ZZ^d} f=\mahler(f)=\h(\al_f)$.
\end{example}

Indeed, it was the equality of Mahler measure, entropy, and the
Fuglede-Kadison determinant in L\"uck's book \cite[Example 3.13]{Luck}
that originally inspired Deninger to investigate whether this phenomenon
extended to more general groups. He was able to show in \cite{D} that with
some conditions it did. Further work \cite{DS}, \cite{LiAutomorphisms}
extended the generality, culminating in the comprehensive
results of Li and Thom
\cite{LiThom}.

To describe their work, recall that $\rf$ is a bounded linear operator on
$\ltdc$. More generally, if $F\in\ZD^{k\times l}$, then there is a
bounded linear operator $\rho_F\colon\ltdc^k\linebreak[0]\to\ltdc^l$ given by right
multiplication by $F^*$, where $(F^*)_{i,j}=(F_{i,j})^*$. There is an
extension of $\dnd$ to such $F$ (see \cite[\S2.1]{LiThom}) for details.

\begin{theorem}[{\cite[Thm.\ 1.2]{LiThom}}]
   \label{thm:li-thom}
   Let $\Del$ be a countable discrete amenable group, and let
   $f\in\ZD$. Suppose that $\rf\colon\ltdc\to\ltdc$ is injective. Then
   $\h(\af)=\log \dnd f$. More generally, if $F\in\ZD^{k\times l}$ and
   if $\rho_F\colon\ltdc^k\to\ltdc^l$ is injective, then $\h(\al_F)\le
   \log \dnd F$. If $k=l$, then $\h(\al_F^{})=\h(\al_{F^*})=\log \dnd F$.
\end{theorem}

In particular, $\h(\al_{f^{}})=\h(\al_{f^*})$. This is a highly
nontrivial fact, since there is no direct connection between
$\al_{f^{}}$ and $\al_{f^*}$.

The computation, or even estimation, of the values of Fuglede-Kadison determinants is not
easy. In the next two sections we will explicitly calculate entropy for
certain principal actions of the Heisenberg group.

As pointed out by Deninger \cite{D}, there are examples of lopsided polynomials $f$
for which $\dnd f$ can be computed by a rapidly converging series.

\begin{example}
   \label{exam:traces}
   Let $f(x,y,z)=5-x-x^{-1}-y-y^{-1}\in\ZG$. Write $f=5(1-g)$, where
   $g=\frac15(x+x^{-1}+y+y^{-1})$. It is easy to see \cite[Lemma
   2.7]{BVZ} that the spectrum of~ $g$, considered as an element in
   $\SN\G$, is contained in $[-4/5,4/5]$. Hence we can apply the
   functional calculus to compute $\log f$ via the power series for
   $\log(1-t)$, yielding
   \begin{displaymath}
      \log f = \log 5+\log(1-g)=\log5-\sum_{n=1}^\infty\frac{g^n}{n}.
   \end{displaymath}
   Now $\tr_{\SN\G}(g^n)$ is the value of the constant term of $g^n$. If
   $r_{\G}^{}(n)$ denotes the number of words over
   $S=\{x,x^{-1},y,y^{-1}\}$ of length $n$ whose product is 1, then
   clearly $\tr_{\SN\G}(g^n)=5^{-n}r_{\G}^{}(n)$. In any word over $S$ with
   product 1, the number of $x$ and $x^{-1}$ must be equal and similarly
   with $y$ and $y^{-1}$, so that $r_{\G}^{}(n)=0$ for $n$ odd. The numbers
   $r_{\G}^{}(2n)$ grow rapidly:
   \begin{displaymath}
      r_{\G}^{}(2n)=4^{2n}\Bigl( \frac{1}{2n^2}+O\Bigl(\frac{1}{n^3}\Bigr)\Bigl),
   \end{displaymath}
   (see \cite{Gretete}, although the result stated there is off by a
   factor of 2).
   We thank David Wilson for providing us a short \textit{Mathematica}
   program that computes $r_{\G}^{}(2n)$ up to $r_{\G}^{}(60)$,
   which is a 33 digit number. Hence
   \begin{displaymath}
      \tr_{\SN\G} \log f = \log 5 -
      \sum_{n=1}^{30}\frac{r_{\G}^{}(2n)}{(2n)5^{2n}} -
      \sum_{n=31}^{\infty}\frac{r_{\G}^{}(2n)}{(2n)5^{2n}}.
   \end{displaymath}
   Using the trivial estimate $r_{\G}^{}(2n)\le 4^{2n}$, the last sum has value
   less than $10^{-7}$. The remaining part therefore gives the value for
   $\h(\af)=\tr_{\SN\G} \log f\cong 1.514708$, correct to six decimal
   places.

   We remark that $5g=x+x^{-1}+y+y^{-1}$, considered as an operator in
   $\SN\G$, has been studied intensively, and is closely related to
   Kac's famous Ten Martini Problem. Indeed, the image of $5g$ in the
   rotation algebra factor $\mathscr{A}_{\theta}$ of $\SN\G$ is
   Harper's operator ~$H_\theta$ (cf. \cite{BVZ}). Kac conjectured that for every
   irrational $\theta$ the spectrum of $H_\theta$ is a Cantor set of
   zero Lebesgue measure. This conjecture was recently confirmed by deep
   work of Avila.
\end{example}

\begin{remark}
   The calculations in the previous example can be carried out just as
   well in any group $\Del$ containing two elements $x$ and $y$. The
   only change is that the number $r_{\Del}(n)$ of words over
   $S=\{x,x^{-1},y,y^{-1}\}$
   whose product is 1 will be different, depending on $\Del$. For
   comparison with the example, we work this out for $\Del=\ZZ^2$ with
   commuting generators $x, y$, and for $\Del=\F_2$, the free group with
   generators $x, y$.

   Recalling that $f=5(1-g)$, let us define
   \begin{displaymath}
      L(f,\Del):= \log 5 - \sum_{n=1}^\infty \frac{\tr
      g^n}{n}=\log 5-\sum_{n=1}^\infty \frac{r_{\Del}(n)}{n\cdot 5^n}.
   \end{displaymath}
   The preceding example shows that $L(f,\G)\cong 1.514708$.

   For $\Del=\ZZ^2$, we are computing entropy for $f$ considered as a
   polynomial in the commuting variables $x,y$, and this is given by
   Mahler measure to be
   \begin{displaymath}
      L(f,\ZZ^2)=\int_0^1\int_0^1 \log|5-2\cos(2\pi s)-2\cos(2\pi
      t)|\,ds\,dt \cong 1.507982\,.
   \end{displaymath}
   Observe that any word in $S\subset\G$ whose product is 1 abelianizes
   to one in $\ZZ^2$, and so $r_{\G}^{}(n)\le r_{\ZZ^2}^{}(n)$, which is
   reflected in the inequality $L(f,\G)\ge L(f,\ZZ^2)$.

   The case $\Del=\F_2$ is more interesting. Here any word in the
   generators $x,y$ whose product is 1 must also give a word in any
   group $\Del$ containing $x, y$ with product 1, so that $r_{\F_2}(n)\le
   r_{\Del}(n)$ for all groups $\Del$. Hence $L(f,\F_2)\ge L(f,\Del)$, so
   that $L(f,\F_2)$ gives a universal upper bound.

   To compute $L(f,\F_2)$, start with the generating function for
   $r_{\F_2}(2n)$, which by \cite[I.9]{delaHarpe} is known to be
   \begin{displaymath}
      G(t)=\sum_{n=1}^\infty r_{\F_2}(2n)t^n=\frac{3}{1+2\sqrt{1-12t}}
      =1+4t+28t^2+232t^3+ 2092t^4+\dots \,.
   \end{displaymath}
   Letting
   \begin{displaymath}
      H(t)=\int_0^t \frac{G(u)-1}{u}\,du,
   \end{displaymath}
   then a calculation with \textit{Mathematica} shows that
   \begin{displaymath}
      L(f,\F_2)=\log 5 -\frac{1}2 H(5^{-2})=\log
      \Bigl[\frac1{18}(35+13\sqrt{13})\Bigr]\cong 1.514787
   \end{displaymath}
   Indeed $L(f,\F_2)>L(f,\G)$, but only in the fifth decimal place.
\end{remark}

We remark that Lewis Bowen \cite{LewisBowenEntropy} has extended a
notion of entropy to actions of sofic groups, and in particular free
groups (although these do not have F{\o}lner sequences).
In \cite[Example 1.1]{LewisBowenEntropy} he shows that
in the previous discussion $L(f,\F_2)$
equals the sofic entropy of the corresponding algebraic $\F_2$-action.

It is somewhat surprising that here $L(f,\F_2)$ is the logarithm of an
algebraic number, since in the case $\Del=\ZZ^d$ with $d>1$ this appears
not to generally be the case. For example, when $d=3$ and $g =
1+u_1+u_2+u_3$, Smyth \cite{Smyth2} has computed the logarithmic Mahler
measure to be $\mahler(g)=\h(\al_g)=\log [7\zeta(3)/2\pi^2]$, where $\zeta$ is the
Riemann zeta-function.

\smallskip Next we extend the face entropy inequality described for principal
$\zd$-actions in \cite[Remark 5.5]{LSW} to principal $\ZG$-actions. This
inequality proves that many such actions have strictly positive entropy.

We start with the basic case.

\begin{proposition}
   \label{prop:basic-inequality}
   Let $f(x,y,z)=\sum_{r=0}^D g_r(x,z)y^r\in\ZG$ with
   $g_0(x,z)\ne0$. Then $\h(\af)\ge\mahler(g_0)$.
\end{proposition}

\begin{proof}
   By definition the map $\rho_{g_0(x,z)}\colon
   \TT^{\ZZ^2}\to\TT^{\ZZ^2}$ has kernel $X_{g_0}$, and is surjective
   since multiplication by $g_0(x,z)$ is injective on
   $\ZZ[x^{\pm},z^{\pm}]$. Define $\Phi\colon\TT^{\ZZ^2}\to\TT^{\ZZ^2}$
   by $(\Phi u)_{i,j}=u_{i+j,j}$. Then for every $k\in\ZZ$ we have that
   $u\in\ker \rho_{g_0(x,z)}$ iff $\Phi^k(u)\in\ker
   \rho_{g_0(xz^{-k},z)}$.

   The algebraic $\ZZ^2$-action $\al_{g_0}$ on $X_{g_0}$ has entropy
   $\mahler(g_0)$. Fix $\eps>0$ and $\del>0$. Then for sufficiently
   large rectangles $Q\subset\ZZ^2$ there is a $(Q,\eps)$-separated set
   $\{u_1,\dots,u_N\}\subset X_{g_0}$ with $N\ge
   e^{(\mahler(g_0)-\del)|Q|}$. Note that since $d_{\TT}$ is translation
   invariant, for every $u\in\TT^{\ZZ^2}$ the translated set
   $\{u_1+u,\dots,u_N+u\}$ is also $(Q,\eps)$-separated.

   Let $\ttt{n}{}=\sum_{i,j} \ttt{n}{i,j}x^iz^j\in\TT^{\ZZ^2}$ be
   arbitrary, and put $t=\sum_{n=-\infty}^\infty
   \ttt{n}{}y^n\in\TT^\G$. The condition for $t$ to be in $X_f$ is that
   $\rho_f(t)=0$, which in terms of the $\ttt{n}{}$ becomes
   \begin{equation}
      \label{eqn:convolution}
      \sum_{r=0}^D \rho_{g_r(xz^k,z)}(\ttt{k+r}{})=0 \text{\quad for all $k\in\ZZ$.}
   \end{equation}

   Let $L\ge1$. For each $(j_0,j_1,\dots,j_{L-1})\in\{1,\dots,N\}^L$ we
   will construct a sequence $\{\ttt{0}{j_0},\ttt{-1}{j_0,j_1}, \dots,
   \ttt{-L+1}{j_0,j_1,\dots,j_{L-1}}\}$ in $\TT^{\ZZ^2}$ that will be used
   to create an $\eps$-separated set in $X_f$.

   Put $\ttt{n}{}=0$ for all $n\ge1$, so that \eqref{eqn:convolution} is
   trivially satisfied for $k\ge1$. Define $\ttt{0}{j_0}=u^{}_{j_0}$
   for $1\le j_0\le N$. Then \eqref{eqn:convolution} is satisfied at
   $k=0$ since $\rho_{g_0(x,z)}(u_{j_0})=0$.

   Since $\rho_{g_0(xz^{-1},z)}$ is surjective on $\TT^{\ZZ^2}$, for
   each $j_0$ there is a $\ttt{-1}{j_0}\in\tzt$ such that
   $$\rho_{g_0(xz^{-1},z)}(\ttt{-1}{j_0})=-\rho_{g_1(xz^{-1},z)}(\ttt{0}{j_0}).$$ Let
   $\ttt{-1}{j_0,j_1}=\ttt{-1}{j_0}+ \Phi(u^{}_{j_1})$ for $1\le j_1\le N$. Since
   $\rho_{g_0(xz^{-1},z)}\bigl(\Phi(u_{j_1})\bigr)=0$, it follows that
   \eqref{eqn:convolution} is satisfied at $k=-1$ for all choices of
   $j_0$ and $j_1$.

   Similarly, for each $\{\ttt{0}{j_0},\ttt{-1}{j_0,j_1}\}$ there is a
   $\ttt{-2}{j_0,j_1}\in\tzt$ with
   \begin{displaymath}
      \rho_{g_0(xz^{-2},z)}\bigl(\ttt{-2}{j_0,j_1}\bigr) =
      - \rho_{g_1(xz^{-2},z)}\bigl(\ttt{-1}{j_0,j_1}\bigr)
      - \rho_{g_2(xz^{-2},z)}\bigl(\ttt{0}{j_0}\bigr).
   \end{displaymath}
   Put $\ttt{-2}{j_0,j_1,j_2}=\ttt{-2}{j_0,j_1}+\Phi^2(u^{}_{j_2})$ for
   $1\le j_2\le N$. Then \eqref{eqn:convolution} holds for $k=-2$ and
   all choices of $j_0,j_1,j_2$.

   Continuing in this way, for every $L$-tuple
   $(j_0,j_1,\dots,j_{L-1})\in\{1,\dots,N\}^L$ we have constructed
   $\{\ttt{0}{j_0},\ttt{1}{j_0,j_1},\dots,\ttt{-L+1}{j_0,\dots,j_{L-1}}\}$
   so that \eqref{eqn:convolution} is satisfied for $k\ge -L+1$. Each
   choice can be further extended to find $\ttt{n}{}$ for $n\le-L$ for
   which the resulting point is in $X_f$.

   Identify $\ZZ^3$ with $\G$ via $(i,j,k)\leftrightarrow x^iz^jy^k$,
   and consider
   $\overline{Q}=\bigcup_{k=0}^{L-1}\Phi^k(Q)\times\{-k\}\subset\G$. We
   claim that the $N^L$ points in $X_f$ constructed above are
   $(\overline{Q},\eps)$-separated. For at the first index $k$ for
   which $j_k\ne j_k'$, the points $\Phi(u_{j_k})$ and
   $\Phi^k(u_{j_k'})$ differ by at least $\eps$ at some coordinate of
   $\Phi^k(Q)$.

   Finally, if we choose $Q$ be be very long in the $x$-direction
   compared with the $z$-direction, and make $L$ small compared with
   both, we can make $\overline{Q}$ as right-invariant as we please. Hence
   there is a F{\o}lner sequence $\{\overline{Q}_m\}$ in $\G$ with
   \begin{displaymath}
      s(\overline{Q}_m,\eps)\ge
      e^{\bigl(\mahler(g_0)-\del\bigr)|\overline{Q}_m|},
   \end{displaymath}
   and thus $\h(\af)\ge\mahler(g_0)$.
\end{proof}

Recall from Section \ref{sec:algebraic-actions} that the \textit{Newton
polygon} $\Newt(f)$ of $f=\sum_{k,l}f_{kl}(z)x^ky^l\in\ZG$ is the convex
hull in $\RR^2$ of those points $(k,l)$ for which $f_{kl}(z)\ne0$. A
\textit{face} of $\Newt(f)$ is the intersection of $\Newt(f)$ with a
supporting hyperplane, which is either a point or a line segment. For
each face $F$ of $\Newt(f)$ let $f_F(x,y,z)=\sum_{(k,l)\in F\cap\ZZ^2}
f_{k,l}(z)x^ky^l$. For every face $F$ of $\Newt(f)$ there is a change of
variables followed by multiplication by a monomial transforming $f$ so
that $F$ now lies on the $x$-axis with the rest of $\Newt(f)$ in the
upper half-plane. Since entropy is invariant under such transformations,
we can apply Proposition \ref{prop:basic-inequality} to obtain the
following face entropy inequality.

\begin{corollary}
   \label{cor:face-entropy}
   If $f\in\ZG$ and $F$ is a face of $\Newt(f)$, then
   $\h(\af)\ge\h(\al_{f_F})$.
\end{corollary}

Face entropies are essentially logarithmic Mahler measures of
polynomials in commuting variables, and so easy to compute. Observe that
if $\h(\af)=0$, then the Corollary shows that $h(\al_{f_F})=0$ for every
face $F$ of $\Newt(f)$, and that Proposition \ref{prop:zero-entropy} gives a
complete characterization of what $f_F$ can be. Indeed, Smyth used the
face entropy inequality as the starting point for his proof of
Proposition \ref{prop:zero-entropy}. However, the algebraic complexity
of $\ZG$ prevents a direct extension of his methods, leaving open a very
interesting question.

\begin{problem}
   Characterize those $f\in\ZG$ for which $\h(\af)=0$.
\end{problem}

Recall that the Pinsker $\sigma$-algebra of a measure-preserving action
$\al$ is the largest $\sigma$-algebra for which the entropy of $\al$ on
this $\sigma$-algebra is zero. An action has \textit{completely positive
entropy} if its Pinsker $\sigma$-algebra is trivial. An old argument of
Rohlin shows that the Pinsker $\sigma$-algebra of an algebraic action
$\al$ on $X$ is invariant under translation by every periodic
point. Hence if the periodic points are dense, then the Pinsker
$\sigma$-algebra is invariant under all translations, and so arises from
the quotient map $X\to X/Y$, where $Y$ is a compact $\al$-invariant
subgroup. Thus the restriction of $\al$ to its Pinsker $\sigma$-algebra
is again an algebraic action (see \cite[Prop.\ 6.2]{LSW}), providing one
reason for the importance of the problem above. For algebraic
$\zd$-actions there is an explicit criterion for completely positive
entropy in terms of associated prime ideals (see \cite[Thm.\ 6.5]{LSW}),
but even in the case of Heisenberg actions no similar criterion is
known.

\begin{problem}
   Characterize the algebraic $\G$-actions with completely positive entropy.
\end{problem}

For algebraic $\zd$-actions, completely positive entropy is sufficient
to imply that they are isomorphic to Bernoulli shifts
\cite{RudolphSchmidt}.
Is the same true for algebraic $\G$-actions?

\begin{problem}
   If an algebraic $\G$-actions has completely positive entropy, is it
   necessarily measurably isomorphic to a Bernoulli $\G$-action?
\end{problem}

\section{Periodic Points and Entropy}\label{sec:periodic-points}

Let $\Del$ be a countable discrete group, and $\al$ be an algebraic
$\Del$-action on $X$. A point $t\in X$ is \textit{periodic} for $\al$ if
its $\Del$-orbit is finite. The stabilizer $\{\del\in\Del:\del\cdot
t=t\}$ of such a point $t$ has finite index in $\Del$, and we will need a
generous supply of such subgroups. Call $\Del$ \textit{residually
finite} if, for every finite subset $K$ of $\Del$, there is a
finite-index subgroup $\Lam$ of $\Del$ such that
$\Lam\cap(K\setminus\{1\})= \emptyset$. Every finite-index subgroup
$\Lam$ of $\Del$ contains a further finite-index subgroup $\Lam'$ of
$\Lam$ that is normal in $\Del$, so that residual finiteness can be
defined using finite-index normal subgroups. If $\{\Lam_n\}$ is a
sequence of finite-index subgroups of $\Del$, we will say
$\Lam_n\to\infty$ if, for every finite set $K\subset\Del$, there is an
$n_K$ such that $\Lam_n\cap(K\setminus\{1\})=\emptyset$ for all $n\ge
n_K$.

For a finite-index subgroup $\Lam$ of $\Del$, let
\begin{displaymath}
   \FixL(\al):=\{t\in X:\lam\cdot t=t \text{ for all $\lam\in\Lam$}\}.
\end{displaymath}
If $\Lam$ is normal in $\Del$, then for $\lam\in\Lam$ and $\del\in\Del$
we have that $\del\lam\del^{-1}=\lam'\in\Lam$. Hence in this case
$\FixL(\al)$ is $\Del$-invariant, since if $t\in\FixL(\al)$ then
$\lam\cdot(\del\cdot t)=(\lam\del)\cdot t=(\del\lam')\cdot t
=\del\cdot(\lam'\cdot t)=\del\cdot t$.

We will focus on expansive principal actions. Let $\Del$ be a countable
residually finite discrete group, and let $f\in\ZD$ be expansive. Recall
the notations and results from Proposition \ref{prop:cover}. For a
finite-index subgroups $\Lam$ of $\Del$, let a superscript of $\Lam$ on
a space denote the set of those elements in the space fixed by $\Lam$,
so for example
\begin{displaymath}
   \lidr^\Lam=\{ w\in\lidr :\lam\cdot t=t \text{ for all
   $\lam\in\Lam$} \}.
\end{displaymath}

\begin{proposition}[{\cite[Prop.\ 5.2]{DS}}]
   \label{prop:fix}
   Let $\Del$ be a countable discrete group, and $f\in\ZD$ be
   expansive. Using the notations of Proposition \ref{prop:cover}, for
   every finite-index subgroup $\Lam$ of $\Del$ we have that
   \begin{displaymath}
      \FixL(\af)=\pi\bigl(\lidz^\Lam\bigr)\cong\lidz^\Lam/\rf\bigl(
      \lidz^\Lam \bigr).
   \end{displaymath}
\end{proposition}

\begin{proof}
   Any point $t\in\FixL(\af)$ can be lifted to a point
   $\ttil\in\lidr^\Lam$, and then the proof of Proposition
   \ref{prop:cover} yields $\Lam$-invariant points at every step.
\end{proof}

Note that $\lidz^\Lam$ is a free abelian group of rank $[\Del:\Lam]$, the
index of $\Lam$ in $\Del$, and that $\rf$ is an injective endomorphism
on this group by expansiveness.

\begin{corollary}
   \label{cor:fixed-point-count}
   Under the hypotheses in Proposition \ref{prop:fix},
   \begin{equation}
      \label{eqn:fixed-point-count}
      |\FixL(\af)|=|\det(\rf|_{\lidr^\Lam})|.
   \end{equation}
\end{corollary}

\begin{proof}
   Observe that $\rf$ is an injective linear map on the real vector space
   $\lidr^\Lam$ of dimension $n=[\Del:\Lam]$ and maps the lattice
   $\lidz^\Lam$ to itself. If $A$ is any $n\times n$ matrix with integer
   entries and nonzero determinant, then it is easy to see, for
   example from the Smith normal form, that $|\ZZ^n/A\ZZ^n|=|\det(A)|$.
\end{proof}

\begin{remark}
   Since we will need to complexify some spaces in order to use complex
   eigenvalues, let us say a word about conventions regarding
   determinants. If $A$ is an $n\times n$ real matrix, we could regard
   $A$ as an $n\times n$ complex matrix acting on $\CC^n$, or as a
   $(2n)\times(2n)$ real matrix acting on $\RR^n\oplus i \RR^n$, and
   these have different determinants. We will always use the first
   interpretation.
\end{remark}

With Corollary \ref{cor:fixed-point-count} we begin to see the
connections among periodic points, entropy, and Fuglede-Kadison
determinants. For expansive actions, periodic points are separated, and
so  $|\Lam|^{-1}\log |\FixL(\af)|$ should approximate, or at least
provide a lower bound for, the entropy
$\h(\af)$. But by \eqref{eqn:fixed-point-count}, this is also a
finite-dimensional approximation to the logarithm $\log \det_{\nd} f$ of
the Fuglede-Kadison determinant of $\rf$. The main technical issue is
then to show both of these approximations converge to the desired
limits. For expansive $\af$ is is relatively easy, but for general $\af$
the are numerous difficulties to overcome.

We can now deduce two important properties of expansive principal
actions.

\begin{proposition}
   \label{prop:dense-and-positive-entropy}
   Let $\Del$ be a residually finite countable discrete amenable group,
   and let $f\in\ZD$ be expansive. Then
   \begin{enumerate}
     \item the $\af$-periodic points are dense in $X_f$, and
     \item if $|X_f|>1$ \textup{(}i.e., if $f$ is not invertible in $\ZD$\textup{)}, then $\h(\af)>0$.
   \end{enumerate}
\end{proposition}

\begin{proof}[Sketch of proof]
   For (1), let $t\in X_f$ and find $u \in\lidz$ with $\pi(u)
   =\beta(u\cdot\wtri)=t$. By
   Prop.\ \ref{prop:cover}(4), there is a finite subset $K$ of
   $\Del$ such that if $u_K\in\lidz$  denotes the restriction of $u$ to
   $K$ and 0 elsewhere, then $\pi(u_K)$ is close to $t$. By enlarging
   $K$ if necessary, we can assume that $\sum_{\del\notin
   K}|\wtri_{\del}|$ is small. Let $\Lam$ be a finite-index subgroup of
   $\Del$ such that $\lam K\cap\lam' K=\emptyset$ for distinct
   $\lam,\lam'\in\Lam$. Let $t_0=\sum_{\lam\in\Lam} \lam\cdot
   \pi(u_K)$. Then $t_0\in\FixL(\af)$ and $t_0$ is close to $t$.

   For (2), again choose a finite $K$ so that $\sum_{\del\notin
   K}|\wtri_{\del}|$ is small, where we may assume that $\wtri_1=1$.
   Find $\Lam$ with $\lam K\cap\lam'
   K=\emptyset$ for distinct
   $\lam,\lam'\in\Lam$. If $F$ is any sufficiently invariant F{\o}lner
   set, then $|F\cap\Lam|$ is about $|F|/[\Del:\Lam]$. Then for every
   choice of $b_\lam=0 \text{ or }1$ for $\lam\in F\cap\Lam$, the points
   $\sum_{\lam\in F\cap\Lam} b_\lam(\lam\cdot \wtri)$ are
   $(F,1/2)$-separated. Hence if $\{F_n\}$ is any F{\o}lner sequence,
   then for every $\epsilon<1/2$ we have that
   \begin{displaymath}
      \limsup_{n\to\infty} \frac1{|F_n|}\log\, s(F_n,\epsilon)
      \ge \limsup_{n\to\infty}\frac{1}{|F_n|} |F_n\cap\Lam|\,\log2
      \ge \frac{\log 2}{[\Del:\Lam]}>0. \qedhere
   \end{displaymath}
\end{proof}

 For $\Del=\ZZ^d$, it turns out that the periodic points for $\al_M$ are
 always dense in $X_M$ for every finitely generated $\ZZ\ZZ^d$-module $M$
 \cite[Cor.\ 11.3]{SchmidtBook}. The simple example of multiplication by
 $3/2$ on $\QQ$ dualizes to an automorphism of a compact group with no
 nonzero periodic points, since $(3/2)^n-1$ is invertible in $\QQ$ (see
\cite[Example 5.6(1)]{SchmidtBook}). We do not know the answer to the
following.

\begin{problem}
   Let $\Del$ be a countable discrete residually finite group, and $M$
   be a finitely generated $\ZD$-module. Must the $\al_M$-periodic
   points always be dense in~ $X_M$?
\end{problem}

We now focus on using periodic points to calculate entropy. It is
instructive to see how these calculations work in a simple example.

\begin{example}
   \label{exam:periodic-golden-mean}
   Let $\Del=\ZZ$, and let $f(u)=u^2-u-1\in\ZZ\Del=\ZZ[u^{\pm}]$. Let
   $\Lam_n=n\ZZ$ and $F_n=\{0,1,\dots,n-1\}$. Then $\{F_n\}$ is a
   F{\o}lner sequence in $\ZZ$ that is also a fundamental domain for
   $\Lam_n$. Denote by $\Om_n$ the set of all $n$-th roots of unity in
   $\CC$. For $\zeta\in\Om_n$ let
   $v_\zeta=\sum_{k\in\ZZ} \zeta^k u^k\in\ell^\infty(\ZZ,\CC)^{\Lam_n}$.
   Then $\rho_u(v_{\zeta})=\zeta v_\zeta$, so the $v_\zeta$ form an
   eigenbasis for the shift on $\ell^\infty(\ZZ,\CC)^{\Lam_n}$. Hence
   $\rho_f(v_\zeta)=f(\zeta)v_\zeta$ for each $\zeta\in\Om_n$.

   We can consider the elements of $\ell^\infty(\ZZ,\CC)^{\Lam_n}$ as
   elements in the $n$-dimensional complex vector space
   $\ell^\infty(\ZZ/n\ZZ,\CC)$. The matrix of $\rf$ with respect to the
   basis $\{1,u,\dots,u^{n-1}\}$ is the circulant matrix
   \begin{equation}
      \label{eqn:circulant}
      C_n(f)=
      \begin{bmatrix*}[r]
         -1 & -1 & 1       &0 &\ldots&0&0 \\
         0  & -1 & -1      &1 &\ldots&0& 0\\
            &    & &  \ddots   &&&\\
         1  & 0  &         &\ldots&& -1 &-1 \\
         -1 & 1  &         &\ldots&&  0 & -1
      \end{bmatrix*}.
   \end{equation}
   We can compute the determinant of $C_n(f)$ using the eigenbasis
   $\{v_\zeta:\zeta\in\Om_n\}$. Factor $f(u)=(u-\tau)(u-\sigma)$, where
   $\tau=(1+\sqrt{5})/2$ and $\sigma=-1/\tau$. Then
   \begin{align*}
         |\Fix_{\Lam_n}(\af)|&=|\det \rf|_{\ell^\infty(\ZZ/n\ZZ,\CC)}|
         =\prod_{\zeta\in\Om_n}|f(\zeta)| \\
         &= \prod_{\zeta\in\Om_n}|\tau-\zeta|\cdot|\sigma-\zeta|=|\tau^n-1|\cdot|\sigma^n-1|.
   \end{align*}
   Then
   \begin{displaymath}
      \lim_{n\to\infty} \frac1n \log|\Fix_{\Lam_n}(\af)|=\log \tau=\mahler(f).
   \end{displaymath}
   Since $\af$ is expansive, periodic points are separated, and so $\log
   \tau$ is certainly a lower bound for $\h(\af)$. But it is also easy
   in this case to see that it is an upper bound, using for example
   approximations from homoclinic points as in the proof of Proposition
   \ref{prop:dense-and-positive-entropy}.

   Note that if $f\in\ZZ[u^{\pm}]$ had a root $\xi\in\SS$, then the
   factor $|\xi^n-1|$ in the calculation of determinant would occasionally
   be very small, which could cause the limit not to exist. This is one
   manifestation of the difficulties with nonexpansive actions.
\end{example}

We turn to the Heisenberg case $\Del=\G$. For $q,r,s>0$ put
$\Lamq=\<x^{rq},y^{sq},z^q\>$, which is a normal subgroup of $\G$ of
index $rsq^3$. Let $f\in\ZG$ be expansive. Recall that
$\kappa_0=\kappa_0(f)=1/(3\|f\|_1)$ is an expansive constant for $\af$. In
particular, if $\Lam$ is a finite-index subgroup of $\G$, and if $t\ne
u\in \FixL(\af)$, then for any fundamental domain $Q$ of $\Lam$ there is
a $\gamma\in Q$ such that $d_{\TT}(t_\gamma,u_\gamma)\ge\kappa_0$.

A bit of notation about the limits we will be taking is convenient. If
$\psi(\Lamq)$ is a quantity that depends on $\Lamq$, we write
$\lim_{q\to\infty} \psi(\Lamq)=c$ to mean that for every $\epsilon>0$
there is a $q_0$ such that $|\psi(\Lamq)-c|<\epsilon$ for every $q\ge
q_0$ and all sufficiently large $r$ and $s$.

\begin{theorem}[{\cite[Thm.\ 5.7]{DS}}]
   \label{thm:periodic-point-limit}
   Let $f\in\ZG$ be expansive, and define $\Lamq$ as above. Then
   \begin{displaymath}
      \lim_{q\to\infty} \frac{1}{[\G:\Lamq]} \log|\Fix_{\Lamq}(\af)|=
      \h(\af).
   \end{displaymath}
\end{theorem}

\begin{proof}
   In light of Example \ref{exam:periodic-golden-mean}, it is tempting
   to use a fundamental domain for $\Lamq$ of the form
   $Q=\{x^ky^lz^m:0\le k<rq,0\le l<sq,0\le m<q\}$, but such a $Q$ is far
   from being right-F{\o}lner, since right multiplication by $x$ will
   drastically shear $Q$ in the $z$-direction.

   The method used in \cite[Thm.\ 5.7]{DS} is to decompose $Q$ into
   pieces, each of which is thin in the $y$-direction, and translate
   these pieces to different locations in $\G$. The union of these
   translates will still be a fundamental domain, but now it will also
   be F{\o}lner, and so can be used for entropy calculations. This
   method depends in general on a result of Weiss \cite{Weiss}, and
   ultimately goes back to the $\epsilon$-quasi-tiling machinery of
   Ornstein and Weiss \cite{OrnsteinWeiss}. In our case, we can give a
   simple description of this decomposition.

   Choose integers $a(q)$ such that $a(q)\to\infty$ but $a(q)/q\to0$ as
   $q\to\infty$. Consider the set
   \begin{displaymath}
      F_{rq,L,q}=\{x^ky^lz^m:0\le k<rq,L\le l< L+a(q),0\le m <q\}.
   \end{displaymath}
   It is easy to verify that since $a(q)$ is small compared with $q$,
   right multiplication by a fixed $\gamma\in\G$ creates only small
   distortions, and so for every $\gamma\in\G$ we have that
   \begin{equation}
      \label{eqn:folner}
      \lim_{q\to\infty}\frac{|F_{rq,L,q}\setdiff F_{rq,L,q}\gamma|}
      {|F_{rq,L,q}|} =0.
   \end{equation}
   Define
   \begin{displaymath}
      \Qq=\bigcup_{j=0}^{[rq/a(q)]-1}x^{rqj}F_{rq,a(q)j,q}\,\,,
   \end{displaymath}
   where if $a(q)$ does not evenly divide $rq$, make the obvious
   modification in the last set.
   Then $\Qq$ is also a fundamental domain for $\Lamq$, but now it is
   also a F{\o}lner sequence as $q\to\infty$ by \eqref{eqn:folner}. By
   the separation property of periodic points, for all
   $\epsilon<\kappa_0$ we have that $|\Fix_{\Lamq}(\af)|\le
   s(\Qq,\epsilon)$. Since $\{\Qq\}$ is F{\o}lner,
   \begin{displaymath}
      \limsup_{q\to\infty}\frac1{|\Qq|}\log|\Fix_{\Qq}(\al)|\le\h(\af).
   \end{displaymath}

   For the reverse inequality, let $\del>0$ and let
   $\epsilon<\del/3$. Choose a finite set $E\subset\G$ such that
   $\sum_{\gamma\notin E}|\wtri_{\gamma}|<\epsilon/\|f\|_1$. The sets
   $\Pq=\bigcap_{\gamma\in E}\Qq\gamma$ also form a F{\o}lner sequence,
   and $|\Pq|/|\Qq|\to1$ as $q\to\infty$.

   Fix $q$, $r$, and $s$ for the moment, and choose a
   $(\Pq,\del)$-separated set $S$ of maximal cardinality. For every
   $t\in S$, let $\ttil\in\ligr$ be its lift, with
   $\|\ttil\|_\infty\le1$ and $\beta(\ttil)=t$. Write $v(t)\in
   \ligz^{\Lamq}$ for the unique point with
   $v(t)_\gamma=(\rf(\ttil))_\gamma$ for all $\gamma\in\Qq$. Our choice
   of $E$ implies that the points in $\{\pi\bigl(v(t)\bigr):t\in S\}
   \subset \Fix_{\Lamq}(\af)$  are $(\Pq,\del/3)$-separated, hence
   distinct. Hence $|S|\le|\Fix_{\Lamq}(\af)|$. Since $\{\Pq\}$ is
   F{\o}lner and $|\Pq|/|\Qq|\to1$ as $q\to\infty$, we see that
   \begin{displaymath}
      \h(\af)=\liminf_{q\to\infty} \frac1{|\Pq|}\log s(\Pq,\del)
      \le\liminf_{q\to\infty}\frac1{|\Qq|}\log|\Fix_{\Lamq}(\af)|,
   \end{displaymath}
   completing the proof.
\end{proof}

We apply this result, combined with Corollary
\ref{cor:fixed-point-count}, to compute entropy for expansive principal
$\G$-actions $\af$. Using the above notations, let
$V_{rq,sq,q}=\ligc^{\Lamq}$, so that
\begin{equation}
   \label{eqn:det-rho}
   |\Fix_{\Lamq}(\af)|=|\det(\rf|_{V_{rq,sq,q}})|\,.
\end{equation}
We will compute this determinant by decomposing $V_{rq,sq,q}$ into
$\rf$-invariant subspaces, each having dimension 1 or $q$. To do this,
for each $(\xi,\eta,\zeta)\in\SS^3$ let
\begin{displaymath}
   \vxyz=\sum_{k,l,m=-\infty}^\infty \xi^k\eta^l\zeta^m\,x^ky^lz^m\in\ligc.
\end{displaymath}
Observe that $\rho_y(\vxyz)=\vxyz\cdot y^{-1}=\eta\,\vxyz$, and
similarly $\rho_z(\vxyz)=\zeta\,\vxyz$, so that $\vxyz$ is a common
eigenvector for $\rho_y$ and $\rho_z$. However,
\begin{align*}
   \rho_x(\vxyz)&=\sum_{k,l,m} \xi^k\eta^l\zeta^m\,x^ky^lz^m\cdot x^{-1}
   =\sum_{k,l,m}\xi^k\eta^l\zeta^m\,x^{k-1}y^lz^{m-l}\\
   &=\sum_{k,l,m}\xi^{k+1}(\eta\zeta)^l\zeta^m\,x^ky^lz^m=\xi\,v_{\xi,\eta\zeta,\zeta}\,.
\end{align*}
Let $\Om_q=\{\zeta\in\SS:\zeta^q=1\}$, and
$\Om_q'=\Om_q\setminus\{1\}$. For arbitrary $(\xi,\eta)\in\SS^2$ let
\begin{displaymath}
   W_\zeta(\xi,\eta)=
   \begin{cases}
      \bigoplus_{j=0}^{q-1} \CC v_{\xi,\eta\zeta^j,\zeta} &\text{if
      $\zeta\in\Om_q'$,}\\
      \CC v_{\xi,\eta,1} &\text{if $\zeta=1$}.
   \end{cases}
\end{displaymath}
By the above, for every $(\xi,\eta)\in\SS^2$ and $\zeta\in\Om_q$, the
subspace $W_\zeta(\xi,\eta)$ is invariant under the right action of
$\G$.

Now let $f\in\ZG$, and assume that $f$ is expansive. Adjusting $f$ by a
power of $x$ if necessary, we may assume that $f$ has the form
$f(x,y,z)=\sum_{j=0}^D x^jg_j(y,z)$, where each
$g_j(y,z)\in\ZZ[y^{\pm},z^{\pm}]$ with $g_0(y,z)\ne 0$ and
$g_D(y,z)\ne0$. The action of $\rf$ on $\vxyz$ is then given by
\begin{displaymath}
   \rf(\vxyz)=\sum_{j=0}^D\rho_{x^jg_j(y,z)}(\vxyz)=\sum_{j=0}^D
   \xi^jg_j(\eta,\zeta)v_{\xi,\eta\zeta^j,\zeta}\,.
\end{displaymath}
If $\zeta=1$, then $\rf(v_{\xi,\eta,1)}=f(\xi,\eta,1)v_{\xi,\eta,1}$, so
is given by the $1\times1$ matrix $A_{1,f}(\xi,\eta)=[f(\xi,\eta,1)]$. If
$\zeta\in\Om_q'$, then the matrix of $\rf$ on $W_{\zeta}(\xi,\eta)$ takes
the following $q\times q$ circulant-like form, where for notational
convenience we assume that $q>D$:
\begin{center}
\resizebox{\linewidth}{!}{$
   A_{\zeta,f}(\xi,\eta)=
   \begin{bmatrix}
      g_0(\eta,\zeta) & g_1(\eta,\zeta)\xi & \cdots &
      g_D(\eta,\zeta)\xi^D &  \cdots &0 \\
      0 & g_0(\eta\zeta,\zeta) & g_1(\eta\zeta,\zeta)\xi &  g_2(\eta\zeta,\zeta)\xi^2
       & \cdots &0 \\
      0 & 0 & g_0(\eta\zeta^2,\zeta) & g_1(\eta\zeta^2,\zeta)\xi &
      \cdots &0 \\
      \vdots & \vdots & \ddots & \ddots & \ddots  & \vdots\\
      g_1(\eta\zeta^{q-1},\zeta)\xi & g_2(\eta\zeta^{q-1},\zeta)\xi^2 &
       \cdots&0& \cdots & g_0(\eta\zeta^{q-1},\zeta)
    \end{bmatrix}$.}
\end{center}
By our expansiveness
assumption $\rf$ is injective, and each $W_\zeta(\xi,\eta)$ is
$\rf$-invariant, hence $\det A_{\zeta,f}(\xi,\eta)\ne0$ for all
$(\xi,\eta)\in\SS^2$ and all $\zeta\in\Om_q$.

Now suppose that $r,s\ge0$. For convenience we will assume that both $r$
and $s$ are primes distinct from $q$. Then
$\{\vxyz:(\xi,\eta,\zeta)\in\Om_{rq}\times\Om_{sq}\times\Om_q\}$ is a
basis for $V_{rq,sq,q}$. Note that if $\zeta\in\Om_q'$, then
$W_{\zeta}(\xi,\eta)=W_{\zeta}(\xi,\eta\zeta^j)$ for $0\le j<q$. Since
$\Om_{sq}\cong\Om_s\times\Om_q$ by relative primeness of $s$ and $q$, we
can parameterize the spaces $W_{\zeta}(\xi,\eta)$ by $\eta\in\Om_s$.
This gives the $\rf$-invariant decomposition
\begin{multline}
   \label{eqn:big-decomp}
   V_{rq,sq,q}=\bigoplus\bigl\{
   W_1(\xi,\eta):(\xi,\eta)\in\Om_{rq}\times\Om_{sq}\bigr\}\\
   \oplus \,\, \bigoplus\bigl\{ W_{\zeta}(\xi,\eta): (\xi,\eta,\zeta)\in
   \Om_{rq}\times\Om_{s}\times \Om_q'\bigr\}.
\end{multline}
We now evaluate the limit of
\begin{equation}
   \label{eqn:limit}
   \frac1{rsq^3}
   \log|\det(\rf|_{V_{rq,sq,q}})|
\end{equation}
as $r,s\to\infty$ using the decomposition \eqref{eqn:big-decomp}.

On each of the $1$-dimensional spaces $W_1(\xi,\eta)$ in
\eqref{eqn:big-decomp} $\rf$ acts as multiplication by
$f(\xi,\eta,1)$. Hence the contribution  to
\eqref{eqn:limit} of the first large summand in
\eqref{eqn:big-decomp}  is
\begin{equation}
   \label{eqn:boundary}
   \frac1{rsq^3}\sum_{(\xi,\eta)\in\Om_{rq}\times\Om_{sq}}\log|f(\xi,\eta,1)|.
\end{equation}
By expansiveness, $f(\xi,\eta,1)$ never vanishes for
$(\xi,\eta)\in\SS^2$, so that $\log|f(\xi,\eta,1)|$ is continuous on
$\SS^2$. By convergence of the Riemann sums to the integral,
as $r,s\to\infty$ we have
that
\begin{displaymath}
   \frac{1}{rsq^2}
   \sum_{(\xi,\eta)\in\Om_{rq}\times\Om_{sq}}\log|f(\xi,\eta,1)|
   \to\iint_{\SS^2}\log|f(\xi,\eta,1)|\,d\xi d\eta .
\end{displaymath}
The additional factor of $q$ in the denominator of \eqref{eqn:boundary}
shows that it converges to 0 as $r,s\to\infty$.

For the spaces $W_{\zeta}(\xi,\eta)$ with $\zeta\in\Om_q'$, the
expansiveness assumption shows that $\det A_{\zeta,f}(\xi,\eta)$ never
vanishes for $(\xi,\eta)\in\SS^2$, so again $\log|\det
A_{\zeta,f}(\xi,\eta)|$ is continuous on $\SS^2$. Hence for
$\zeta\in\Om_q'$, as $r,s\to\infty$ we have that
\begin{displaymath}
   \frac{1}{|\Om_{rq}\times\Om_{sq}|}\sum_{(\xi,\eta)\in\Om_{rq}\times\Om_s}
   \log|\det A_{\zeta,f}(\xi,\eta)|\to
   \iint_{\SS^2}\log|\det A_{\zeta,f}(\xi,\eta)|\,d\xi d\eta .
\end{displaymath}
Adding these up over $\zeta\in\Om_q'$, and observing that
$|\Om_s|=|\Om_{qs}|/q$, we have shown the following.

\begin{theorem}
   \label{thm:expansive-entropy}
   Let $f(x,y,z)=\sum_{j=0}^D x^jg_j(y,z)\in\ZG$ be expansive. Then
   \begin{equation}
      \label{eqn:riemann-sum-entropy}
      \h(\af)=\lim_{\substack{q\to\infty\\ q\mathrm{\  prime}}}\frac1{q^2}
      \sum_{\zeta\in\Om_q} \iint_{\SS^2}\log|\det
      A_{\zeta,f}(\xi,\eta)|\,d\xi d\eta,
   \end{equation}
   where the matrices $A_{\zeta,f}(\xi,\eta)$ are as given above.
\end{theorem}

At first glance the denominator $q^2$ in \eqref{eqn:riemann-sum-entropy}
appears puzzling. The explanation is that one $q$ comes from averaging
over the $q$-th roots of unity, while the other $q$ comes from the size
of the matrices $A_{\zeta,f}(\xi,\eta)$. From the point of view of von
Neumann algebras, we should really be using the ``normalized
determinant'' $|\det A_{\zeta,f}(\xi,\eta)|^{1/q}$ corresponding to the
normalized trace on $\CC^n$, and then the second $q$ would not appear.

For expansive polynomials in $\ZG$ that are linear in $x$ the entropy
formula in the preceding theorem can be simplified considerably.

\begin{theorem}
   \label{thm:linear-entropy}
   Let $f(x,y,z)=g(y,z)+xh(y,z)\in\ZG$, where $g(y,z)$ and $h(y,z)$ are
   Laurent polynomials in $\ZZ[y^{\pm},z^{\pm}]$. If $\af$ is expansive,
   then
   \begin{equation}
      \label{eqn:max-entropy}
      \h(\af)=\int_{\SS}\max\bigl\{\bigl(\mahler(g(\cdot,\zeta)\bigr),
      \mahler\bigl(h(\cdot,\zeta)\bigr)\bigr\}\,d\zeta,
   \end{equation}
   where
   $\mahler\bigl(g(\cdot,\zeta)\bigr)=\int_{\SS}\log|g(\eta,\zeta)|\,d\eta$
   and similarly for $h$.
\end{theorem}

\begin{corollary}
   Let $f(x,y,z)=g(y,z)+xh(y,z)$, and assume that neither $g$ nor $h$
   vanish anywhere on $\SS^2$. If $\af$ is expansive, then
   $\h(\af)=\max\{\mahler(g),\mahler(h)\}$.
\end{corollary}

\begin{proof}[Proof of the Corollary]
   After the change of variables $x\mapsto y$, $y\mapsto x$, and
   $z\mapsto z^{-1}$, we can apply Theorem \ref{thm:linear-expansive} to
   conclude that either $\mahler\bigl(g(\cdot,\eta)\bigr)>
   \mahler\bigl(h( \cdot,\eta)\bigr)$ for all $\eta\in\SS$, or
    $\mahler\bigl(g(\cdot,\eta)\bigr)<
   \mahler\bigl(h( \cdot,\eta)\bigr)$ for all $\eta\in\SS$. The result
   then follows from \eqref{eqn:max-entropy} by integrating over
   $\eta\in\SS$.
\end{proof}

The intuition behind \eqref{eqn:max-entropy} is simple to explain. The
matrices $A_{\zeta,f}(\xi,\eta)$ for $\zeta\in\Om_q'$ take the form
\begin{equation}
   \label{eqn:linear-matrix}
   A_{\zeta,f}(\xi,\eta)=
   \begin{bmatrix}
      g(\eta,\zeta) & h(\eta,\zeta)\xi & 0 & \cdots & 0 \\
      0  & g(\eta\zeta,\zeta) & h(\eta\zeta,\zeta)\xi & \cdots & 0 \\
      \vdots & \vdots & \ddots & \ddots & \vdots \\
      h(\eta\zeta^{q-1},\zeta)\xi & 0 & \cdots & & g(\eta\zeta^{q-1},\zeta)
   \end{bmatrix}.
\end{equation}
Then
\begin{equation}
   \label{eqn:linear-det}
   \det A_{\zeta,f}(\xi,\eta)=\prod_{j=0}^{q-1}g(\eta\zeta^j,\zeta)+(-1)^{q-1}\xi^q
   \prod_{j=0}^{q-1}h(\eta\zeta^j,\zeta).
\end{equation}
For fixed $\zeta$ the first product behaves like
$\Mahler\bigl(g(\cdot,\zeta)\bigr)^q$, and similarly for $h$. Whichever
is larger will dominate, suggesting the formula \eqref{eqn:max-entropy}.

To make this intuition precise, we require several lemmas.

\begin{lemma}
   For every $\xi\in\CC$, $\zeta\in\SS$, and $n\ge1$,
   \begin{displaymath}
      \frac1n\log|\xi^n-\zeta^n|\le\log^+|\xi|+\frac{\log 2}{n}.
   \end{displaymath}
\end{lemma}

\begin{proof}
   If $|\xi|\le1$, then $|\xi^n-\zeta^n|\le2$. If $|\xi|>1$, then
   \begin{align*}
      \frac1n \log|\xi^n-\zeta^n| &\le \frac1n \log(|\xi|^n+1) \le
      \frac1n \Bigl\{\log |\xi|^n + \log\Bigl(
      \frac{|\xi|^n+1}{|\xi|^n}\Bigr) \\
      &= \log|\xi|+\frac1n\log\Bigl(1+\frac1{|\xi|^n}\Bigr)
      \le\log|\xi|+ \frac{\log 2}{n}. \qedhere
   \end{align*}
\end{proof}

If $0\ne \phi(u)\in\CC[u]$, then \eqref{eqn:mahler} shows that
$\mahler(\phi)>-\infty$. We will use the convention
$\mahler(0)=-\infty$. With the usual conventions about
arithmetic and inequalities involving
$-\infty$, the results that follow will make sense and are true
even if some of the
polynomials are ~0.

\begin{lemma}
   Suppose that $\phi(u)\in\CC[u]$ has degree $\le D$. Then for every
   $\zeta\in\SS$ and $n\ge1$,
   \begin{equation}
      \label{eqn:riemann-sum-inequality}
      R_n\bigl(\log|\phi|\bigr)(\zeta) := \frac1n \sum_{j=0}^{n-1}
      \log|\phi(e^{2\pi 2 j/n}\zeta)|\le\mahler(\phi)+\frac{D\log 2}{n}.
   \end{equation}
\end{lemma}

\begin{proof}
   Fix $n\ge1$ and set $\om=e^{2\pi i/n}$, so that
   $R_n(\log|\phi|)(\zeta)
   =\frac1n \sum_{j=0}^{n-1}\log|\phi(\om^j\zeta)|$.   Write
   $\phi(u)=c_du^d +\dots+c_0$, where $d\le D$ and $c_d\ne0$. Factor
   $\phi(u)=c_d\prod_{k=1}^d(u-\xi_k)$. Then using the previous lemma,
   \begin{align*}
      R_n\bigl(\log|\phi|\bigr)(\zeta) &= \frac1n \sum_{j=0}^{n-1}
      \log\Bigl|c_d\prod_{k=1}^d(\om^j\zeta-\xi_k)\Bigr|
      = \log|c_d|+\frac1n \sum_{k=1}^d \sum_{j=0}^{n-1}
      \log|\om^j\zeta-\xi_k| \\
      &=\log|c_d|+\frac1n \sum_{k=1}^d
      \log\Bigl|\prod_{j=0}^{n-1}(\om^j\zeta-\xi_k)\Bigr|
      =\log|c_d|+\frac1n \sum_{k=1}^d \log|\zeta^n-\xi_k^n| \\
      &\le|c_d|+\sum_{k=1}^d\log^+|\xi_k|+\frac{d \log 2}{n}\le
      \mahler(\phi) +\frac{D\log 2}{n}. \qedhere
   \end{align*}
\end{proof}

\begin{lemma}
   \label{lem:max-entropy}
   Let $\phi(u),\psi(u)\in\CC[u]$ each have degree $\le D$. Then for
   every $n>(D\log2)/\epsilon$,
   \begin{align*}
      \max\bigl\{\mahler(\phi),\mahler(\psi)\bigr\}
      &\le \int_{\SS} \max\{ R_n\bigl(\log|\phi|)(\zeta),
      R_n(\log|\psi|)(\zeta)\bigr\}\,d\zeta \\
      &\le  \max\bigl\{\mahler(\phi),\mahler(\psi)\bigr\} +\epsilon .
   \end{align*}
\end{lemma}

\begin{proof}
   If $n>(D\log2)/\epsilon$, then the previous lemma implies that for
   every $\zeta\in\SS$ we have that
   $R_n(\log|\phi|)(\zeta)\le\mahler(\phi)+\epsilon$ and
   $R_n(\log|\psi|)(\zeta)\le \mahler(\psi)+\epsilon$. Hence
   \begin{displaymath}
       \max\{ R_n\bigl(\log|\phi|\bigr)(\zeta),
      R_n\bigl(\log|\psi|\bigr)(\zeta) \}\le\max\{
      \mahler(\phi),\mahler(\psi)\}+ \epsilon,
   \end{displaymath}
   and the second inequality follows by integrating over $\zeta\in\SS$.

   For the first inequality, observe that
   \begin{displaymath}
      \mahler(\phi)= \int_{\SS}
      R_n\bigl(\log|\phi|\bigr)(\zeta)\,d\zeta
      \le \int_{\SS} \max\{ R_n\bigl(\log|\phi|)(\zeta),
      R_n(\log|\psi|)(\zeta)\bigr\}\,d\zeta,
   \end{displaymath}
   and similarly for $\phi$.
\end{proof}

We need one more property of Mahler measure, proved by David Boyd
\cite{BoydContinuity}.  Recall that
$\Mahler(\phi)=\exp\bigl(\mahler(\phi)\bigr)$, and by convention
$\exp(-\infty)=0$.

\begin{theorem}[\cite{BoydContinuity}]
   \label{thm:boyd-continuity}
   The map $\CC^{D+1}\to[0,\infty)$ given by $$(c_0,c_1,\dots,c_D)\mapsto
   \Mahler(c_0+c_1u+\dots+c_Du^D)$$ is continuous.
\end{theorem}

Continuity is clear when the coefficients remain bounded away from 0
since the roots are continuous functions of the coefficients, but for
example if $c_D\to0$ then continuity is more subtle. Boyd's proof, which
also applies the polynomials in several variables with bounded degree,
uses Graeffe's root-squaring method, managing to sidestep various
delicate issues, leading to a remarkably simple proof.

\begin{proof}[Proof of Theorem \ref{thm:linear-entropy}]
   If $\zeta\in \Om_q'$, then we have seen in \eqref{eqn:linear-det} that
   \begin{displaymath}
       \det A_{\zeta,f}(\xi,\eta)=\prod_{j=0}^{q-1}g(\eta\zeta^j,\zeta)+(-1)^{q-1}\xi^q
   \prod_{j=0}^{q-1}h(\eta\zeta^j,\zeta).
   \end{displaymath}
   By \eqref{eqn:mahler}, for any complex numbers $a$ and $b$,
   \begin{displaymath}
      \int_{\SS}\log|a+\xi^q b|\,d\xi=\max\{\log|a|,\log|b|\}.
   \end{displaymath}
   Hence
   \begin{displaymath}
      \int_{\SS}\log\bigl|\det A_{\zeta,f}(\xi,\eta)\bigr|\,d\xi
      =\max\Bigl\{
      \log\Bigl|\prod_{j=0}^{q-1}g(\eta\zeta^j,\zeta)\Bigr|,
      \log\Bigl|\prod_{j=0}^{q-1}h(\eta\zeta^j,\zeta)\Bigr|\Bigr\}.
   \end{displaymath}
   It follows that
   \begin{displaymath}
      \frac1q \iint_{\SS^2}\log\bigl|\det
      A_{\zeta,f}(\xi,\eta)\bigr|\,d\xi d\eta
      =\int_{\SS}\max\bigl\{ R_q\bigl(\log|g(\cdot,\zeta)|\bigr)(\eta),
       R_q\bigl(\log|h(\cdot,\zeta)|\bigr)(\eta)\bigr\}\,d\eta.
    \end{displaymath}
    Let $\epsilon>0$. Writing $g_\zeta(y)=g(y,\zeta)$ and
    $\h_\zeta(y)=h(y,\zeta)$, note that these are Laurent polynomials of
    uniformly bounded degree $\le D$ in $\CC[y^{\pm}]$, where the degree of
    a Laurent polynomial is the length of its Newton polygon. Applying
    \ref{lem:max-entropy} to $g_\zeta$ and $h_\zeta$, we obtain that for
    $q>(D\log2)/\epsilon$,
    \begin{align*}
       \max\bigl\{\mahler(g(\cdot,\zeta)),\mahler(h(\cdot,\zeta))\bigr\}
       &\le \frac1q \iint_{\SS^2}\log|\det A_{\zeta,f}(\xi,\eta)|\,d\xi
       d\eta \\
       &<\max\bigl\{
       \mahler(g(\cdot,\zeta)),\mahler(h(\cdot,\zeta))\bigr\} +\epsilon.
    \end{align*}

    We must now consider the possibility that for some $\zeta\in\SS$
    both $g(y,\zeta)$ and $h(y,\zeta)$ vanish as polynomials in
    $\CC[y^{\pm}]$, and claim this never happens provided that $\af$ is
    expansive. Write $g(y,z)=\sum_j g_j(z)y^j$ and $h(y,z)=\sum_k
    h_k(z)y^k$, where $g_j(z),h_k(z)\in\ZZ[z^{\pm}]$. If $\zeta_0\in\SS$ were
    a common zero for all the $g_j(z)$ and $h_k(z)$, then its minimal
    polynomial over $\QQ$ would divide every $g_j(z)$ and $h_k(z)$. But
    then the nonzero point
    $w=\sum_{i,j,k}\zeta_0^k\,x^iy^jz^k\in\ligc$ would have $\rf(w)=0$,
    contradicting expansiveness.

    Thus the function $\zeta\mapsto
    \max\bigl\{\mahler(g(\cdot,\zeta)),\mahler(h(\cdot,\zeta))\bigr\}$
    is continuous on $\SS$. Hence the sums
    \begin{displaymath}
       \frac1{q^2} \sum_{\zeta\in\Om_q} \iint_{\SS^2}\log\bigl|\det
        A_{\zeta,f}(\xi,\eta)\bigr|\,d\xi d\eta
     \end{displaymath}
     simultaneously approximate $\h(\af)$ by Theorem
     \ref{thm:expansive-entropy}, and are also Riemann sums
     over $\Om_q$
     of the
     continuous function  $\max\bigl\{\mahler(g(\cdot,\zeta)),
     \mahler(h(\cdot,\zeta))\bigr\}$, and so converge to the integral in
     \eqref{eqn:max-entropy}, concluding the proof.
\end{proof}

Our proof of  Theorem \ref{thm:linear-entropy} made use of the
expansiveness hypothesis on $\af$ in order to use periodic points to
approximate entropy. But surely the entropy formula
\eqref{eqn:max-entropy} is valid more generally.

\begin{problem}
   Is the entropy formula \eqref{eqn:max-entropy} for linear polynomials
   in $\ZG$ valid for arbitrary $g(y,z)$ and $h(y,z)$ in $\ZZ[y^{\pm},z^{\pm}]$?
\end{problem}

It will follow from results in the next section that
\eqref{eqn:max-entropy} is valid if either $g\equiv 1$ or $h\equiv1$.

Let us turn to the quadratic case
$f(x,y,z)=g_0(y,z)+xg_1(y,z)+x^2g_2(y,z)$. We start with a simple result
about determinants.

\begin{lemma}
   Let $a_j$, $b_j$, and $c_j$ $(0\le j\le q-1)$ be arbitrary complex
   numbers. Then
   \begin{displaymath}
      \det
      \begin{bmatrix}
         a_0  & b_0  & c_0  &  0  & \ldots &  0  &  0  \\
         0    & a_1  & b_1  & c_1 & \ldots & 0   &  0  \\
              &      &      &     & \ddots &     &     \\
         c_{q-2} & 0  &  0   & 0   & \ldots & a_{q-2} & b_{q-2} \\
         b_{q-1} & c_{q-1} & 0 & 0 & \ldots & 0 & a_{q-1}
      \end{bmatrix}
      =\prod_{j=0}^{q-1} a_j - \tr\prod_{j=0}^{q-1}
      \begin{bmatrix}
         -b_j & c_j\\-a_j&0
      \end{bmatrix} +\prod_{j=0}^{q-1} c_j.
   \end{displaymath}
   If $c_j a_{j+1}=-b_jb_{j+1}$ for every $j$, where subscripts are
   taken mod $q$, then the value of this determinant simplifies to
   \begin{displaymath}
      \prod_{j=0}^{q-1} a_j\, -(-1)^q(\tau^q+\sigma^q)\prod_{j=0}^{q-1}
       b_j+\prod_{j=0}^{q-1} c_j,
    \end{displaymath}
    where $\tau=(1+\sqrt{5})/2$ and $\sigma=-1/\tau$.
\end{lemma}

\begin{proof}
   Taking subscripts mod $q$, a permutation $\pi$ of $\{0,1,\dots,q-1\}$
   contributes a nonzero summand in the expansion of the determinant if
   and only if it has the form $\pi(j)=j+\epsilon_j$, where
   $\epsilon_j=0$, 1, or 2. The sequences $\{\epsilon_j\}\in\{0,1,2\}^q$
   corresponding to permutations are precisely the closed paths of
   length $q$ in the labeled shift of finite type depicted below.

   \centerline{\includegraphics[scale=.23]{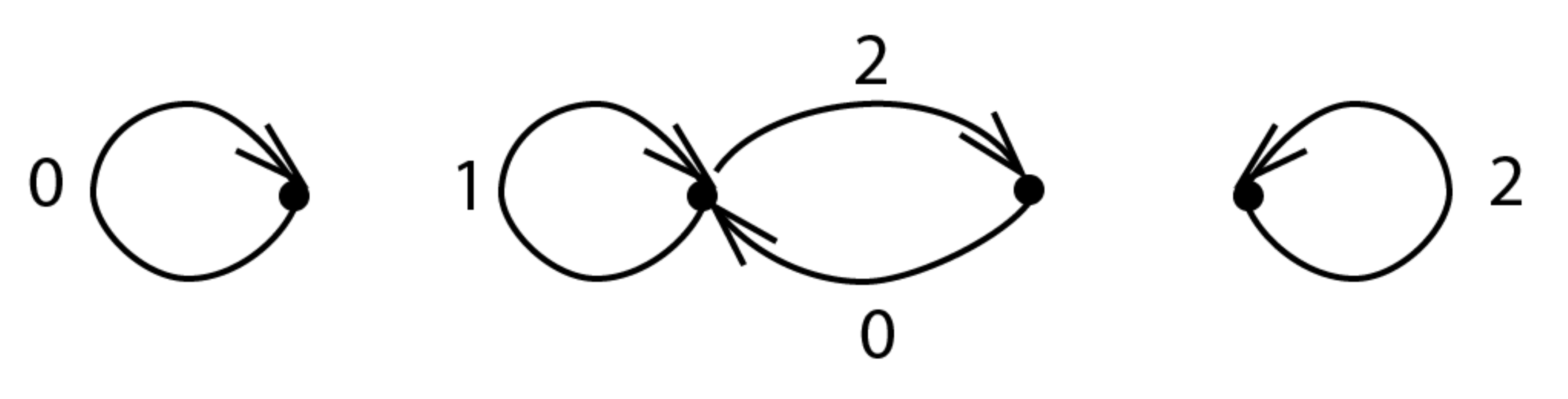}}
   The paths $00\ldots 0$ and $22\ldots2$ give the terms $a_0a_1\cdots
   a_{q-1}$ and $c_0c_1\cdots c_{q-1}$, respectively, while it is easy to
   check that the golden mean shift of finite type produces middle term
   of the result.

   If $c_ja_{j+1}=-b_jb_{j+1}$ for all $j$, then each occurrence of a
   block 20 in a closed path of length $q$ can be replaced by the block
   11, changing the factor $c_ja_{j+1}$ to $b_jb_{j+1}$ together with an
   appropriate sign change. The result of these substitutions is that
   every closed path of length $q$ in the golden mean shift gives the
   same contribution $-(-1)^qb_0b_1\cdots b_{q-1}$ to the expansion, and
   there are $\tr\left[
   \begin{smallmatrix}
      1&1\\1&0
   \end{smallmatrix}\right]^q=\tau^q+\sigma^q$ of them.
\end{proof}

We use this to give the quadratic analogues of \eqref{eqn:linear-matrix}
and \eqref{eqn:linear-det}.

\begin{corollary}
   Let $f(x,y,z)=g_0(y,z)+xg_1(y,z)+x^2g_2(y,z)\in\ZG$. Suppose that
   $\zeta\in\Om_q'$, and for $(\xi,\eta)\in\SS^2$ define
   $A_{\zeta,f}(\xi,\eta)$ as above. Then
   \begin{equation}\label{eqn:quadratic-det}
      \begin{split}
     \det A_{\zeta,f}(\xi,\eta)=\prod_{j=0}^{q-1}g_0(\eta\zeta^j,\zeta)
     &+(-1)^{q-1}\tr\prod_{j=0}^{q-1} \begin{bmatrix} -g_1(\eta\zeta^j,\zeta)\xi &
        g_2(\eta\zeta^j,\zeta)\xi^2\\ -g_0(\eta\zeta^j,\zeta)&0\end{bmatrix}\\
     &+\xi^{2q}\prod_{j=0}^{q-1}g_2(\eta\zeta^j,\zeta).
      \end{split}
   \end{equation}
   If $g_1(y,z)g_1(yz,z)=-g_2(y,z)g_0(yz,z)$,
   then
   \begin{equation}
      \label{eqn:simple-det}
      \begin{split}
      \det A_{\zeta,f}(\xi,\eta)=\prod_{j=0}^{q-1}g_0(\eta\zeta^j,\zeta)
      &+(-1)^{q-1}(\tau^q+\sigma^q)\xi^q
      \prod_{j=0}^{q-1}g_1(\eta\zeta^j,\zeta)\\
      &+\xi^{2q}\prod_{j=0}^{q-1}g_2(\eta\zeta^j,\zeta).
      \end{split}
   \end{equation}
\end{corollary}

Motivated by Theorem \ref{thm:expansive-entropy} and the rigorous
results from the expansive linear case, we can now formulate a
reasonable conjecture about entropy in the quadratic case.

As in the linear case, for fixed irrational $\zeta$, the growth rate of
the first and third terms in \eqref{eqn:quadratic-det} are given my
$\mahler\bigl(g_0(\cdot,\zeta)\bigr)$ and
$\mahler\bigl(g_2(\cdot,\zeta)\bigr)$, respectively. The growth rate for
the second term should be the same for almost every choice of
$(\xi,\eta)\in\SS^2$, and if so denote this value by $b_f(\zeta)$. For
example, in case $g_1(y,z)g_1(yz,z)=-g_2(y,z)g_0(yz,z)$, from
\eqref{eqn:simple-det} we see that
$b_f(\zeta)=\log\tau+\mahler\bigl(g_1(\cdot,\zeta)\bigr)$.

\begin{problem}
   Let $f(x,y,z)=g_0(y,z)+xg_1(y,z)+x^2g_2(y,z)\in\ZG$. Is
   \begin{displaymath}
      \h(\af)=\int_{\SS}\max\bigl\{\mahler\bigl(g_0(\cdot,\zeta)\bigr),
      b_f(\zeta),\mahler\bigl(g_2(\cdot,\zeta)\bigr)
      \bigr\}\,d\zeta\,?
   \end{displaymath}
   In particular, if $g_1(y,z)g_1(yz,z)=-g_2(y,z)g_0(yz,z)$, is
   \begin{equation}
      \label{eqn:quadratic-entropy}
      \h(\af)=\int_{\SS}\max\bigl\{\mahler\bigl(g_0(\cdot,\zeta)\bigr),
      \log\tau+\mahler\bigl(g_1(\cdot,\zeta)\bigr),\mahler\bigl(g_2(\cdot,\zeta)\bigr)
      \bigr\}\,d\zeta\,?
   \end{equation}
\end{problem}

\begin{example}
   Fix $g(y,z)\in\ZG$, and let $g_0(y,z)=-g(yz^{-1},z)g(y,z)$,
   $g_1(y,z)=g(y,z)$, and $g_2(y,z)=1$. These satisfy the conditions for
   \eqref{eqn:quadratic-entropy}, and this formula becomes
   \begin{equation}
      \label{eqn:simple-max}
      \h(\af)=\int_{\SS}\max\bigl\{2\mahler\bigl(g(\cdot,\zeta)\bigr),
      \log\tau+\mahler\bigl(g(\cdot,\zeta)\bigr),0
      \bigr\}\,d\zeta .
   \end{equation}
   For example, letting $g(y,z)=(z-1)y+z^2-1$, for each of the three
   functions in \eqref{eqn:simple-max} there is a range of $\zeta$ for
   which it is the maximum.
\end{example}

A similar analysis can be carried out for higher degree polynomials, but
the evaluation of the relevant determinants now involves a finite family
of more complicated (and interesting) shifts of finite type.

\section{Lyapunov Exponents and Entropy}\label{sec:geometric-entropy}

The methods of the previous section to compute entropy have some serious
limitations because of the expansiveness assumptions. There is a more
geometric approach to entropy, using the theory of Lyapunov exponents,
which Deninger used in \cite{DeningerLyapunov} to calculate
Fuglede-Kadison determinants, or equivalently by Theorem \ref{thm:li-thom},
entropy, in much greater generality.

To motivate this approach, first recall the linear example
$f(x,y,z)=y-g(x,z)\in\ZG$. For $\vxz$ defined in \eqref{eqn:vxz} and
$w=\sum_{-\infty}^\infty c_n\vxz\,y^n$ with $\rf(w)=0$, we have from
\eqref{eqn:recurrence} that
\begin{equation}
   \label{eqn:product-recurrence}
   c_n=\Bigl[\prod_{j=0}^{n-1}g(\xi\zeta^j,\zeta)\Bigr] c_0.
\end{equation}
For irrational $\zeta$, the products in \eqref{eqn:product-recurrence}
have growth rate
\begin{displaymath}
   \frac1n \log\Bigl|\prod_{j=0}^{n-1}g(\xi\zeta^j,\zeta)\Big|=
   \frac1n \sum_{j=0}^{n-1}\log|g(\xi\zeta^j,\zeta)|\to\mahler(g(\cdot,\zeta)),
\end{displaymath}
and the limit is the same for almost every $\xi$ by ergodicity of
irrational rotations. For toral automorphisms, entropy is the sum of the
growth rates on various eigenspaces that are positive. By analogy, we
would expect entropy here to be the integral of the positive growth
rates, i.e., that
\begin{equation}
   \label{eqn:monic-linear}
   \h(\af)=\int_{\SS}\max\{\mahler(g(\cdot,\zeta)),0\}\, d\zeta.
\end{equation}
This is a special case of \eqref{eqn:max-entropy}, but with no
assumptions on $g$.

Indeed, since the eigenspaces here are 1-dimensional, the techniques
used in \cite{LSW} can be adapted to prove the validity of
\eqref{eqn:monic-linear} for every
$0\neq g\in\ZZ[x^{\pm},z^{\pm}]$. However, since this will be subsumed
under Deninger's results, there is no need to provide an independent
proof here.

To state the main result in \cite{DeningerLyapunov}, we need to give a little
background. For each irrational $\zeta\in\SS$ there is the rotation
algebra $\mathcal{A}_\zeta$, which is the von Neumann algebra version of
the twisted $l^1$ algebras used in Allan's local principle (see
\cite{AndersonPaschke} for details). There are also natural maps
$\pi_{\zeta}\colon\nG\to\mathcal{A}_\zeta$. As explained in
\cite[\S5]{D}, there is a faithful normalized trace function
$\tr_{\zeta}$ on each $\mathcal{A}_\zeta$ such that
$\tr_{\nG}(a)=\int_{\SS}\tr_{\zeta}(a)\,d\zeta$ for every
$a\in\nG$. This implies that for determinants we have
\begin{equation}
   \label{eqn:FK-determinants}
   \log \mathrm{det}_{\nG}(a)=\int_{\SS}\log \det\bigl( \pi_{\zeta}(a)\bigr)\, d\zeta.
\end{equation}
Hence we need a way of evaluating the integrands for $a=\rf$.

Suppose that $f\in\ZG$ is monic in $y$ and of degree $D$, and so has the form
\begin{displaymath}
   f(x,y,z)=y^D - g_{D-1}(x,z)y^{D-1}-\dots -g_0(x,z),
\end{displaymath}
where the $g_j(x,z)\in\ZZ[x^{\pm},z^{\pm}]$ and $g_0\ne 0$. Calculations
similar to Example \ref{exam:y^2-xy-1} show that if
$\rf(\sum_{-\infty}^\infty c_n\vxz y^n)=0$, then, for every ~$n$,
\begin{displaymath}
   c_{n+D}=g_{D-1}(\xi\zeta^{n},\zeta)c_{n+D-1}+g_{D-2}(\xi\zeta^{n},\zeta)c_{n+D-2}
   + \dots +g_0(\xi\zeta^{n},\zeta)c_n .
\end{displaymath}
Put
\begin{equation}
   \label{eqn:companion-matrix}
   A(\xi,\zeta)=
   \begin{bmatrix}
      0 & 1 & 0 & \cdots & 0 \\
      0 & 0 & 1 & \cdots & 0 \\
      \vdots & \vdots & \vdots & \ddots & \vdots \\
      g_0(\xi,\zeta) & g_1(\xi,\zeta)&
      g_2(\xi,\zeta)
      &\cdots & g_{D-1}(\xi,\zeta)
   \end{bmatrix},
\end{equation} and let
\begin{displaymath}
   A_n(\xi,\zeta)=A(\xi\zeta^{n-1},\zeta)A(\xi\zeta^{n-2},\zeta)\dots A(\xi,\zeta)
\end{displaymath}
be the corresponding cocycle. Then the
recurrence relation for the $\{c_n\}$ can be written as
\begin{displaymath}
   A_n(\xi,\zeta)
   \begin{bmatrix}
      c_0\\ \vdots \\ c_{D-1}
   \end{bmatrix}
   =
   \begin{bmatrix}
      c_{n}\\ \vdots \\ c_{n+D-1}
   \end{bmatrix}.
\end{displaymath}

The $A_n(\xi,\zeta)$ are $D\times D$ matrices, and we need to include all
directions with positive growth rate. There is a deep and important
theorem governing this.

\begin{theorem}[Oseledec Multiplicative Ergodic Theorem]
   \label{thm:met}
   Let $T$ be an invertible ergodic measure-preserving transformation on
   a measure space $(\Om,\nu)$, and $B\colon \Om\to \CC^{D\times D}$ be
   a measurable map from $\Om$ to the $D\times D$ complex matrices such that
   $\int_{\Om}\log^+\|B(\om)\|\,d\nu(\om)<\infty$. Then there is a
   $T$-invariant measurable set $\Om_0\subset\Om$ with $\nu(\Om_0)=1$, an $M\le D$,
   real numbers $\chi_1< \chi_2 <\dots<\chi_M$ and multiplicities
   $r_1,\dots, r_M\ge1$ such that $r_1+\dots+r_M=D$, and measurable maps
   $V_j$ from $\Om_0$ to the space of $r_j$-dimensional subspaces of
   $\CC^D$, such that for all $\om\in\Om_0$ we have that
   \begin{enumerate}
     \item $\CC^D=V_1(\om)\oplus\dots\oplus V_M(\om)$,
     \item $B(\om)V_j(\om)=V_j(T\om)$, and
      \item $\displaystyle \frac1n\log\frac {\|B(T^{n-1}\om)B(T^{n-2}\om)\dots
       B(\om)v\|}{\|v\|} \to \chi_j$ uniformly for  $0\ne v\in V_j(\om)$.
   \end{enumerate}
\end{theorem}

To apply this result to our situation, fix an irrational
$\zeta\in\SS$. Let $T_{\zeta}\colon\SS\to\SS$ be given by
$T_{\zeta}(\xi)=\xi\zeta$. Let $B(\xi)=A(\xi,\zeta)$. Since the
entries of $B(\xi)$ are continuous in $\xi$, clearly
$\int_{\SS}\log^+\|B(\xi)\|<\infty$. Hence there are Lyapunov exponents
$\chi_j(\zeta)$ and multiplicities $r_j(\zeta)$. With these in hand,
we can now state Deninger's main result from \cite{DeningerLyapunov} as
it applies to the Heisenberg case.

\begin{theorem}
   \label{thm:deninger}
   Let $f(x,y,z)=y^D-g_{D-1}(x,z)y^{D-1}-\dots-g_0(x,z)\in\ZG$, where
   the $g_j(x,z)\in\ZZ[x^{\pm},z^{\pm}]$ and $g_0\ne0$. For every
   irrational $\zeta$ denote the Lyapunov exponents for $A(\xi,\zeta)$
   and $T_{\zeta}$ as above by $\chi_j(\zeta)$ with multiplicities
   $r_j(\zeta)$. Then
   \begin{displaymath}
      \log \mathrm{det}_\zeta(\pi_{\zeta}(\rf))=\sum_{j}r_j(\zeta)\chi_j(\zeta)^+,
   \end{displaymath}
   and hence by Theorem \ref{thm:li-thom} and \eqref{eqn:FK-determinants},
   \begin{equation}
      \label{eqn:entropy-via-lyapunov}
      \h(\af)=\log \mathrm{det}_{\nG} \rf =\int_{\SS}\sum_j r_j(\zeta)\chi_j(\zeta)^+.
   \end{equation}
\end{theorem}

\begin{example}
   Let $f(x,y,z)=y^2-(2x-1)y+1$. For each $(\xi,\zeta)\in\SS^2$ there is
   a nonzero vector $v(\xi,\zeta)\in\CC^2$ and a multiplier
   $\kappa(\xi,\zeta)\in\CC$ with $|\kappa(\xi,\zeta)|\ge1$ such that
   \begin{displaymath}
      \begin{bmatrix}
         0 & 1\\ -1 & 2\xi-1
      \end{bmatrix}
      v(\xi,\zeta)=\kappa(\xi,\zeta)v(\xi\zeta,\zeta).
   \end{displaymath}
   Since the determinants of these matrices all have absolute value 1,
   there is exactly one nonnegative Lyapunov exponent, of multiplicity
   one. For each irrational $\zeta\in\SS$ its value is given by
   $\chi(\zeta)=\int_{\SS}\log|\kappa(\xi,\zeta)|\,d\xi$, and hence
   \begin{displaymath}
      \h(\af)=\int_{\SS}\chi(\zeta)\,d\zeta = \iint_{\SS^2}
      \log|\kappa(\xi,\zeta)|\,d\xi d\zeta.
   \end{displaymath}
   A numerical calculation of the graph of $\log|\kappa(\xi,\zeta)|$ is
   shown in Figure  \ref{fig:surface}, and indicates the complexity of these phenomena
   even in the quadratic case.
\end{example}

   \begin{figure}[!h]
     \centering
     \includegraphics[scale=.21]{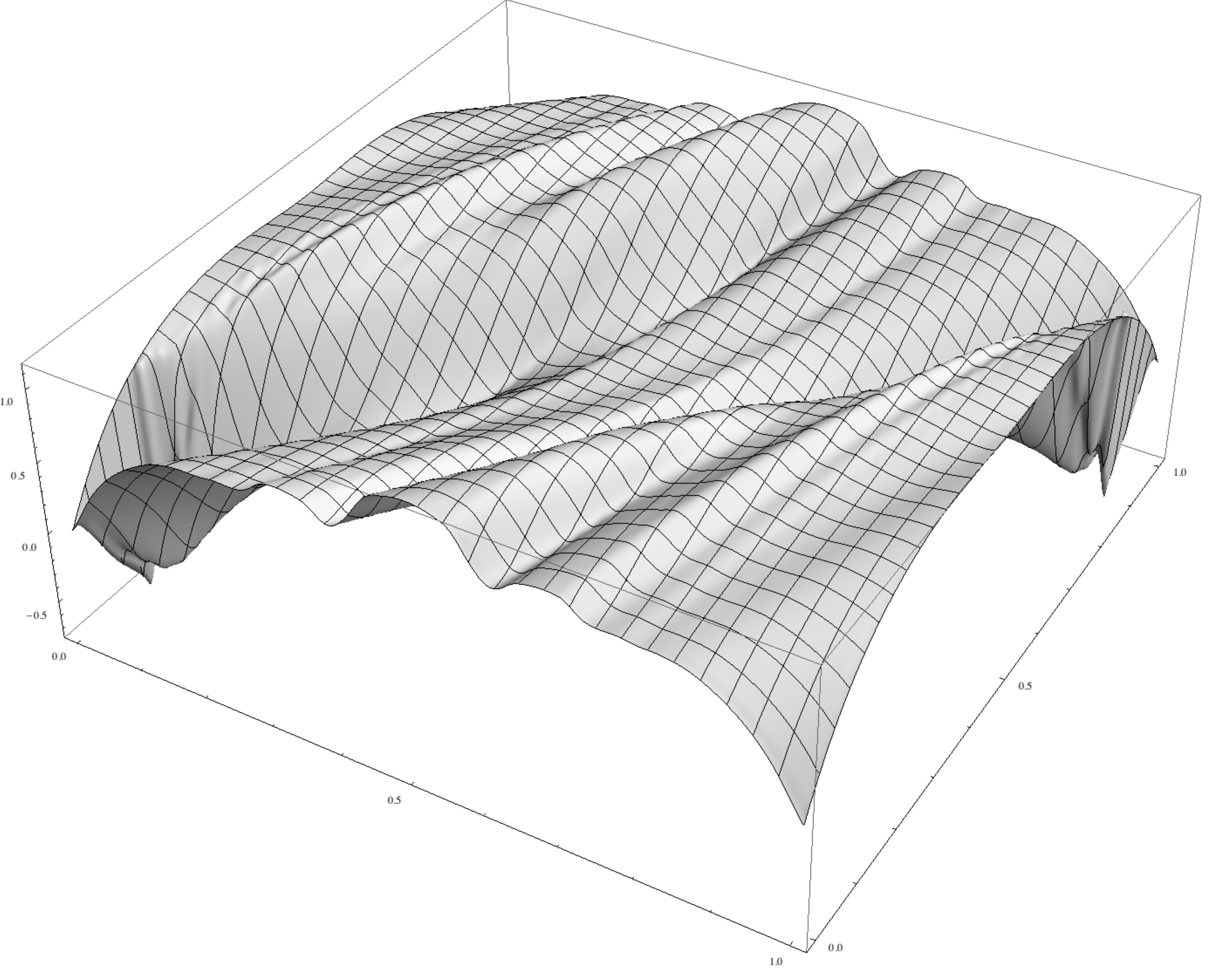}
     \caption{Graph of $\log|\kappa(\xi,\zeta)|$.  \label{fig:surface}}
  \end{figure}

Although Lyapunov exponents are generally difficult to compute, there is
a method to obtain rigorous lower bounds on the largest Lyapunov exponent known as
``Herman's subharmonic trick.'' Its use in our context was suggested to
us by Michael Bjorklund.

\begin{proposition}
   Let $f(x,y,z)=y^D-g_{D-1}(x,z)y^{D-1}-\dots-g_0(x,z)\in\ZG$, where
   the $g_j(x,z)\in\ZZ[x,z^{\pm}]$, so that only nonnegative powers of
   $x$ are allowed, and $g_0(x,z)\ne0$. For every
   irrational $\zeta\in\SS$ let $\chi_\infty(\zeta)$ denote the largest
   Lyapunov exponent resulting from Theorem \ref{thm:deninger}. Then
   $\chi_\infty(\zeta)\ge \log \spr A(0,\zeta)$, where $A(\xi,\zeta)$ is
   the matrix given in \eqref{eqn:companion-matrix}, and $\spr$ denotes the
   spectral radius of a complex matrix. In particular,
   \begin{equation}
      \label{eqn:integrated-herman}
      \h(\af)\ge \int_{\SS}\log^+ \spr A(0,\zeta)\,d\zeta.
   \end{equation}
\end{proposition}

\begin{proof}
   Fix an irrational $\zeta\in\SS$. Put $B(x)=A(x,\zeta)$, and let
   $T(x)=\zeta x$. For a complex matrix $C=[c_{ij}]$ define
   $\|C\|=\max_{i,j}\{|c_{ij}|\}$. Theorem \ref{thm:met} shows that for almost
   every $\xi\in\SS$ we have that
   \begin{displaymath}
      \chi_\infty(\zeta)=\lim_{n\to\infty} \frac{1}{n}
      \log\|B(T^{n-1}\xi) B(T^{n-2}\xi)\dots B(\xi)\|.
   \end{displaymath}
   Expand
   \begin{displaymath}
      B(T^{n-1}x)B(T^{n-2}x)\dots B(x)= \bigl[ b_{ij}^{(n)}(x)\bigr],
   \end{displaymath}
   where the $b_{ij}^{(n)}(x)$ are polynomials in $x$ with complex
   coefficients. Now each $\log |b_{ij}^{(n)}(\xi)|$ is subharmonic for
   $\xi\in\CC$, and hence $\max_{i,j} \{\log|b_{ij}^{(n)}(\xi)|\}$ is also
   subharmonic for $\xi\in\CC$. Thus
   \begin{displaymath}
      \max_{i,j} \{\log |b_{ij}^{(n)}(0)|\} \le\int_{\SS}\max \{\log
      |b_{ij}^{(n)}(\xi)|\} \,d\xi.
   \end{displaymath}
   Furthermore,
   \begin{displaymath}
      \frac1{n} \max_{i,j} \{\log |b_{ij}^{(n)}(0)|\}=\frac{1}{n}
      \log \|B(0)^n\|\to \log \spr B(0)=\log \spr A(0,\zeta)
   \end{displaymath}
   as $n\to\infty$. The entries in $A(\xi,\zeta)$ are uniformly bounded
   above, and hence $\frac{1}{n}\log|b_{ij}^{(n)}(\xi)|$ is uniformly
   bounded for all $n\ge1$ and $\xi\in\SS$. Thus
   \begin{align*}
     \chi_\infty(\zeta)&=\int_{\SS}\lim_{n\to\infty}\frac{1}{n}\log\|
                         B(T^{n-1}\xi)B(T^{n-2}\xi)\dots B(\xi)\|\,d\xi \\
     &\ge\limsup_{n\to\infty} \int_{\SS}\frac{1}{n}\log\|
                         B(T^{n-1}\xi)B(T^{n-2}\xi)\dots B(\xi)\|\,d\xi\\
     &\ge \limsup_{n\to\infty}\frac{1}{n}\log\|B(0)^n\|=\log \spr
       A(0,\zeta).
       \qedhere
   \end{align*}
\end{proof}

Observe that since only nonnegative powers of $x$ are allowed, this
result is reminiscent of the face entropy inequality in Corollary
\ref{cor:face-entropy}. It is stronger since it gives a lower bound for
every irrational $\zeta$, but the integrated form
\eqref{eqn:integrated-herman} is weaker since it uses only the top
Lyapunov exponent to give a lower bound for the entropy of the face
corresponding to $x=0$.

\begin{example}
   \label{exam:big-example}
   We finish by returning to Example \ref{exam:y^2-xy-1}. Let
   $f(x,y,z)=y^2-xy-1\in\ZG$. Using the change of variables $x\mapsto y$,
   $y\mapsto x$ and $z\mapsto z^{-1}$, $f$ becomes monic and linear in $y$,
   hence by the previous theorem we can compute that $\h(\af)=0$. On the
   other hand,
   treating $f$ as monic and quadratic in $y$, we see that the Lyapunov
   exponents must all be $\le 0$. But the determinant has absolute value 1,
   and so in fact the Lyapunov exponents must vanish almost everywhere. In
   other words,
   \begin{equation}
      \label{eqn:big-example}
      \lim_{n\to\infty}\frac1n \log\Biggl\| \prod_{j=n-1}^0
      \begin{bmatrix}
         0 & 1\\1 & \xi\zeta^j
      \end{bmatrix} \Biggr\| = 0 \text{\quad for almost every $(\xi,\zeta)\in\SS^2$}.
   \end{equation}
   By taking transposes to reverse the order of the product, we obtain
   \eqref{eqn:random-product} from the Introduction.
   Although this appears to be a simple result, we have not been able to
   obtain this as a consequence of any known results in random matrix theory.
\end{example}

\end{document}